\documentclass{amsart}
\usepackage{fullpage}
\usepackage{include}
\numberwithin{thm}{subsection}

\usepackage[french,english]{babel}
\newcommand{\newabstract}[1]{%
  \par\bigskip
  \csname otherlanguage*\endcsname{#1}%
  \csname captions#1\endcsname
  \item[\hskip\labelsep\scshape\abstractname.]
}

\begin{document}

\title{Anticyclotomic $p$-adic $L$-functions and Ichino's formula}
\author{Dan J. Collins}
\address{Department of Mathematics \\ 
	University of Milan \\ 
	Milan, Italy} 
\email{daniel.collins@unimi.it} 


\begin{abstract}
We give a new construction of a $p$-adic $L$-function $\CL(f,\Xi)$, for $f$ a
holomorphic newform and $\Xi$ an anticyclotomic family of Hecke characters of
$\Q(\rt{-d})$. The construction uses Ichino's triple product formula to express
the central values of $L(f,\xi,s)$ in terms of Petersson inner products, and
then uses results of Hida to interpolate them. The resulting construction is
well-suited for studying what happens when $f$ is replaced by a modular form
congruent to it modulo $p$, and has future applications in the case where $f$
is residually reducible.

\newabstract{french}
Nous donnons une nouvelle construction d'une fonction $L$ $p$-adique
$\CL(f,\Xi)$, pour $f$ une forme primitive holomorphe et $\Xi$ une famille
anticyclotomique de caract\`eres de Hecke de $\Q(\rt{-d})$. La construction
utilise la formule du produit triple d'Ichino pour exprimer les valeurs
centrales de $L(f,\xi,s)$ en terme de produits scalaires de Petersson, et
certains r\'esultats de Hida pour les interpoler $p$-adiquement. La construction qui en d\'ecoule
permet d'\'etudier ce qui se passe quand $f$ est remplac\'ee par
une forme modulaire qui lui est congruente modulo $p$, et a des applications
futures dans le cas o\`u $f$ est r\'esiduellement r\'eductible.
\end{abstract}

\selectlanguage{english}

\maketitle 
%\tableofcontents

\section{Introduction}

Given a classical holomorphic newform $f$ and a Hecke character $\xi$ of an
imaginary quadratic field, we can consider the classical Rankin-Selberg
$L$-function $L(f\x\xi,s)$ and in particular study its special values. If we
vary the character $\xi$ in a $p$-adic family $\Xi$, we obtain a collection of
special values which (one hopes) can be assembled into a $p$-adic analytic
function $\CL_p(f,\Xi)$. This paper gives a new construction of a certain type
of \E{anticyclotomic $p$-adic $L$-function} of this form. Such $p$-adic
$L$-functions have been studied fruitfully in recent years: their special
values have been related to algebraic cycles on certain varieties
(\cite{BertoliniDarmonPrasanna2013}, \cite{Brooks2015},
\cite{CastellaHsieh2018}, \cite{LiuZhangZhang2018}), and a Main Conjecture of
Iwasawa theory was proven for them (\cite{Wan2015}). In the case where $f$ is a
weight-2 newform corresponding to an elliptic curve $E$, these results have
been used by Skinner and others to obtain progress towards the Birch and
Swinnerton-Dyer conjecture for curves that have algebraic rank 1
(\cite{Skinner2014}, \cite{JetchevSkinnerWan2017}). 
\par
Even though $p$-adic $L$-functions are characterized uniquely by an
interpolation property for their special values, merely knowing their
\E{existence} is not enough to be able to prove many theorems involving them.
Instead one usually needs to work with the underlying formulas and methods used
to construct them. Prior papers (e.g. \cite{BertoliniDarmonPrasanna2013}, 
\cite{Brakocevic2011}, \cite{Brooks2015}, \cite{Hsieh2014},
\cite{LiuZhangZhang2018}) have usually realized $\CL_p(f,\Xi)$ by using
formulas of Waldspurger that realize special values of $L(f\x\xi,s)$ as toric
integrals. The purpose of this paper is to give a construction instead using a
triple-product formula due to Ichino \cite{Ichino2008}. We remark that the
paper of Hsieh \cite{Hsieh2017} which appeared after this work performs similar
computations in a more general setting. 
\par
Having a different construction for $\CL_p(f,\Xi)$ will lead to new results
about this $p$-adic $L$-function. In particular, in future work we will study
the case where $f$ is congruent to an Eisenstein series $E$ modulo $p$, and
show that we get a congruence between $\CL_p(f,\Xi)$ and a product of simpler
$p$-adic $L$-functions (arising just from Hecke characters). Having this
congruence information will allow us to obtain Diophantine consequences for
certain families of elliptic curves. It will also allow us to work with Iwasawa
theory in the ``residually reducible'' case, providing complementary results to
the work of \cite{Wan2015} in the residually irreducible case. 

\subsubsection*{Notations and conventions.} Before stating our results more
precisely, we start with some basic conventions for this paper. Throughout, we
will be fixing a prime number $p > 2$ and then studying $p$-adic families (of
modular forms, $L$-values, etc.). Thus we will fix once and for all embeddings
$\io_\oo : \L\Q \into \C$ and $\io_p : \L\Q \into \L\Q_p$. 
\par
We will also be working with an imaginary quadratic field $K = \Q(\rt{-d})$,
which again we will fix subject to some hypotheses specified later. We will
treat $K$ as being a subfield of $\L\Q$, and thus having distinguished
embeddings into $\C$ and into $\L\Q_p$ via $\io_\oo$ and $\io_p$, respectively.
The embedding $\io_p : K \into \L\Q_p$ is a place of $K$, and thus corresponds
to a prime ideal $\Fp$ lying over $p$. We will always be working in the
situation where $p$ splits in $K$; thus we can always take $\Fp$ to denote the
prime ideal lying over $p$ that corresponds to $\io_p$, and $\L\Fp$ the other
one. We let $\oo$ denote the unique infinite place of $K$, coming from
composing the embedding $\io_\oo : K \into \C$ with the complex absolute
value. Our conventions on Hecke characters for $K$ (in their many guises) are
spelled out in Section \ref{sec:HeckeCharConventions}. 
\par
We will work heavily with the theory of classical modular forms and newforms,
as developed in e.g. \cite{MiyakeModForms} or \cite{ShimuraAutForms}. If $\ch$
is a Dirichlet character modulo $N$ we let $M_k(N,\ch) = M_k(\Ga_0(N),\ch)$ be
the $\C$-vector space of modular forms that transform under $\Ga_0(N)$ with
weight $k$ and character $\ch$. We let $S_k(N,\ch)$ denote the subspace of cusp
forms; on this space we have the Petersson inner product, which we always take
to be normalized by the volume of the corresponding modular curve: 
\[ \g{f,g} = \f{1}{\vol(\DH\q\Ga_0(N))} 
		\S_{\DH\q\Ga_0(N)} f(z) \L{g(z)} \Im(z)^k \f{dx\ dy}{y^2}. \]

\subsubsection*{The anticyclotomic $p$-adic $L$-function.} Given this setup, we will
fix a newform $f \in S_k(N)$ with trivial central character. We will also want
to fix an ``anticyclotomic family'' of Hecke characters $\xi_m$ for our
imaginary quadratic field $K$. The precise meaning of this is defined in
Section \ref{sec:FamiliesOfHeckeCharacters}, but it amounts to starting with a
fixed character $\xi_a$ (normalized to have infinity-type $(a+1,-a+k-1)$) and
then constructing closely-related characters $\xi_m$ of infinity-type
$(m+1,-m+k-1)$ for each integer $m\ee a\md{p-1}$. With these normalizations,
the Rankin-Selberg $L$-function $L(f\x\xi_m\1,s)$ (which we also denote as
$L(f,\xi_m\1,s)$, thinking of it as ``twisting the $L$-function of $f$ by
$\xi_m\1$'') has central point $0$. 
\par
The anticyclotomic $p$-adic $L$-function $\CL_p(f,\Xi\1)$ that we want to
construct is then essentially a $p$-adic analytic function $L_p : \Z_p \to
\C_p$ such that for $m > k$ satisfying $m\ee a\md{p-1}$, the value of $L_p$ at
$s = m$ is the central $L$-value $L(f,\xi_m\1,0)$. Of course, this doesn't make
sense as-is because $L(f,\xi_m\1,0)$ is a (likely transcendental) complex
number. So we need to define an ``algebraic part'' $L_\alg(f,\xi_m\1,0) \in
\L\Q \se \C$, which we move to $\L\Q \se \L\Q_p$ via our embeddings $i_\oo$ and
$i_p$, and then modify to a ``$p$-adic part'' $L_p(f,\xi_m\1,0)$. The basic 
algebraicity result is due to Shimura, and the exact choices we make to define
these values are specified in Section \ref{sec:NormalizationsOfLFunctions}
\par
Also, rather than literally take $\CL_p(f,\Xi\1)$ an analytic function on
$\Z_p$, we instead use an algebraic analogue of this: we construct
$\CL_p(f,\Xi\1)$ as an element of a power series ring $\CI \iso \Z_p^\uR\bb{X}$.
Certain continuous functions $P_m : \CI \to \Z_p^\uR$ serve as ``evaluation at
$m$''; this is defined in Section \ref{sec:BDPLfunction}. With all
of these definitions made, we can precisely specify what $\CL_p(f,\Xi\1)$
should be: 

\begin{defn}
	The anticyclotomic $p$-adic $L$-function $\CL_p(f,\Xi\1)$ is the unique
	element of $\CI$ such that, for integers $m > k$ satisfying $m\ee
	a\md{p-1}$, we have $P_m(\CL_p(f,\Xi\1)) = L_p(f,\xi_m\1,0)$. 
\end{defn}

\subsubsection*{Ichino's formula, classically.} The definition of $L$-values does
not lend itself to $p$-adic interpolation. Instead, $p$-adic $L$-functions are
constructed by relating $L$-values to something else that is more readily
interpolated. This often comes from the theory of automorphic representations,
where there are many formulas relating $L$-values to integrals of automorphic
forms. Our approach is to use \E{Ichino's triple product formula}
\cite{Ichino2008}, which relates a certain global integral (for three
automorphic representations $\pi_1,\pi_2,\pi_3$ on $\GL_2$) to a product of
local integrals. The constant relating them is the central value of a
triple-product $L$-function $L(\pi_1\x\pi_2\x\pi_3,s)$. 
\par
We will apply this by taking $\pi_1$ to correspond to the modular form $f$ in
question, and letting $\pi_2$ and $\pi_3$ be representations induced from Hecke
characters $\ps$ and $\ph$ on $K$. Translating Ichino's formula into classical
language gives an equation of the form 
\[ |\g{f(z)g_\ph(z),g_\ps(c z)}|^2 
		= C \dt L(f,\ph\ps\1,0) L(f,\ps\1\ph\1 N^{m-k-1},0), \]
where $g_\ph$ and $g_\ps$ are the classical CM newforms associated to the Hecke
characters $\ph,\ps$. Our goal will be to set up $\ph,\ps$ to vary
simultaneously in $p$-adic families, so that one of our two $L$-values is a
constant and the other realizes $L(f,\xi_m\1,0)$. The $p$-adic theory we
develop will allow us to interpolate the Petersson inner product on the left,
and the formula will tell us that this realizes $\CL(f,\Xi\1)$ times some
constants.
\par
The fact that we use two characters $\ph$ and $\ps$ gives us quite a bit of
flexibility in our calculations. This flexibility allows us to appeal to
theorems in the literature of the form ``for all but finitely many Hecke
characters, a certain $L$-value is a unit mod $p$'' because we can avoid the
finitely many bad characters. For instance, we can use results like Theorem C
of \cite{Hsieh2014} to arrange the auxiliary $L$-value $L_\alg(f,\ph\ps\1,0)$
is a $p$-adic unit and thus doesn't interfere with integrality or congruence
statements for the rest of our formula.
\par
Obtaining the constant $C$ in Ichino's formula explicitly is carried out in
Chapter \ref{sec:ExplicitIchino}. The main difficulty is that the formula
involves local integrals at each bad prime $q$. Specifically, the integrals
are over a product of three matrix coefficients, one for the newvector in each
of the local representations of $\GL_2(\Q_p)$ coming from $f$, $g_\ph$, and
$g_\ps$. Evaluating these integrals seems to be a hard problem in general,
though several cases have been worked out in the literature (for instance in
\cite{Woodbury2012}, \cite{NelsonPitaleSaha2014}, \cite{Hu2017}). 
\par
By choosing our hypotheses on $f$, $g_\ph$, and $g_\ps$ carefully, we place
ourselves in a situation where most of the local integrals we need are already
calculated in the literature. However, we cannot avoid having to compute one
new case: when one local representation is spherical and the other two are
ramified principal series of conductor 1. We carry out this computation
following ideas from \cite{NelsonPitaleSaha2014}; the result (Proposition
\ref{prop:LocalIntegralsTwoHalfRamified}) may be of interest to others who want
to apply Ichino's formula. 
\par
We have performed numerical computations as a check of the correctness of all
of the local Ichino integral computations we use, as well as the overall form
of the explicit formula; this is carried out in our paper \cite{Collins2018}.

\subsubsection*{The $\La$-adic theory.} With an explicit version of Ichino's formula
established, we next need to establish that we can $p$-adically interpolate
Petersson inner products of the form $\g{f(z)g_\ph(z),g_\ps(c z)}$. In Chapter
\ref{chap:HidaTheory} we recall the basics of Hida's theory of $\La$-adic
modular forms in the form we will need to use them. Chapter
\ref{sec:FamiliesOfCharsCMForms} gives an explicit construction of $\La$-adic
families of Hecke characters and then of the associated $\La$-adic CM forms,
allowing us to construct forms $f\Bg_\Ph$ and $\Bg_\Ps$ that interpolate the
modular forms $fg_\ph$ and $g_\ps$ that appear in our version of Ichino's
formula when $\ph,\ps$ vary in a suitable family. 
\par
Chapter \ref{sec:PeterssonInterpolate} constructs an element
$\g{f\Bg_\Ph,\Bg_\Ps}$ that interpolates the Petersson inner products
$\g{fg_\ph,g_\ps}$. The construction is due to \cite{Hida1988b}, and is based
on the fact that if $h$ is a newform and $1_h$ is the associated projector in
the Hecke algebra, then $1_h g$ is $\g{g,h}/\g{h,h}$ times $g$. Using this 
idea, we need to make it completely explicit how to start with the complex
number $\g{fg_\ph,g_\ps}$, associate to it an algebraic part
$\g{fg_\ph,g_\ps}_\alg$, and finally a $p$-adic part $\g{fg_\ph,g_\ps}_p$. This
is carried out throughout Chapter \ref{sec:PeterssonInterpolate} and is
summarized in Section \ref{sec:SummarizingPeterssonInterpolation}. 
\par
The most involved part of the calculation is relating $\g{fg_\ph,g_\ps}_\alg$
(which is directly defined in terms of $\g{fg_\ph,g_\ps}$) to
$\g{fg_\ph,g_\ps}_p$, which arises as a specialization of
$\g{f\Bg_\Ph,\Bg_\Ps}$. The difficulty here is that $\g{fg_\ph,g_\ps}_p$ is 
actually defined using $\g{fg_\ph^\sh,g_\ps^\na}$, where $g_\ph^\sh$ and
$g_\ps^\na$ are modifications of $g_\ph$ and $g_\ps$ related to the process of
$p$-stabilization. Relating $\g{fg_\ph^\sh,g_\ps^\na}$ to $\g{fg_\ph,g_\ps}$ 
is carried out in Section \ref{sec:EulerFactorsStabilization} and involves some
delicate manipulations of Petersson inner products. The result is that the two
values differ by a \E{removed Euler factor} that is expected to appear in the  
construction of $p$-adic $L$-functions. 

\subsubsection*{Our results.} We can now state our main theorem. Our hypotheses are
that we begin with: 
\begin{itemize}
	\item A holomorphic newform $f$ of some weight $k$, level $N = N_0
		p^{r_0}$, and trivial central character.
	\item An imaginary quadratic field $K = \Q(\rt{-d})$ with odd fundamental
		discriminant $d$.
	\item A Hecke character $\Uxi_a$ of weight $(a-1,-a+k+1)$ for some
		integer $a$ satisfying $a-1\ee k\md{w_K}$, with trivial central character
		and with conductor $(c)$ for an integer $c$ coprime to $dN$. (Here $w_K =
		|\CO_K^\x|$ is $6$ if $d = 3$ and $2$ in all other cases).
	\item An odd prime $p$ coprime to $2dcN$. 
\end{itemize}
We also have the following auxiliary data we can freely choose: 
\begin{itemize}
	\item A prime $\l \nd 2pdcN$ inert in $K$, and a power $\l^{c_\l}$ of it.
	\item A character $\nu$ of $(\CO_K/\l^{c_\l}\CO_K)^\x$ that's trivial on
		$(\Z/\l^{c_\l}\Z)^\x$ and $\CO_K^\x$.
\end{itemize}
Given this data we can construct: 
\begin{itemize}
	\item An anticyclotomic family of Hecke characters $\Xi$ that goes through
		$\Uxi_a$ and specializes to the characters $\xi_m$ mentioned earlier.
		(Lemma \ref{lem:AnticyclotomicHeckeFamily})
	\item Two families $\Ph,\Ps$ of Hecke characters giving rise to families
		of CM newforms $\Bg_\Ph,\Bg_\Ps$. These families are such that $\Ph\Ps
		N^{-(m-k-1)} = \Xi$ and that $\Ph\1\Ps$ is a constant family for a 
		Hecke character $\et$ with its local behavior at $\l$ corresponding to
		$\nu^2$. (Lemma \ref{lem:ConstructPhiPsi})
\end{itemize}

Our main theorem is then the following construction of $\CL(f,\Xi\1)$ as an 
element of $\CI^\uR$, where $\CI$ is a certain extension of $\La$ (discussed in
detail in Sections \ref{sec:BDPLfunction},
\ref{sec:CIadicModularForms}, and \ref{sec:FamiliesOfHeckeCharacters}). 

\begin{thm} \label{thm:MainTheorem}
	Under the above hypotheses and notation, the element $\CL(f,\Xi\1)$ is equal
	to a product 
	\[ \f{1}{(*)_{f,\l} \dt L_p^*(f,\et\1,0)} \dt \CC 
			\dt \g{f\Bg_\Ph,\Bg_\Ps} \g{f^\rh\Bg_{\Ph^\rh},\Bg_{\Ps^\rh}}. \]
	Here $\g{f\Bg_\Ph,\Bg_\Ps}$ and $\g{f^\rh\Bg_{\Ph^\rh},\Bg_{\Ps^\rh}}$ are
	the elements of Chapter \ref{sec:PeterssonInterpolate} discussed above, 
	$\CC \in \La^\x$ is a unit satisfying 
	\[ P_m(\CC) = \et(\L\Fp)^r \f{2^2 \l^{2c_\l(k+5)}}{N_0^{m+3}}, \]
	the term $L_p^*(f,\et\1,0)$ is essentially the algebraic part of the
	$L$-value $L(f,\et\1,0)$ (see Section \ref{sec:NormalizationsOfLFunctions}),
	and
	\[ (*)_{f,\l} = \b( \s_{i=0}^{c_\l} \pf{\al}{\l^{(k-1)/2}}^{2i-c_\l} 
			- \f{1}{\l} \s_{i=i}^{c_\l-1} \pf{\al}{\l^{(k-1)/2}}^{2i-c_\l} \e). \]
	with $\al_{f,\l}$ one of the roots of the Hecke polynomial for $f$ at $\l$. 
\end{thm}

Strictly speaking, this is only an equality in $\CI^\uR[1/p]$ (and could
potentially be undefined if $L_p^*(f,\et\1,0)$ and $(*)_{f,\l}$ are zero).
However, the point is that both $L_p^*(f,\et\1,0)$ and $(*)_{f,\l}$ are terms
involving our auxiliary choice of data $\l$, $c_\l$, and $\nu$, and by choosing
this data carefully we can arrange for them to be $p$-adic units (or otherwise
suitably controlled) in the situations we want to study. 

\subsubsection*{Acknowledgements.} I would like to thank my advisor Christopher
Skinner for suggesting this project to me and for offering insights and
encouragement, and to thank Peter Humphries and Vinayak Vatsal for helpful
conversations. Much of this research was completed while the author was
supported by an NSF Graduate Research Fellowship (DGE-1148900).
 
\section{An explicit version of Ichino's formula}
\label{sec:ExplicitIchino}

In this section we obtain an explicit version of Ichino's triple-product
formula \cite{Ichino2008} for classical holomorphic newforms. Ichino's formula
is stated abstractly in terms of automorphic representations; we will use the
case where the quaternion algebra is $\GL_2$ and the \etale cubic algebra is
$\Q\x\Q\x\Q$ over $\Q$. In this case the formula can be written: 

\begin{thm}[Ichino's Formula] \label{thm:IchinosFormulaAutomorphic}
	Let $\pi_1,\pi_2,\pi_3$ be irreducible unitary cuspidal automorphic
	representations of $\GL_2(\Q)$ with the product of their central characters
	trivial, If we set $\Pi = \pi_1\ox\pi_2\ox\pi_3$, we have an equality of
	$\GL_2^\x(\A_\Q)$-linear functionals $\Pi\ox\TPi \to \C$ of the form 
	\[ \f{I(\ph,\Tph)}{\g{\ph,\Tph}} = \f{(6/\pi^2)}{8} 
			\f{\ze^*(2)^2 L^*(\pi_1\x\pi_2\x\pi_3,1/2)}
				{L^*(\ad\pi_1,1)L^*(\ad\pi_2,1)L^*(\ad\pi_3,1)} 
			\p_v \f{I_v^*(\ph_v,\Tph_v)}{\g{\ph_v,\Tph_v}_v}. \]
\end{thm}

The notation used in this theorem is as follows. The left hand side involves
global integrals on the quotient set $[\GL_2^\x(\A)] = \A_\Q^\x \GL_2^\x(\Q)
\q \GL_2^\x(\A_\Q)$. In particular, $I$ is the global integration functional
given on simple tensors $\ph = \ph_1\ox\ph_2\ox\ph_3$ and $\Tph =
\Tph_1\ox\Tph_2\ox\Tph_3$ by 
\[ I(\ph,\Tph) 
	= \b( \S_{[\GL_2^\x(A)]} \ph_1(g)\ph_2(g)\ph_3(g) dx_T \e) 
			\b( \S_{[\GL_2^\x(A)]} \Tph_1(g)\Tph_2(g)\Tph_3(g) dx_T \e). \]
The global pairing $\g{\ph,\Tph}$ is given on simple tensors by 
\[ \g{\ph,\Tph} = \p_{i=1}^3 \b( \S_{[\GL_2^\x(A)]}\ph_i(g) \Tph_i(g)dx_T\e).\]
\par
All of the $L$-functions (and the $\ze$-value) are written as $L^*$ and $\ze^*$
to denote that these are taken to include their $\Ga$-factors at infinity. The
triple-product $L$-function is the one studied by Garrett \cite{Garrett1987}
and Piatetski-Shapiro and Rallis \cite{PSR1987}, and $L(\ad\pi,s)$ is the
trace-zero adjoint $L$-function associated to $\pi$ (originally constructed by
Gelbart and Jacquet in \cite{GelbartJacquet1978}, and closely related to
symmetric square $L$-functions).
\par
The right-hand side involves a product of local functionals $I_v^*$, each of
which is a $\GL_2^\x(\Q_v)$-invariant functional on $\Pi_v\ox\TPi_v$. To set
this up we fix an invariant bilinear local pairing $\g{\ ,\ }_{i,v}$ on
$\pi_{v,i}\ox\Tpi_{v,i}$ for each place $v$ and each $i=1,2,3$, and use this to
define a pairing $\g{\ ,\ }_v$ on $\Pi_v\ox\TPi_v$ determined on simple tensors
$\ph_v = \ph_{1,v}\ox\ph_{2,v}\ox\ph_{3,v}$ and $\Tph_v =
\Tph_{1,v}\ox\Tph_{2,v}\ox\Tph_{3,v}$ by 
\[ \g{\ph_v,\Tph_v}_v = \g{\ph_{1,v},\Tph_{1,v}}_{1,v} 
		\g{\ph_{2,v},\Tph_{2,v}}_{2,v} \g{\ph_{3,v},\Tph_{3,v}}_{3,v}. \]
We then define a functional $I_v$ on $\Pi_v\ox\TPi_v$ by 
\[ I_v(\ph_v,\Tph_v) 
		= \S_{\Q_v^\x\q\GL_2(\Q_v)} \g{\pi(g_v)\ph_v,\Tph_v}_v dx_v \] 
We then normalize this functional with (the reciprocal of) the local factors of
the $L$-functions that show up in the global equation to get $I_v^*$: 
\[ I_v^*(\ph_v,\Tph_v) 
		= \f{L_v(\ad\pi_1,1)L_v(\ad\pi_2,1)L_v(\ad\pi_3,1)}
				{\ze(2)^2 L_v(\pi_1\x\pi_2\x\pi_3,1/2)} I_v(\ph_v,\Tph_v). \]
\par
Finally, the global measure $dx_T$ is taken to be the Tamagawa measure for
$\PGL_2(\A)$, which has volume 2 (see e.g. Theorem 3.2.1 of
\cite{WeilAdelesAlgGrps}). The local Haar measures $dx_p$ on $\PGL_2(\Q_p)$ are
chosen so that $\PGL_2(\Z_p)$ has volume 1, and the local Haar measure $dx_\oo$
on $\PGL_2(\R)$ is chosen as the quotient of the measure 
\[ x_\oo = \M{cc}{\al&\be\\\ga&\de} \qq\qq 
		dx_\oo = \f{d\al\ d\be\ d\ga\ d\de}{|\det(x_\oo)|^2}. \]
on $\GL_2(\R)$ by the usual multiplicative Haar measure $dx_{Lebesgue}/|x|$ on
the center $\Z(\GL_2(\R)) \iso \R^\x$. With these normalizations we can check
$\p dx_v$ has volume $\pi^2/3$, so $dx_T = (6/\pi^2) \p dx_v$, hence the factor
of $6/\pi^2$ in our formula.
\par
With this setup, Ichino showed that $I_v^*(\ph_v,\Tph_v) / \g{\ph_v,\Tph_v}_v =
1$ whenever $v$ is a place such that all of the $\pi_v$'s are unramified,
$\ph_v$ and $\Tph_v$ are spherical vectors, and $\PGL_2(\CO_v)$ has volume 1.
Because of this, the product over all places in Ichino's formula is actually a
finite product. 

\subsection{Ichino's formula in the classical case} 

The situation we want to apply Ichino's formula in is the following: We will
fix integers $m > k > 0$ and take classical newforms $f,g,h$ of weights $k$,
$m-k$, and $m$, respectively. We set notation that $N_f$ and $\ch_f$ denote the
level and character of $f$, and similarly for $g$ and $h$; let $N_{fgh} =
\lcm(N_f,N_g,N_h)$. We ultimately will want to use Ichino's formula to relate
the triple product $L$-value a classical Petersson inner product pairing $h$
with a product of $f$ and $g$.
\par
The naive idea is to work with the Petersson inner product
$\g{f(z)g(z),h(z)}$. However, this may not quite work if the levels of the
newforms don't match up - if the LCM of $N_f$ and $N_g$ is a
proper divisor of $N_h$, for instance, then certainly $f(z)g(z)$ is old at
level $N_h$ and thus $\g{f(z)g(z),h(z)} = 0$. We can fix this issue by
replacing the newforms with oldforms of higher level associated to them. We
will consider a pairing $\g{f(M_f z) g(M_g z),h(M_h z)}$ where these integers
are chosen so that, for each prime $q$, $q$ divides at most one of
$M_f,M_g,M_h$ and the largest power of $q$ to divide any of the products $M_f
N_f$, $M_g N_g$, and $M_h N_h$ actually divides two of them. 
\par
To apply Ichino's formula, we let $\pi_f$ be the unitary automorphic
representation associated to $f$. The classical newforms $f,g,h$ correspond to
specific vectors in the automorphic representation, namely $F$ given by 
\[ F(x) = \6( (y^{k/2} f)|[x_\oo]_k\6)(i) \Tch_f(k_0), \]
Here we decompose $x = \ga x_\oo k_0$ with $\ga \in \GL_2(\Q)$, $x_\oo \in
\GL_2^+(\R)$, and $k_0 \in K_0(N_f)$, and we let $\Tch_f$ be the character of
$K_0(N_f)$ given by applying $\ch_f$ to the lower-right entry. 
\par 
Of course, since this vector $F$ corresponds to the newform $f(z)$, we need to
suitably modify it to get something that will correspond to $f(M_f z)$ instead.
To do this we take the notation that if $M$ is an integer we let  
\[ \de_v(M) = \M{cc}{ M & 0 \\ 0 & 1 } \in \GL_2(\Z_v) 
		\qq\qq \de(M) = (\de_v(M)) \in \GL_2(\A_\Q). \]
Moreover, if $v$ is a finite place and we've fixed a uniformizer $\vp_v$ of
$\Q_v$ we let $\de^0_v(M) = \de_v(\vp^{v(M)})$, and we let $\de^0(M) \in
\GL_2(\A^\fin)$ have coordinates $\de^0_v(M)$. Then, the adelic lift of
$f_{M_f}(z) = f(M_f z)$ is given by 
\[ x \mt \6( (y^{k/2} f_{M_f})|[x_\oo]_k\6)(i) \Tch_f(k_0)
		= M_f^{-k} \6( (y^{k/2} f)|[\de_\oo(M_f) x_\oo]_k\6)(i) \Tch_f(k_0) \]
for a decomposition $x = \ga x_\oo k_0 \in \GL_2(\Q) \GL_2^+(\R) K_0(M_f N_f)$.
A straightforward computation shows that if we take $y = \de^0(M_f\1) \in
\GL_2(\A^\fin)$, then the vector $F_{M_f} = \pi(y)F \in \pi_f$ is a multiple of
the adelic lift of $f_{M_f}$.
\par
Similarly, we can take shifts of adelic lifts of $g$ and $h$, and come up with
an input vector 
\[ \ph = \de^0(M_f) F \ox \de^0(M_g) G \ox \de^0(M_h) \LH
		= F_{M_f} \ox G_{M_g} \ox \LH_{M_h} \]
in $\pi_f\ox\pi_g\ox\Tpi_h$. We let $\Tph$ to be the vector of complex
conjugates of these in the contragredients. The automorphic condition for the
central characters being trivial is equivalent to asking $\ch_f \ch_g = \ch_h$
as an equality of Dirichlet characters; assuming this Ichino's formula states
\[ \f{I(\ph,\Tph)}{\g{\ph,\Tph}} = \f{(6/\pi^2)}{8} 
			\f{\ze_F^*(2)^2 L^*(\pi_f\x\pi_g\x\Tpi_h,1/2)}
				{L^*(\ad\pi_f,1)L^*(\ad\pi_g,1)L^*(\ad\Tpi_h,1)} 
			\p_v \f{I_v^*(\ph_v,\Tph_v)}{\g{\ph_v,\Tph_v}_v}. \]
Next, we want to interpret the global integrals $I(\ph,\Tph)$ and
$\g{\ph,\Tph}$ in terms of Petersson inner products. In general, if $\Ps,\Ps'$
are adelic lifts of modular forms $\ps,\ps'$ then computing $\S \Ps\Ps' dx_T$
on $\A^\x\GL_2(\Q)\q\GL_2(\A)$ may be done by passing to a fundamental domain
of the form $D_\oo K_0(N)$ for $D_\oo$ a fundamental domain of
$\Ga_0(N)\q\PGL_2^+(\R)$. This can then be reinterpreted as an integral over
$\Ga_0(N)\q\DH$; keeping track of all of our normalizations (including that
Petersson inner products are normalized by the volume of $\Ga_0(N)\q\DH$) we
obtain
\[ \S_{\PGL_2(\Q)\q\PGL_2(\A)} \Ps(g) \LPs'(x) dx_T = 2\g{\ps,\ps'}. \]
Since $F_{M_f}$ is $M_f^k$ times the adelic lift of $f_{M_f}$ and likewise 
similarly for $G_{M_g}$ and $H_{M_h}$; the left-hand side of Ichino's
formula becomes 
\[ \f{I(\ph,\Tph)}{\g{\ph,\Tph}} 
	= \f{2^2 M_f^{2k} M_g^{2(m-k)} M_h^{2m}}{2^3 M_f^{2k} M_g^{2(m-k)} M_h^{2m}}
			\f{|\g{f_{M_f}g_{M_g},h_{M_h}}|^2}
			{\g{f_{M_f},f_{M_f}}\g{g_{M_g},g_{M_g}}\g{h_{M_h},h_{M_h}}}. \]
We can further see $\g{f_{M_f},f_{M_f}} = M_f^{-k} \g{f,f}$ by a simple
change-of-variables (similar to Lemma \ref{lem:ScalingPetersson}), this becomes 
\[ \f{|\g{f_{M_f}g_{M_g},h_{M_h}}|^2}
		{2 M_f^{-k} M_g^{k-m} M_h^{-m} \g{f,f}\g{g,g}\g{h,h}}. \]
Also, the value of $\ze^*(2)$ is $\pi^{-2/2} \Ga(2/2) \ze(2) = \pi/6$, so at
this point we've simplified Ichino's formula to 
\[ \f{|\g{f_{M_f}g_{M_g},h_{M_h}}|^2}{\g{f,f}\g{g,g}\g{h,h}} = 
			\f{ M_f^{-k} M_g^{k-m} M_h^{-m} \dt L^*(\pi_f\x\pi_g\x\Tpi_h,1/2)}
				{2^3 \dt 3 \dt L^*(\ad\pi_f,1)L^*(\ad\pi_g,1)L^*(\ad\Tpi_h,1)} 
			\p_v \f{I_v^*(\ph_v,\Tph_v)}{\g{\ph_v,\Tph_v}_v}. \]
There are a few more simplifications to make as well. First of all, a formula
of Shimura and Hida (see \cite{Shimura1976}, Section 5 of \cite{Hida1981}, and
Section 10 of \cite{Hida1986}) tells us that the Petersson inner product
$\g{f,f}$ is equal to $L^*(\ad\pi_f,1)$ up to an explicit factor, so we can can
remove those terms from our formula. Specifically, we can formulate the result
as follows, where we define a modified version of the adjoint $L$-value to
absorb some factors at bad places (where we'll deal with them on a
prime-by-prime basis later). 

\begin{thm} \label{thm:AdjointVsPetersson}
	Let $\ps \in S_\ka(N,\ch)$ be a newform, and let $N_\ch$ be the conductor of
	the Dirichlet character $\ch$ (which we take to be primitive). Then we have
	an equality 
	\[ L^H(\ad\ps,1)
			= \f{\pi^2}{6} \f{(4\pi)^\ka}{(\ka-1)!} \g{\ps,\ps}. \]
	Here, $L^H(\ad\ps,1)$ is defined by starting from a shift of the
	``naive'' twisted symmetric square $L$-function:
	\[ L^\naive_q(\ad\ps,s)\1
			= \b(1 - \f{\Lch(q) \al_q^2}{q^{\ka-1}} q^{-s}\e) 
				\b(1 - \f{\Lch(q) \al_q\be_q}{q^{\ka-1}} q^{-s}\e) 
				\b(1 - \f{\Lch(q) \be_q^2}{q^{\ka-1}} q^{-s}\e). \]
	where $L_q(\ps,s)\1 = (1 - \al_q q^{-s})(1- \be_q q^{-s})$, and then setting
	\[ L^H(\ad\ps,1) 
			= \pw{ L^\naive_q(\ad\ps,1) & q\nd N \\ 
						(1 - q^{-2})\1 (1+q\1)\1 & q\| N, q\nd N_\ch \\ 
						(1 - q^{-2})\1 & q | N, q \nd (N/N_\ch) \\ 
						(1 + q\1)\1 & \tx{otherwise} }. \]
\end{thm}

We note that $L(\ad\ps,1)$ equals $L(\ad\pi_\ps,1)$ without a shift, and also
equals $L(\ad\Tpi_\ps,1)$ by self-duality. Also, newforms have discrete series
representations at infinity, so the archimedean $L$-factor is worked out
directly to be 
\[ L_\oo(\ad\pi_\ps,s) = 2 (2\pi)^{-(s+\ka-1)} \Ga(s+\ka-1) \pi^{-(s+1)/2} 
		\Ga\pf{s+1}{2}; \]
ultimately we conclude 
\[ L^*(\ad\pi_\ps,1) = \f{2^\ka \pi}{3} \g{f,f}
					\p_{q|N} \f{L_q(\ad\ps,1)}{L_q^H(\ad\ps,1)}. \]
In the context of Ichino's formula we see we can write
\[ \f{1}{L^*(\ad\pi_f,1)L^*(\ad\pi_g,1)L^*(\ad\Tpi_h,1)}
	= \f{1}{\g{f,f}\g{g,g}\g{h,h}} 
			\f{3^3}{\pi^3 2^k 2^{m-k} 2^m} \p_q \CE_q \]
where 
\[ \CE_q = \f{L^H_q(\ad f,1)}{L_q(\ad f,1)} 
	\f{L^H_q(\ad g,1)}{L_q(\ad g,1)} 
	\f{L^H_q(\ad h,1)}{L_q(\ad h,1)}. \]
\par
We can also look at the $L$-factor $L^*(\pi_f\x\pi_g\x\Tpi_h,1/2)$. The
archimedean factor can be computed to be 
\[ L_\oo(1/2,\pi_f\ox\pi_g\ox\Tpi_h) 
		= 2^4 (2\pi)^{-2m} (m-2)! (k-1)! (m-k-1)!, \]
and we can also check that our normalizations are such that the non-complete
central $L$-value $L(\pi_f\x\pi_g\x\Tpi_h,1/2)$ equals $L(f\x g\x\Lh,m-1)$ when
written classically. Finally, the local integral $I_\oo^*$ at the archimedean
place is known by results of Ichino-Ikeda (\cite{IchinoIkeda2010} Proposition
7.2) or Woodbury (\cite{Woodbury2012} Proposition 4.6), and is $2\pi$ with our
normalizations. So we conclude: 

\begin{thm}[Ichino's formula, classical version] \label{thm:ClassicalIchino}
	Fix integers $m > k > 0$, and let $f \in S_k(N_f,\ch_f)$, $g \in
	S_{m-k}(N_g,\ch_g)$, and $h \in S_m(N_h,\ch_h)$ be classical newforms such
	that the characters satisfy $\ch_f \ch_g = \ch_h$. Take $N_{fgh} =
	\lcm(N_f,N_g,N_h)$ and choose positive integers $M_f,M_g,M_h$ such that the
	three numbers $M_f N_f$, $M_g N_g$, $M_h N_h$ divide $N_{fgh}$ and moreover
	none of the three is divisible by a larger power of any prime $q$ than both
	of the others. Then we have 
	\[ |\g{f_{M_f}g_{M_g},h_{M_h}}|^2 
		= \f{3^2 (m-2)! (k-1)! (m-k-1)!}
					{\pi^{2m+2} 2^{4m-2} M_f^k M_g^{m-k} M_h^m} 
			L(f\x g\x\Lh,m-1) \p_{q|N_{fgh}} \CE_q I_q^*, \]
	where $I_q^*$ are the Ichino local integrals and $\CE_q$ is the term coming
	from our modified adjoint $L$-value.
\end{thm}

Here we use that $I_q^*$ is known to be $1$ at unramified primes, and that
$\CE_q$ is trivially $1$ at such primes as well. 

\subsection{Known results on the local integrals}  
\label{sec:LocalIntegralResults}

The difficult part of making Ichino's formula completely explicit is evaluating
the local integrals $I_q$ at each ramified prime, which has to be done on a
case-by-case basis. Upon decomposing $\pi_f$ as a product of local
representations $\OxP \pi_{f,v}$, a result of Casselman \cite{Casselman1973}
tells us that the vector $F$ corresponds to a simple tensor $\ox F_v$ where
each $F_v$ is a ``newvector'' in $\pi_{f,v}$. Thus $F_{M_f} = \de^0(M_f) F$ has
local components $\de^0_v(M_f) F$. Similarly the local components of $G_{M_g}$
and $H_{M_h}$ are newvectors shifted by an appropriate matrix $\de^0_v(M)$.
So $I_q^*$ only depends on the isomorphism types of $\pi_{f,v}$, $\pi_{g,v}$,
and $\Tpi_{h,v}$, plus perhaps a choice of which newvector to apply a matrix
$\de_v(\vp_q^m)$ to.
\par
To deal with these integrals abstractly, let $\pi_1,\pi_2,\pi_3$ be local
representations of $G = \GL_2(\Q_q)$, always assumed to have the product of
their central characters trivial. We let $c_i$ denote the conductor of $\pi_i$
and let $x_i$ be a \E{newvector}: a vector in the one-dimensional invariant
subspace for the group 
\[ K_2(\Fa) = \b\{ \M{cc}{a&b\\c&d} 
		: c \in \Fa, d\in 1+\Fa, a\in\Z_q^\x, b \in \Z_q \e\} \]
where $\Fa = (p^{c_i})$ (so if $\pi_i$ is unramified then $x_i$ is a spherical
vector). We note that we look at newvectors invariant under $K_2$ rather than
$K_1$ (as in \cite{Casselman1973}) in accordance with our convention that we
extend $\ch_f$ to a character $\Tch_f$ of $K_0$ by applying $\ch_f$ to the
lower-right entry rather than the upper-left. 
\par
We assume without loss of generality that $\pi_3$ has the largest conductor,
i.e. $c_3 \ge c_1,c_2$. Then we set 
\[ I(\pi_1,\pi_2,\pi_3) =
		\S_{Z\q G} \f{\g{gx_1,x_1}}{\g{x_1,x_1}} \f{\g{gy_2,y_2}}{\g{y_2,y_2}}
			\f{\g{gx_3,x_3}}{\g{x_3,x_3}} dg, \]
where $y_2$ is the translate $\de_v(\vp^{c_3-c_2}) x_2$ of our newvector,
and we normalize by setting
\[ I^*(\pi_1,\pi_2,\pi_3) 
		= \f{L(\ad\pi_1,1) L(\ad\pi_2,1) L(\ad\pi_3,1)}
			{L(\pi_1\x\pi_2\x\pi_3,1/2) \ze_q(2)^2} I(\pi_1,\pi_2,\pi_3). \]
Then every local integral $I_q^*$ from Ichino's formula is of the form
$I(\pi_1,\pi_2,\pi_3)$. 
\par 
The values of the local integrals $I^*(\pi_1,\pi_2,\pi_3)$ are not known in
general. Instead they have been computed in various special cases, as needed
for various applications of Ichino's formula. We will quote some of these
special cases that we need, and then make a computation in one new case, in
order to deal with the choices of newforms $f,g,h$ we will need for this paper. 
We start by stating the following easy lemma, which is useful for simplifying
computations (for instance, letting us assume our unramified principal series
are of the form $\pi(\ch,\ch\1)$). 

\begin{lem} \label{lem:UnramifiedCharacterTwist}
	Suppose $\pi_1,\pi_2,\pi_3$ are as above and $\ch_1,\ch_2,\ch_3$ are
	unramified characters satisfying $\ch_1\ch_2\ch_3 = 1$. Let $\ch_i\pi_i =
	\ch_i\ox\pi_i$ be the associated twists. Then we have
	$I(\ch_1\pi_1,\ch_2\pi_2,\ch_3\pi_3) = I(\pi_1,\pi_2,\pi_3)$ and similarly
	for $I^*$.
\end{lem}

The simplest case is when $\pi_1,\pi_2,\pi_3$ have conductors $c_i \le 1$ (i.e.
all of the original modular forms have squarefree level at $q$). In the case of
trivial central characters (and thus for unramified central characters via the
above lemma), this is worked out explicitly by Woodbury \cite{Woodbury2012},
and is implicit in the computations of Watson \cite{Watson2002}. In particular
this covers the case where two of the representations are unramified and the
third is special, which we will need. 
\par
Another case where $I(\pi_1,\pi_2,\pi_3)$ can be computed in a fairly uniform
way is when $\pi_3$ has a much larger conductor than $\pi_1$ or $\pi_2$; this
is carried out by Hu \cite{Hu2017}. In particular it applies to the case where
two of the representations are unramified and the third has conductor at least
$2$. Including the factor $\CE_q$ that appears in our formula, we have the
following uniform result.

\begin{cor}[\cite{Woodbury2012}, \cite{Hu2017}]
	\label{cor:LocalIntegralsTwoUnramified}
	Suppose that $\pi_1,\pi_2$ are unramified and $\pi_3$ is any ramified
	representation (necessarily having an unramified central character). Then we
	have 
	\[ \CE_q I^*(\pi_1,\pi_2,\pi_3) = q^{-c_3} (1+q\1)^{-2}. \]
\end{cor}

If two of the three representations are ramified, the formulas become more
complicated, and can start to involve factors that more heavily depend on the
parameters of the local representations being studied. These have been computed
in the literature in some cases; in particular, \cite{NelsonPitaleSaha2014}
computes $I^*(\pi_1,\pi_2,\pi_3)$ in the cases where all three representations
have trivial central character, $\pi_1$ is unramified, and $\pi_2 \iso \pi_3$.
We discuss their method in the next section, where we use it to prove one new
identity; it applies whenever $\pi_1$ is unramified, though it's unclear
whether the computations would be tractable for all choices of $\pi_2$ and
$\pi_3$. 
\par
For our purposes we only need such computations in two cases. The first one is
the case where $\pi_2,\pi_3$ are ramified principal series of conductor 1.
These representations have nontrivial central character and this computation
does not seem to have been done in the literature; we carry it out in Section
\ref{sec:NewLocalIntegralCalc}. 

\begin{prop} \label{prop:LocalIntegralsTwoHalfRamified}
	Suppose $\pi_1$ an unramified principal series, and $\pi_2,\pi_3$ both
	principal series of conductor $1$ (so both are of the form
	$\pi(\ch_1,\ch_2)$ with $\ch_1$ having conductor $1$ and $\ch_2$ unramified,
	or vice-versa), such that the product of central characters
	$\om_1\om_2\om_3$ is trivial. Then we have
	\[ I^*(\pi_1,\pi_2,\pi_3) = q\1 \qq\qq \CE_q = (1+q\1)^{-2}. \]
\end{prop}

The second is when $\pi_2,\pi_3$ are both supercuspidal representations, in
particular ones of ``type 1'' in the notation of \cite{NelsonPitaleSaha2014}:
these are invariant under twisting by the nontrivial unramified quadratic
character of $\Q_q^\x$. For simplicity we state the result in the case where
the supercuspidal $\pi$ has the same conductor as $\pi\x\pi$. (A type 1
representation $\pi$ must be dihedral corresponding to a character $\xi$ of the
unramified quadratic extension of $\Q_q$, and the conductors of $\pi$ and
$\pi\x\pi$ are two times the conductors of $\xi$ and $\xi^2$, respectively. So
if $q$ is odd these conductors are automatically equal as long as $\xi$ is not
a quadratic character.)

\begin{prop} \label{prop:LocalIntegralTwoSupercuspidals}
	Suppose $\pi_1 = \pi(\ch,\ch\1)$ is an unramified principal series (for
	$\ch$ an unramified unitary character) and $\pi_2 \iso \pi_3 \iso \pi$ is
	supercuspidal of Type 1 with conductor $n$ (necessarily even) and trivial
	central character. If we assume that $\pi\x\pi$ also has conductor $n$, then
	\[ \CE_q I^*(\pi_1,\pi_2,\pi_3) = q^{-n} (1+q\1)^{-2} \dt (*) \]
	where we set $\al = \ch(q)$ and define
	\[ (*) = \pf{(\al^{n/2+1}-\al^{-n/2-1}) 
			- q\1 (\al^{n/2-1}-\al^{-n/2+1})}{\al-\al\1}^2. \]
\end{prop}

In Section \ref{sec:IchinoForCMForms}, we will use these local integral
computations to give a totally explicit version of Ichino's formula in certain
cases where $f$ is a newform and $g,h$ are CM newforms. Also, we remark that we
have performed numerical computations to provide evidence for the correctness
of all of the factors in the various cases of local integral computations
described above (as well as the global constant in our explicit Ichino's
formula); this is described in detail in \cite{Collins2018}.

\subsection{A new local integral computation} \label{sec:NewLocalIntegralCalc}

The method of \cite{NelsonPitaleSaha2014} is based on the following result,
which is a key lemma from \cite{MichelVenkatesh2010}.

\begin{prop}[Michel-Venkatesh, \cite{MichelVenkatesh2010} Lemma 3.4.2] 
	If $\pi_1,\pi_2,\pi_3$ are tempered smooth representations of $\GL_2(\Q_q)$,
	with $\pi_1 \iso \pi(\ch,\ch\1)$ unramified and satisfying $\ch(q) = q^s$,
	satisfy $\om_1\om_2\om_3 = 1$, then we have 
	\[ I^*(\pi_2,\pi_3;s) = (1+q\1)^2 L(\ad\pi_2,1) L(\ad\pi_3,1) 
			J^*(\pi_2,\pi_3;s) J^*(\Tpi_2,\Tpi_3;-s). \]
\end{prop}

Here $J^*$ comes from a certain \E{local Rankin-Selberg integral} associated to
the two representations, namely 
\[ J(\pi_2,\pi_3;s) = \S_{NZ\q G} f^\o_s(g) W_2^\ps(g) W_3^\Lps(g) dg \]
which is then normalized by
\[ J^*(\pi_2,\pi_3;s) 
		= \f{\ze_q(1+2s)}{L(\pi_2\x\pi_3,1/2+s)} J(\pi_2,\pi_3;s). \]
Here $f^\o_s$ is the normalized spherical vector of $\pi(\ch,\ch\1)$ given by
\[ f^\o_s\b( \M{cc}{a & b \\ 0 & d} k \e) = \b|\f{a}{d}\e|^{s+1/2}, \]
and $W^\ps$ denotes the Whittaker newvector (in the Whittaker model
$W(\pi,\ps)$ of $\pi$) normalized by requiring $W^\ps(1) > 0$ and
$\g{W^\ps,W^\ps} = 1$ under the natural pairing 
\[ \g{W_1,W_2} = \S_{\Q_q^\x} W_1\b( \M{cc}{y & 0 \\ 0 & 1} \e) 
		\L{ W_2\b( \M{cc}{y & 0 \\ 0 & 1} \e) } d^\x y. \]
Our integral is over a quotient of $G = \GL_2(\Q_q)$ by a product of the center
$Z \iso \Q_q^\x$ and the standard unipotent radical $N \iso (\Q_q,+)$.

\subsubsection*{A decomposition for our integral.} To use this result, we
need to evaluate the integrals $J(\pi_2,\pi_3;s)$. The first step is to expand
out our integral over domains we understand how to integrate over. We start by
setting up a bit of notation (following \cite{NelsonPitaleSaha2014}); we set 
\[ w = \M{cc}{0 & 1 \\ -1 & 0} \qu a(y) = \M{cc}{y & 0 \\ 0 & 1} \qu
	z(t) = \M{cc}{t & 0 \\ 0 & t} \qu n(x) = \M{cc}{1 & x \\ 0 & 1} \]
for $y,t\in\Q_q^\x$ and $x \in \Q_q$, and accordingly we set $A = \{ a(y) :
y\in\Q_q^\x \}$, $Z = \{ z(t) : t\in\Q_q^\x\}$ and $N = \{ n(x) : x\in\Q_q \}$. 
With this notation, the usual upper-triangular Borel subgroup is $B = ZNA$. The
normalized Haar measures on $\Q_q$ and $\Q_q^\x$ (giving $\Z_q$ and $\Z_q^\x$
volumes 1, respectively) pass to Haar measures on $Z$, $N$, and $A$. 
\par
We then use the following decomposition of our group $G$, extending the Iwasawa
decomposition. We first decompose
\[ K = \dU_{i=0}^n (B\i K) \ga_i K_2(\vp^n) 
		\qq\qq \ga_i = \M{cc}{1 & 0 \\ \vp^i & 1 } \]
and then conclude $G = \dU_{i=0}^n B\ga_i K$. Note in particular in the extreme
cases of $i = 0$ and $i = n$ we have $B\ga_0 K_2(\vp^n) = Bw K_2(\Fp^n)$ and $B
\ga_n K_2(\vp^n) = B K_2(\vp^n)$. This decomposition is discussed in Section
2.1 of \cite{Schmidt2002} and in Section 2.1 of \cite{Hu2016}. For a function
$g$ invariant by $K_2(\vp^n)$ on the right, it leads to us being able to write
an integral over $G$ as 
\[ \S_G f(g) dg = \s_{0\le i\le n} v_i \S_B f(b\ga_i) db 
		\qq v_i = \pw{ \f{1}{(1+q\1)} & i = 0 \\
						\f{(1-q\1)}{(1+q\1)} q^{-i} & 0 < i < n \\
								\f{1}{(1+q\1)} q^{-n} & i = n }, \]
where $db$ is the usual Haar measure on the Borel subgroup $B = ZNA$, given by
$|a|\1 d^\x\!z\, dn\, d^\x\!a$ on this decomposition. We have a similar
expression for integrals over $Z\q G$ or $ZN\q G$. Thus, if $\pi_2$ and $\pi_3$
have conductor $\vp^n$ (so their Whittaker model is right
$K_2(\vp^n)$-invariant) we can write 
\[ J(\pi_2,\pi_3;s) = \s_{0\le i\le n} v_i \S_{\Q_q^\x} 
			|y|^{s-1/2} W_2^\ps(a(y)\ga_i) W_3^\Lps(a(y)\ga_i) d^\x y, \]
using that $f_s^\o(a(y)\ga_i) = |y|^{s+1/2}$ by definition. So, if we can come
up with an explicit enough expression for these values of the Whittaker
function, we can compute this integral directly via this decomposition. 

\subsubsection*{Whittaker newvectors of principal series.} We now want to
compute the Whittaker newvector $W^\ps$ for a principal series representation
of conductor 1 and with central character satisfying $\ch_\pi(\vp) = 1$, so
$\pi = \pi(\mu\ch,\mu\1)$, where $\mu$ is unramified and $\ch$ has conductor 1
and $\ch(\vp) = 1$. We note that a general version of this computation is also
carried out in Section 4 of \cite{Templier2014}. 
\par
We start by recalling that $\pi(\mu\ch,\mu\1)$ can first be realized in its
\E{induced model}
consisting of all smooth functions $f : G\to \C$ satisfying
\[ f\b( \M{cc}{a&b\\0&d} g \e) = |a/d|^{1/2} (\mu\ch)(a) \mu\1(d) f(g). \]
In the induced model, the computation of the newvector is straightforward (see
section 2.1 of \cite{Schmidt2002}, for instance). It is the following function
$f$ defined in terms of the decomposition $G = B\ga_0 K_2(\vp) \du B\ga_1
K_2(\vp)$: 
\[ f\b( \M{cc}{ a & b \\ 0 & d} \ga_i k \e) =
			\pw{ \ch(a) \mu(ad\1) |ad\1|^{1/2} & i = 1 \\ 
						0 & i = 0 }. \]
Next, we need to transfer this to the Whittaker model $W(\pi,\ps)$. The
isomorphism from the induced model is given by $h \mt \S_{\Q_q} \ps(-x) h(wn(x)
g) dx$. Thus the Whittaker newvector $W^\ps \in W(\pi,\ps)$ is the function
$G\to\C$ determined by 
\[ W^\ps(g) = \S_{\Q_q} \ps(-x) f(w n(x) g) dx \]
for our induced model newvector $f(x)$ written above. 
\par
To evaluate this integral for $g = a(y)\ga_i$, we need to compute $f(w
n(z)\ga_i)$ for all $z$ and all $i = 0,1$. To do this, we start by writing
explicitly that if $z\in\Z_q$ then 
\[ w n(z) \ga_i = \M{cc}{0 & 1 \\ -1 & -z} \M{cc}{1 & 0 \\ \vp^i & 1 } 
		= \M{cc}{ \vp^i & 1 \\ -1-z\vp^i & -z } \in K, \]
and we can then see that this lies in $B\ga_0 K$ unless $i = 0$ and $z \in
-1+\vp\Z_q$, and in that case the resulting matrix lies in $K_2$ so $f(w n(z)
\ga_i) = 1$. If $z\nin \Z_q$ we compute 
\[ w n(z) \ga_i = \M{cc}{-z\1 & 1 \\ 0 & -z} \M{cc}{1 & 0 \\ \vp^i+z\1 & 1 } 
		\in B \dt K, \]
and we find this decomposition lies in $B\ga_0 K_2$ if $i = 0$ and in $B\ga_1
K_2$ if $i = 1$. In fact, in the $i = 1$ case the second matrix is in $K_2$
already, so $f(w n(z) \ga_i) = \ch(-z\1) \mu(z^{-2}) |z|\1$. Combining these
facts we conclude 
\[ f(w n(z) \ga_i) = \pw{ 1 & i = 0, z \in -1+\vp\Z_q \\ 
								\ch(-z\1) \mu(z^{-2}) |z|\1 & i = 1, z \nin \Z_q \\
								0 & \tx{otherwise}}. \]
\par
We can then go back to the integral $\S \ps(-x) f(w n(x/y) \ga_i) dx$ we
needed to evaluate to compute $W^\ps(a(y)\ga_i)$. If $i = 0$ we know that the
integrand is nonzero only when $x/y \in -1+\vp\Z_q$, and the integral becomes
the integral of $\ps(-x)$ over $x \in -y + y\vp\Z_q = y + \vp^{v+1}\Z_q$ for $v
= v(y)$. Taking the substitution $x' = -x-y$ we conclude the integral is 
\[ \ps(y) \S_{\vp^{v+1}\Z_q} \ps(x') dx' 
		= \ps(y) \pw{ q^{-v-1} & v+1 \ge 0 \\ 0 & v+1 < 0 }. \]
Noting that $|y| = q^{-v}$ by definition, we conclude that we have 
\[ W^\ps ( a(y) \ga_0 ) 
	= \pw{ \mu(y)\1 |y|^{1/2} \ps(y) q\1 & v(y) \ge -1 \\ 
				0 & v(y) < -1 }. \]
\par
Similarly, for $i = 1$ our computations tell us that $f(wn(x/y)\ga_1)$ is
nonzero exactly when $x/y \nin \Z_q$, i.e. $v(x) < v = v(y)$. For $x$
satisfying $v(x) = u < v$ we have 
\[ f(w n(x/y)\ga_1) = \ch(-y/x) \mu(\vp)^{2v-2u} q^{u-v} \]
and thus we have that $\S \ps(-x) f(w n(x/y)\ga_1) dx$ expands as
\[ \s_{u=-\oo}^{v-1} \ch(y) \mu(\vp)^{2v-2u} q^{u-v}
					\S_{\vp^i\Z_q^\x} \ps(-x) \ch\1(-x) dx. \]
Now, the integral in the sum is zero except for the case $u = -1$, when it
gives the $\ep$-factor $q^{1/2} \ep(1/2,\ch\1,\ps)$. Thus we find 
\[ W^\ps ( a(y) \ga_1 ) 
		= \pw{ \ch(y) \mu(y\vp^2) |y|^{1/2} q^{-1/2} \ep(1/2,\ch\1,\ps)
						& v(y) \ge 0 \\ 
			0 & v(y) < 0 }. \]
\par
So we have a formula for a newvector $W^\ps$; recall that we want to normalize
it by requiring $\g{W^\ps,W^\ps} = 1$ and $W^\ps(1) > 0$. First we note that
\[ W^\ps(1) = W^\ps(a(1)\ga_1) = \mu(\vp^2) \ep(1/2,\ch\1,\ps) q^{-1/2} \]
so we can multiply by $\mu(\vp)^{-2} \ep(1/2,\ch,\Lps)$ to guarantee that this
is positive. Then since $W^\ps(a(y)\ga_1) = W^\ps(a(y))$ we compute 
\[ \g{W^\ps,W^\ps} = \S |W^\ps(a(y))|^2 d^\x y = \S_{v(y)\ge 0} |y| q\1 d^\x y 
		= (1-q\1)\1 \S_{v(y)\ge 0} q\1 dy = (1-q\1)\1 q\1, \]
so we need to multiply by $(1-q\1)^{1/2} q^{1/2}$ to normalize the absolute
value. We conclude that the normalized Whittaker newvector is given by: 
\[ W^\ps(a(y)\ga_i) = \pw{ 
		\ch(y) \mu(y) |y|^{1/2} (1-q\1)^{1/2} & v(y) \ge 0 , i = 1 \\ 
		\mu\1(y\vp^2) |y|^{1/2} (1-q\1)^{1/2} q^{-1/2} \ps(y) \ep(1/2,\ch,\Lps) 
						& v(y) \ge -1 , i = 0 \\
				0 & \tx{otherwise}}. \]

\subsubsection*{The local integral for two representations of this type.} Now,
we want to compute $J(\pi_2,\pi_3;s)$ for $\pi_2 = \pi(\mu\ch,\mu\1)$ and
$\pi_3 = \pi(\nu\ch\1,\nu\1)$ are two representations of the type just
considered (with $\mu,\nu$ unramified and $\ch$ of conductor 1). For
convenience we let $\xi$ be the unramified representation $\xi = |\dt|^s$
(since ultimately our parameter $s$ corresponds to the spherical representation
$\pi(\xi,\xi\1)$). 
\par
Applying our computation of the Whittaker newvectors in the previous section we
get the following formula: 
\[ W_2^\ps(a(y)\ga_i) W_3^\Lps(a(y)\ga_i) = \pw{ 
				(\mu\nu)(y) |y| (1-q\1) & v(y) \ge 0 , i = 1 \\ 
				(\mu\1\nu\1)(y\vp^2) |y| q\1 (1-q\1) & v(y) \ge -1 , i = 0 \\
				0 & \tx{otherwise}}. \]
\par
Using our expression for $J(\pi_2,\pi_3;s)$ from the decomposition in terms of
double cosets $B\ga_i K_2$ we can write 
\[ J(\pi_2,\pi_3;s) = (1+q\1)\1 \s_{i=0}^1 q^{-i} \S_{\Q_q^\x} 
			\xi(y) |y|^{-1/2} W_2^\ps(a(y)\ga_i) W_3^\Lps(a(y)\ga_i) d^\x y. \]
Then the $i=1$ term is 
\[ q\1 (1-q\1) \S_{v(y) \ge 0} (\xi\mu\nu)(y) |y|^{1/2} d^\x y 
		= q\1 (1-q\1) \s_{i=0}^\oo (\xi\mu\nu)(\vp^i) q^{-i/2}, \]
which is a geometric series summing to $q\1(1- (\xi\mu\nu)(\vp) q^{-1/2})\1$.
Similarly, the $i = 0$ term becomes 
\[ q\1 (1-q\1) (\mu\1\nu\1)(\vp)^2 
		\s_{i=-1}^\oo (\xi\mu\1\nu\1)(\vp^i) q^{-i/2}. \]
which sums to 
\[ q\1 (1-q\1) \f{(\mu\1\nu\1)(\vp)^2 \dt (\xi\mu\1\nu\1)(\vp\1) q^{1/2}}
				{1 - (\xi\mu\1\nu\1)(\vp) q^{-1/2}}. \] 
So, we conclude 
\[ J(\pi_2,\pi_3;s) = \f{q\1(1-q\1)}{(1+q\1)}
		\b( \f{1}{ 1 - (\xi\mu\nu)(\vp) q^{-1/2} } 
	+ \f{ (\xi\mu\nu)(\vp\1) q^{1/2} }{1 - (\xi\mu\1\nu\1)(\vp) q^{-1/2}} \e).\] 
Collecting terms we find we get 
\[ J(\pi_2,\pi_3;s) = (1+q\1)\1 (1-q\1) q\1
		\f{ (\xi\mu\nu)(\vp\1) q^{1/2} \dt (1 - \xi^2(\vp)q\1) }
		{ (1-(\xi\mu\1\nu\1)(\vp) q^{-1/2}) (1-(\xi\mu\nu)(\vp) q^{-1/2}) }. \]
\par
Next, we recall that we get $J^*(\pi_2,\pi_3;s)$ by multiplying this quantity
by $\ze_q(1+2s)/L(\pi_2\x\pi_3,1/2+s)$. But 
\[ \ze_q(1+2s) = (1 - q^{-1-2s})\1 = (1 - \xi^2(\vp) q\1)\1 \]
cancels a term on the top of our expression above, and similarly
\[ L(\pi_2\x\pi_3,1/2+s) = (1 - (\mu\nu)(\vp) q^{-1/2-s})\1
		(1 - (\mu\1\nu\1)(\vp) q^{-1/2-s})\1 \]
cancels the bottom. So we conclude 
\[ J^*(\pi_2,\pi_3;s) = (1+q\1)\1 (1-q\1) q^{-1/2} (\xi\1\mu\nu)(\vp). \]
\par
Finally, we recall that the ultimate local integral we want is given by
\[ I^*(\pi_1,\pi_2,\pi_3) = (1+q\1)^2 L(\ad\pi_2,1) L(\ad\pi_3,1) 
			J^*(\pi_2,\pi_3;s) J^*(\Tpi_2,\Tpi_3;-s). \]
Since $\Tpi_2 = \pi(\mu\1,\mu\ch)$ and $\Tpi_3 = \pi(\nu\1,\nu\ch)$ our 
computation above gives us 
\[ J^*(\Tpi_2,\Tpi_3;-s) = (1+q\1)\1 (1-q\1) q^{-1/2} (\xi\mu\1\nu\1)(\vp). \]
Thus we have
\[ J^*(\pi_2,\pi_3;s) J^*(\Tpi_2,\Tpi_3;-s) = (1+q\1)^{-2} (1-q\1)^2 q\1; \]
since we can easily check $L(\ad\pi_2,1) = L(\ad\pi_3,1) = (1-q\1)\1$ we
conclude: 

\begin{prop}
	Let $\pi_1 = \pi(\xi,\xi\1)$, $\pi_2 = \pi(\mu\ch,\mu\1)$, and
	$\pi_3(\nu\ch\1,\nu\1)$ be three principal series representations, with
	$\xi,\mu,\nu$ unramified characters and $\ch$ a ramified character of
	conductor 1 satisfying $\ch(\vp) = 1$. Then we have 
	\[ I^*(\pi_1,\pi_2,\pi_3) = q\1. \]
\end{prop}

To deduce Proposition \ref{prop:LocalIntegralsTwoHalfRamified} from this, we
can use Lemma \ref{lem:UnramifiedCharacterTwist} to twist each principal series
so that the central character has value $1$ at $\vp$ (note the product of
twists is trivial because of the initial assumption that the product of central
characters is trivial!). Then we can write $\pi_1$ and $\pi_2$ in the desired
form, and note that their central characters are $1$ and $\ch$, respectively;
this forces the central character of $\pi_3$ to be $\ch\1$ and thus $\pi_3$ to
have the desired form as well.

\subsection{Specialization to the case of CM forms}
\label{sec:IchinoForCMForms}

Finally, we want to use the local integral computations in Section
\ref{sec:LocalIntegralResults} to obtain a completely explicit version of
Ichino's formula (Theorem \ref{thm:ClassicalIchino}) for certain choices of
modular forms $f,g,h$. In particular, we will assume that $g,h$ are both \E{CM
forms}: they come from Hecke characters $\ps$ of imaginary quadratic field.
Given such a Hecke character, the associated CM form $g_\ps$ should be defined
by 
\[ g_\ps(z) = \s_{\Fa\se\CO_K} \ps(\Fa) e^{2\pi iN(\Fa)z} 
		= \s_{n=1}^\oo \b( \s_{\Fa:N(\Fa)=n} \ps(\Fa) \e) q^{2\pi inz}, \]
to guarantee $L(\ps,s) = L(g_\ps,s)$. If $\ps$ has infinity-type $(m,0)$ or
$(0,m)$ then this does indeed define a newform (see e.g. Section 4.8 of
\cite{MiyakeModForms}). 

\begin{prop} \label{prop:CMForm}
	Let $\ps$ be an algebraic Hecke character of infinity-type $(m,0)$ (for an
	integer $m\ge 0$) for an imaginary quadratic field $K = \Q(\rt{-d})$. Then
	the function $g_\ps$ defined above is a newform of weight $m+1$, level $d
	\dt N(\Fm_\ps)$, and character $\ch_K \dt \ch_\ps$. 
\end{prop}

We then take the following setup to guarantee we get an instance of Ichino's
formula where we know all of the local integrals. 

\begin{itemize}
	\item $f$ is a newform of some weight $k$, level $N$, and with trivial 
		character. 
	\item $K = \Q(\rt{-d})$ is an imaginary quadratic field of odd fundamental
		discriminant $-d$, such that $d$ is coprime to $N$.
	\item $\ph,\ps$ are Hecke characters of $K$ of weights $(m-k-1,0)$ and
		$(m-1,0)$, respectively, for some integer $m > k$.
	\item The central characters $\ch_\ph,\ch_\ps$ (the finite-type parts of
		$\ph$ and $\ps$, restricted to $\Z$) are trivial. This forces the
		conductors of $\ph$ and $\ps$ to be ideals generated by integers in
		$\Z$.
	\item The conductors of $\ph$ and $\ps$ are coprime to $N$ and $d$.
		Moreover, they are given by $c\l^{c_\l}$ and $\l^{c_\l}$, respectively,
		and we have
	\begin{itemize}
		\item $c$ is coprime to $Nd$. 
		\item $\l \nd 2Ndc$ is a prime inert in $K$, and the local components of
			$\ph$ and $\ps$ at $\l$ are inverse to each other and not quadratic
			characters. 
	\end{itemize}
\end{itemize}

Ichino's formula can then be written 
\[ |\g{f(z)g_\ph(z),g_\ps(c^2 N z)}|^2 
		= \f{3^2 (m-2)! (k-1)! (m-k-1)!}
					{\pi^{2m+2} 2^{4m-2} (c^2 N)^m}
			L(f\x g_\ph\x\Lg_\ps,m-1) \p_{q|dNc\l} \CE_q I_q. \]
We also note that the triple-product $L$-value $L(f\x g_\ph\x\Lg_\ps,m-1)$
factors as a product of $L(f,\ph\ps\1,0)$ and $L(f,\ps\1\ph\1 N^{m-k-1},0)$ due
to a decomposition of the corresponding Weil-Deligne representations: 
\[ (\Ind_{W_K}^{W_\Q} \ph) \ox (\Ind_{W_K}^{W_\Q} \ps) 
		\iso \6(\Ind_{W_K}^{W_\Q} \ps\ph\6) 
			\op \6( \Ind_{W_K}^{W_\Q} \ps\ph^c \6). \]
We then consider the factors $\CE_q I_q$ at the primes dividing $dNc\l$: 
\begin{enumerate}
	\item $q|N$: In this case, $\pi_{f,q}$ is ramified (and we can't say much
		else about it since we aren't putting many assumptions on $f$) and the
		other two local representations are unramified. Thus we're in the 
		situation of Corollary \ref{cor:LocalIntegralsTwoUnramified}, so 
		$\CE_q I_q^* = q^{-n_q} (1+q\1)^{-2}$ where $q^{n_q}$ is the power of $q$
		dividing $N$. 
	\item $q|c$: In this case only $\pi_{\ph,q}$ is ramified (either a ramified
		principal series or supercuspidal) and the other two local
		representations are unramified. So again we're in the situation of
		Corollary \ref{cor:LocalIntegralsTwoUnramified} and $\CE_q I_q^* =
		q^{-2n_q} (1+q\1)^{-2}$ where $q^{n_q}$ is the power of $q$ dividing $c$
		(and thus $q^{2n_q}$ is the power of $q$ dividing the conductor of
		$\pi_{\ph,q}$).
	\item $q|d$: Here $q$ is odd, $\pi_{f,q}$ is unramified, and the local 
		representations $\pi_{\ph,q}$ and $\Tpi_{\ps,q}$ are each principal
		series associated to a pair of an unramified character and a character of
		conductor $q$. By Proposition \ref{prop:LocalIntegralsTwoHalfRamified},
		we have $\CE_q I_q^* = q\1$.
	\item $q = \l$: In this case $\pi_{f,q}$ is an unramified principal series,
		and $\pi_{\ph,q}$ and $\Tpi_{\ps,q}$ are supercuspidal representations of
		``Type 1''. More specifically, since $\ps$ and $\ph$ have inverse local
		components at $\l$ (including their values on $\vp_\l$, which are 
		$\ch_K(\l) = -1$), these local components $\pi_{\ph,q}$ and
		$\Tpi_{\ps,q}$ are isomorphic. By Lemma
		\ref{lem:UnramifiedCharacterTwist} we can twist both to have trivial
		central character (without twisting $\pi_{f,q}$), and thus we're in
		the situation of Proposition \ref{prop:LocalIntegralTwoSupercuspidals}
		and we get $\CE_q I_q^* = q^{-n_q} (1+q\1)^{-2} (*)$ where $n_q = 2c_\l$
		and the term $(*)$ has parameter $\al = \al_\l(f)/\l^{(k-1)/2}$. Note
		that our hypothesis that the local representations are not quadratic
		characters is exactly what we need to guarantee the condition on
		conductors stated in that proposition.
\end{enumerate}

Putting this together we conclude:

\begin{thm}[Explicit Ichino's formula, CM case] \label{thm:ExplicitIchinoCM}
	Let $f$, $g = g_\ph$, and $h = g_\ps$ satisfy the hypotheses listed earlier
	in this section. Then we have 
	\begin{multline*}
		|\g{f(z)g_\ph(z),g_\ps(c^2 N z)}|^2 \\ 
				= \f{3^2 (m-2)! (k-1)! (m-k-1)!}
							{\pi^{2m+2} 2^{4m-2} d \l^{2c_\l} (c^2 N)^{m+1}}
			\dt \p_{q|c N d\l} (1+q\1)^{-2} \dt (*)_{f,\l} 
			\dt L(f,\ph\ps\1,0) L(f,\ps\1\ph\1 N^{m-k-1},0)
	\end{multline*}
	where $(*)_{f,\l}$ is determined in terms of the root $\al = \al_\l(f)$ of
	the Hecke polynomial for $f$ at $\l$ and is given by 
	\[ (*)_{f,\l} = \b( \s_{i=0}^{c_\l} \pf{\al}{\l^{(k-1)/2}}^{2i-c_\l} 
			- \f{1}{\l} \s_{i=i}^{c_\l-1} \pf{\al}{\l^{(k-1)/2}}^{2i-c_\l} \e). \]
\end{thm}
 
\section{Background from Hida theory} \label{chap:HidaTheory}

In this section we collect the results from Hida's theory of $\La$-adic modular
forms that we'll need. Ultimately, we want to establish that if $f$ is a fixed
modular form and $\ph,\ps$ are Hecke characters varying suitably in families
$\Ph,\Ps$, we can construct a $p$-adic analytic function $\g{f\Bg_\Ph,\Bg_\Ps}$
that explicitly interpolates the family of Petersson inner products
$\g{fg_\ph,g_\ps}$. After recalling the basic setup of Hida theory in this
section we will proceed to constructing the $\La$-adic families $\Bg_\Ph$ and
$\Bg_\Ps$ in Chapter \ref{sec:FamiliesOfCharsCMForms}, and then to working with
$\g{f\Bg_\Ph,\Bg_\Ps}$ in Chapter \ref{sec:PeterssonInterpolate} (and in
particular finding the removed Euler factors at the prime $p$).
\par
Throughout this section, we will always assume $p$ is an odd prime. In some
cases this is for simplicity, but in others it's because important parts of the
theory have not been worked out for the case $p=2$. We largely follow Hida's
writings, especially \cite{Hida1988b} and \cite{HidaEisSeries}, but also
\cite{Hida1985}, \cite{Hida1986} and \cite{Hida1986a}, as well as Wiles' paper
\cite{Wiles1988}. However we remark that many of these papers predate the
formalism of $\La$-adic forms that we use, so results need to be translated
over; unfortunately we do not know of any comprehensive references for this
theory written in the more modern language we use. 

\subsection{$p$-adic modular forms}

Consider the spaces of classical modular forms $M_k(\Ga,\ch) = M_k(\Ga,\ch;\C)$
or cusp forms $S_k(\Ga,\ch) = S_k(\Ga,\ch;\C)$ for a weight $k$, a congruence
subgroup $\Ga$, and character $\ch$. For any subalgebra $A \se \C$  we can
define $A$-submodules $M_k(\Ga,\ch;A)$ consisting of the forms with Fourier
coefficients lying in $A$, and similarly $S_k(\Ga,\ch;A)$ for cusp forms; we
view both as subspaces of a formal power series ring $A\bb{q^{1/M}}$. Standard
results on integrality of newform coefficients tell us that for any ring $A$
containing the image of $\ch$ we have bases of $M_k$ and $S_k$ with
coefficients in $A$ and thus 
\[ M_k(\Ga,\ch;A)\ox_A \C \iso M_k(\Ga,\ch;\C) 
	\qq S_k(\Ga,\ch;A)\ox_A \C \iso S_k(\Ga,\ch;\C). \]
Applying this, we can change our scalars to (the valuation ring of) a $p$-adic
field $F$: 

\begin{defn}
	Fix a weight $k$, a congruence subgroup $\Ga$, and a character $\ch$. If
	$F/\Q_p$ is a $p$-adic field containing the image of $\ch$ and $F_0 \se F$
	is a number field with $F$ as its completion (which also contains the image
	of $\ch$), and we let $\CO_F$ and $\CO_{F_0}$ be the integer rings of $F$
	and $F_0$, respectively, then we can define 
	\[ M_k(\Ga,\ch;\CO_F) = M_k(\Ga,\ch;\CO_{F_0})\ox_{\CO_{F_0}} \CO_F \]
	and similarly for $S_k$. 
\end{defn}

One can then check that this space is independent of the choice of field $F_0$.
The ring $\CO_F\bb{q^{1/M}}$ is naturally equipped with a norm $|\s a_n
q^{n/M}| = \sup\{ |a_n|_p \}$, making it into a $p$-adic Banach space. We will
define the space of $p$-adic modular forms (over $F$) as a certain closed
subspace of $\CO_F\bb{q^{1/M}}$, which will thus be a $p$-adic Banach space.
Any individual space $M_k(\Ga,\ch;\CO_F)$ is finite-rank and thus already
closed, but we can define
\[ M(\Ga,\ch;\CO_F) 
			= M_{\le\oo}(\Ga,\ch;\CO_F) = \Op_{j=0}^\oo M_j(\Ga,\ch;\CO_F) \]
and take its closure: 

\begin{defn}
	Fix a congruence subgroup $\Ga$, a character, and a $p$-adic field $F/\Q_p$
	containing the values of $\ch$. We define the spaces $\LM(\Ga,\ch;\CO_F)$ of
	\E{$p$-adic modular forms} and $\LS(\Ga,\ch;\CO_F)$ of \E{$p$-adic cusp
	forms} with coefficients in $\CO_F$ as the closures of the space
	$M(\Ga,\ch;\CO_F)$ or $S(\Ga,\ch;\CO_F)$, respectively, in
	$\CO_F\bb{q^{1/M}}$ with the Banach space topology given above.
\end{defn}

Equivalently, $\LM(\Ga,\ch;\CO_F)$ is the completion of $M(\Ga,\ch;\CO_F)$ with
respect to the given norm on $q$-expansions. We are most interested in the case
of $\Ga = \Ga_1(N)$, and we write $\LM(N;\CO_F)$ to denote
$\LM(\Ga_1(N);\CO_F)$ and similarly for $\LS$. We will also occasionally need
to work with the larger space $\LM(\Ga_1(N,M);\CO_F)$. We also recall that by a
theorem of Katz \cite{Katz1976} (using the theory of geometric modular forms),
the spaces we've defined actually have ``$p^\oo$-level'': 

\begin{thm}
	As subspaces of $\CO_F\bb{q^{1/M}}$, we have 
	\[ \LS(\Ga\i\Ga_1(p^r),\CO_F) = \LS(\Ga,\CO_F) 
			\qq\qq \LM(\Ga\i\Ga_1(p^r),\CO_F) = \LM(\Ga,\CO_F) \]
	for $\Ga = \Ga(N_0), \Ga_1(N_0)$, or $\Ga_1(N_0,M_0)$ with $N_0,M_0$ prime
	to $p$. In particular
	\[ M_k(N_0 p^\oo;\CO_F) 
			= M_k(\Ga_1(N_0 p^\oo);\CO_F) = \U_r M_k(\Ga_1(N_0 p^r);\CO_F) \]
 	is a subspace of $\LM(N_0;\CO_F)$.
\end{thm}

Now that we've defined the spaces $\LM(N_0;\CO_F)$ and $\LS(N_0;\CO_F)$ we want
to put a Hecke action on them. To do this we actually first need to define an
action by a profinite group
\[ \HZ_{N_0} = \ilim_r (\Z/N_0 p^r\Z)^\x \iso (\Z/N_0\Z)^\x \x \Z_p^\x 
		\iso (\Z/N_0 p\Z)^\x \x (1+p)^{\Z_p}. \]
We first define an action on $M_k(\Ga_1(N_0 p^r);\CO_F)$ for any $k$ and any $r
\ge 0$ by 
\[ \g{z}f = z_p^k f|_k[\si_z] \qq\qq \si_z\in\SL_2(\Z),
		\si_z \ee\M{cc}{z\1 & 0 \\ 0 & z} \md{N_0 p^r}, \]
where $z\mt z_p$ under the projection $\HZ_{N_0} \to \Z_p^\x$; so this is a
slightly modified version of the classical action of $(\Z/N_0 p^r\Z)^\x$ by
diamond operators (hence the notation). We can then check that these actions
are all compatible: 

\begin{prop}
	The action of $\HZ_{N_0}$ on the spaces $M_k(\Ga_1(N_0 p^r);\CO_F)$ are
	compatible, i.e. they extend to a unique action on $\s_{k,r}
	M_k(\Ga_1(N_0 p^r);\CO_F)$. Moreover, this extends to a continuous action of
	$\HZ_{N_0}$ on $\LM(N_0;\CO_F)$, and $\LS(N_0;\CO_F)$ is invariant under
	this action.
\end{prop}

Since $\LM(N_0;\CO_F)$ is a $\CO_F$-module, this group action naturally gives it
a $\CO_F[\HZ_{N_0}]$-module structure, which in fact extends to a
$\CO_F\bb{\HZ_{N_0}}$-module structure. By just considering the direct factor
$(1+p)^{\Z_p}$ of $\HZ_{N_0}$, we get that $\LM(N_0;F)$ has a $\La$-module
structure for 
\[ \La = \CO_F\bb{(1+p)^{\Z_p}} \iso \CO_F\bb{\Z_p} \iso \CO_F\bb{X}. \]
This module structure will be fundamental for the definition we will give of
families of $p$-adic modular forms! We note that any integer $d$ coprime to
$N_0 p$ maps into $\La$ via the inclusion $(\Z/N_0 p\Z)^\x \into \HZ_{N_0}$ and
then the projection $\HZ_{N_0} \onto (1+p)^{\Z_p}$. We denote the image of such
an operator as $\g{d}_\La$; note that this is \E{not} the same as $\g{d} \in
\HZ_{N_0}$. In fact, if $\Tom : \HZ_{N_0} \onto (\Z/N_0 p^r\Z)^\x \into
\CO_F^\x$ is the natural character (serving the same purpose as the
Teichm\"uller character for $N_0 = 1$), we have $\g{d} = \Tom(d) \g{d}_\La$. 
\par
One thing that the action of $\HZ_{N_0}$ does is lets us recover the character
and weight of the modular form. In particular, if $\ch$ is a character of
$(\Z/p^r N_0 \Z)^\x$ for some $r\ge 1$, then modular form $f \in M_k(N_0
p^r,\ch)$ satisfies 
\[ \g{z}f = z_p^k (\Tom^{-k} \ch)(z) f \]
for all $z\in\HZ_{N_0}$. Accordingly, we say that a $p$-adic modular form $f
\in \LM(N_0;\CO_F)$ has \E{weight $k$ and character $\ch$} if it satisfies this
identify for all $z$. Clearly this is a necessary condition for the form to
actually lie in $M_k(N_0 p^r,\ch)$, but in general it is not sufficient - a
$p$-adic modular form of weight $k$ and character $\ch$ need not be classical
of weight $k$ and character $\ch$. If $\ch$ is a character of $(\Z/N_0 p\Z)^\x$
we can define the space $\LM(N_0;\CO_F)[\ch]$ as the subspace of
$\LM(N_0;\CO_F)$ of forms that have tame-at-$p$ character $\ch$, i.e. such that
$\g{z} f = \ch(z) f$ for all $z\in (\Z/pN_0\Z)^\x$.
\par
Finally, we define Hecke operators and their associated Hecke algebras for
$\LM(N_0;\CO_F)$ and its subspaces. The Hecke operators themselves can be
defined from the usual ones of classical modular forms $M_k(\Ga_1(N_0
p);\CO_F)$, or equivalently by using the usual formula on $q$-expansions:

\begin{prop}
	Fix an integer $n$. Then we can define a Hecke operator $T(n)$ on
	$\LM(N_0;\CO_F)$ as the unique continuous extension of the usual Hecke
	operator $T(n)$ on the sum of all subspaces $M_k(\Ga_1(N_0 p^r),F)$ for $r >
	0$. This can be equivalently described in terms of its Fourier coefficients
	by 
	\[ a(m,f|T(n)) = \s_{d|(m,n), (d,N_0 p) = 1} d\1 a(mn/d^2,\g{d}_\La f). \]
	where $f|d$ denotes the action of $d \in \Z_p^\x \se \HZ_{N_0}$ discussed in
	the previous section. The subspaces $\LS(N_0;\CO_F)$, $\LM(N_0;\CO_F)$, and
	$\LS(N_0;\CO_F)$ are invariant under this action.
\end{prop}

We then define the Hecke algebra $\DT(\LM(N_0;\CO_F)) \se
\End_{F\~\cont}(\LM(N_0;\CO_F))$ as the $F$-subalgebra generated by the Hecke
operators. Evidently the restriction map from $\LM(N_0;\CO_F)$ to any space
$M_k(\Ga_1(N_0p^r);\CO_F)$ induces a surjection of $\DT(\LM(N_0;\CO_F))$ onto
$\DT(M_k(\Ga_1(N_0p^r);\CO_F))$ taking $T(n)$ to $T(n)$ for all $n$, and we can
in fact check that $\DT(\LM(N_0;\CO_F))$ is the inverse limit of
finite-dimensional algebras $\DT(M_{\le k}(\Ga_1(N_0p^r);\CO_F))$ (varying $k$
and/or $r$). The same statements hold for Hecke algebras of $\LS(N_0;\CO_F)$,
and also for replacing $\CO_F$ with $F$ in each case. 
\par
Next, we recall that for finite-dimensional spaces we have a perfect pairing
\[ M_k(\Ga_1(N_0 p^r);\CO_F) \x \DT(M_k(\Ga_1(N_0 p^r),\CO_F)) \to \CO_F \]
given by $(f,t) \mt a(1,f|t)$. This formula induces a perfect pairing
$\LM(N_0;\CO_F) \x \DT(\LM(N_0;\CO_F)) \to \CO_F$, and thus we have
isomorphisms 
\[ \DT(\LM(N_0;\CO_F)) \iso \Hom_{\CO_F}(\LM(N_0;\CO_F),\CO_F), \]
\[ \LM(N_0;\CO_F) \iso \Hom_{\CO_F}(\DT(\LM(N_0;\CO_F)),\CO_F); \]
and similarly for cusp forms; see Theorem 1.3 of \cite{Hida1988b}. We will use
this duality repeatedly in the next section.
\par
Finally, we note that the automorphisms of $\LM(N_0;\CO_F)$ arising from the
action of $\HZ_{N_0}$ defined in the previous section all lie in
$\DT(\LM(N_0;\CO_F))$; this can be deduced from checking that if $\l\nd N_0 p$
is a prime then our formula for Hecke operators gives that the action of
$\l\in\HZ_{N_0}$ is given by the Hecke operator $\l(T(\l)^2-T(\l^2))$. Thus we
have a natural map $\HZ_{N_0} \to \DT(\LM(N_0;\CO_F))$, which extends to a
homomorphism $\CO_F\bb{\HZ_{N_0}} \to \DT(\LM(N_0;\CO_F))$. In particular this
makes $\DT(\LM(N_0;\CO_F))$ into a $\La$-algebra.

\subsection{$\La$-adic families of modular forms} 

Now that we've set up the basic theory of $p$-adic modular forms, we develop
the theory of $\La$-adic modular forms, which are ``$p$-adic families of
$p$-adic of modular forms.'' Recall that, given a finite extension $F/\Q_p$
we're working over, $\La$ was defined as $\CO_F\bb{\Ga}$ for $\Ga =
(1+p)^{\Z_p} \se \Z_p^\x$ abstractly isomorphic to $\Z_p$. Moreover we know
$\La$ is abstractly isomorphic to the formal power series ring $\CO_F\bb{X}$;
if we pick a topological generator $\ga$ of $\Ga$ (usually $\ga = 1+p$) then
the isomorphism is determined by $\ga \lq 1+X$. 
\par
Using the description of $\La$ as a power series ring, we know that the set of
continuous $\CO_F$-algebra homomorphisms $\Hom(\La,\CO_F)$ is in bijection with
elements of the maximal ideal $\Fm_F \se \CO_F$, by associating $x \in \Fm_F$
to the homomorphism $\La \to \CO_F$ characterized by $X \mt x$. In the
framework of rigid geometry, this means that $\La$ is the coordinate ring of
the open unit disc, and its $F$-valued points are the homomorphisms $\La \to
\CO_F$. Motivated by this, we think of of an element $f\in\La$ as an analytic
function on the open unit disc, which we can evaluate at a point $P \in
\Hom(\La,\CO_F)$ by taking $P(f)$. 
\par
Given this setup, for an integer $k$ we define a distinguished point $P_k \in
\Hom(\La,\CO_F)$ by $P_k(X) = (1+p)^k - 1$, or equivalently $P_k(\ga) =
(1+p)^k$. Then, following the formalism of Wiles \cite{Wiles1988}, we define a
$\La$-adic modular form to be a formal $q$-expansion with coefficients in
$\La$, such that evaluating it at a point $P_k$ gives a classical modular form
of weight $k$.

\begin{defn}
	A \E{$\La$-adic modular form} of level $N_0 p^r$ (for $p\nd N_0$ and $r\ge
	1$) and tame character $\ch$ (a Dirichlet character modulo $N_0 p$) is a
	formal power series $\Bf = \s A_n q^n \in \La\bb{q}$ such that, for all but
	finitely many $k \ge 2$, the following is satisfied: 
	\begin{itemize}
		\item The formal power series $\Bf_k = P_k(\Bf) = \s P_k(A_n) q^n$ is in
			fact a classical modular form lying in the space
			$M_k(N_0 p^r,\ch\om^{-k};\CO_F)$.
	\end{itemize}
	If all but finitely many $\Bf_k$'s are actually cusp forms, we say $\Bf$ is
	a $\La$-adic cusp form. We let $\BM(N_0 p^r,\ch;\La)$ denote the set of all
	$\La$-adic modular forms of level $N_0 p^r$ and character $\ch$, and
	$\BS(N_0 p^r,\ch;\La)$ the set of $\La$-adic cusp forms; these are evidently
	sub-$\La$-modules of $\La\bb{q}$. 
\end{defn}

An alternative way to formalize this concept is through the idea of
\E{measures}. This requires a bit of setup:

\begin{defn}
	Let $X$ be a (compact) topological space, and let $C(X;\CO_F)$ be the
	compact $p$-adic Banach space of all continuous functions $X \to \CO_F$ with
	the sup-norm. If $M$ is a $\CO_F$-Banach space, we define the space of
	\E{$M$-valued measures on $X$} as the space 
	\[ \r{Meas}(X;\CO_F) = \Hom_{\CO_F\~\cont}(C(X;\CO_F),M). \]
\end{defn}

This definition is by formal analogy with real-valued measure theory; a measure
(in the classical sense) on a compact space is equivalently determined by the
continuous $\R$-linear integration functional $C(X,\R) \to \R$. In the
literature this analogy is sometimes emphasized by writing measures (in our
sense) as $f \mt \S f d\mu$, but we'll just use $f \mt \mu(f)$ to denote the
continuous homomorphism we're calling a ``measure''. 
\par
The reason that measures come up naturally in our context is that the ring
$\La = \CO_F\bb{\Ga}$ itself can be viewed as a space of them. We let $\log_\Ga
: (1+p)^{\Z_p} \to \Z_p$ be the isomorphism $(1+p)^x \mt x$; this is not equal
to the usual $p$-adic logarithm but is a scalar multiple of it. Then, the
following result is an easy consequence of Mahler's theorem (which says that
all continuous functions $\Z_p \to \Z_p$ can be written as a series $x \mt \s
a_k \bi{x}{k}$). 

\begin{lem}
	We have $\r{Meas}(\Ga,\CO_F) \iso \La$, via the map sending a power series
	$A = \s a_n X^n \in \CO_F\bb{X} \iso \La$ to the measure $\mu_A$ that takes
	the function $x \mt \log_\Ga(x) \mt \bi{\log_\Ga(x)}{n}$ to the value $a_n$.
	Under this isomorphism, the action of $\ga\in\Ga$ by multiplication on $\La$
	corresponds to the action of $\Ga$ on $\r{Meas}(\Ga,\CO_F)$ described by
	$(\ga \dt \mu)(f) = \mu(x \mt f(\ga x))$.
\end{lem}

Using this isomorphism, we can then see that if $A$ is an element of $\La$,
taking the specialization $P_m(A)$ is the same as evaluating the measure
$\mu_A$ under the continuous function $\Ga \to \CO_F$ given by $x \mt x^m$.
Furthermore, we can also check that the above isomorphism extends to an 
isomorphism 
\[ \r{Meas}(\Ga,\CO_F\bb{q}) \iso \La\bb{q}, \]
again such that if $A \lq \mu_A$ then the specialization $P_k(M) \in
\CO_F\bb{q}$ is equal to $\mu_A(x\mt x^m)$. 
\par
Thus, a $\La$-adic modular form $\Bf$ (which is by definition an element of
$\La\bb{q}$) naturally corresponds to a $\La\bb{q}$-valued measure $\mu_\Bf$ on
$\Ga$. Moreover, we know that the specializations $\mu_\Bf(x\mt x^k)$ actually
lie in $\LM(N_0;\CO_F)$ for all $k\gg 0$. Using the following ``density'' lemma
(for which we omit the elementary proof) we can check that this means the whole
domain maps into $\LM(N_0;\CO_F)$:

\begin{lem}
	Let $M$ be a Banach space over $\CO_F$, and $M' \se M$ a closed subspace
	that's saturated in the sense that if $m\in M$ satisfies $pm \in M'$, then
	we have $m\in M'$. If $\mu \in \r{Meas}(\Ga,M)$ is a measure such that for
	all $k\ge k_0$, we have $\mu(x \mt x^k) \in M'$, then in fact the image of
	$\mu$ lies in $M'$ and thus $\mu \in \r{Meas}(\Ga,M')$. 
\end{lem}

We can then apply this with $M = \CO_F\bb{q}$ and $M' = \LM(N_0;\CO_F)$; we note
that this space of $p$-adic modular forms is saturated because we can realize
it as an intersection of a $F$-vector space $\LM(N_0;F)$ with $M$. So if $\Bf$
is a $\La$-adic form, we conclude that $\mu_\Bf$ is actually a
$\LM(N_0;\CO_F)$-valued measure on $\Ga$ by the lemma. We can further analyze
it by defining a $(\La\x\La)$-module structure on the module 
\[ \r{Meas}(\Ga,\LM(N_0;\CO_F)) \iso 
		\Hom_{\CO_F\~\cont}(C(\Ga;\CO_F),\LM(N_0;\CO_F)) \]
induced by our $\La$-actions on the spaces $C(\Ga;\CO_F)$ and $\LM(N_0;\CO_F)$; 
in particular for $(\ga_1,\ga_2) \in \Ga\x\Ga$ and a measure $\mu$ we define 
\[ ((\ga_1,\ga_2)\dt\mu)(x\mt f(x)) = \g{\ga_2} \dt \mu(x\mt f(\ga_1 x)) 
		\in \LM(N_0;\CO_F). \]
Then, if $\mu = \mu_\Bf$ for a $\La$-adic modular form $\Bf$, we claim that the
action of an element $(\ga,\ga\1)$ is trivial; to check this, note that
evaluating $\Bf$ at $x \mt x^k$ gives us a classical modular form $\Bf_k$ for
$k\gg 0$, and then evaluating $(\ga,\ga\1)\dt \Bf$ gives us 
\[ ((\ga,\ga\1)\dt \Bf)(x\mt x^k) = \g{\ga\1} \Bf(x\mt \ga^k x^k) 
		= \ga^k \g{\ga\1}\Bf_k = \Bf_k \]
using linearity of both $\Bf$ and $\g{\ga\1}$. So $(\ga,\ga\1)\dt \Bf = \Bf$
when evaluated at $x \mt x^k$ for $k\gg 0$, and because such functions span a
dense subspace we can conclude $(\ga,\ga\1) \dt \Bf = \Bf$ as measures and thus
$\La$-adic modular forms. Since this is true for all $\ga\in\Ga$, we conclude
that such an $\Bf$ is invariant under the antidiagonal copy of $\La$ in
$\La\x\La$; we say it's ``$\La$-invariant'' for short. Summing up: 

\begin{prop}
	If $\Bf \in \BM(N_0 p^r,\ch;\La)$ is a $\La$-adic modular form, then the
	associated measure $\mu_\Bf$ is valued in $\LM(N_0;\CO_F)$ and is
	$\La$-invariant. Thus we could equivalently define $\La$-adic modular forms
	of this level and character as being $\La$-invariant $\LM(N_0;\CO_F)$-valued
	measures $\mu$ such that the specializations $\mu(x\mt x^k)$ lie in
	$M_k(N_0 p^r,\ch\om^{-k};\CO_F)$ for all but finitely many $k\ge 2$. 
\end{prop}

The space of $\La$-invariant measures still naturally has an action of $\La$
(coming from the quotient of $\La\x\La$ by the antidiagonal $\La$ where the
action is invariant); this is equivalently described by 
\[ (\ga\dt\mu)(x \mt f(x)) = \g{\ga} \mu(x\mt f(x)) = \mu(x\mt f(\ga x)). \]
This resulting $\La$-action on $\La$-invariant measures corresponds to the
natural $\La$-action on $\La$-adic modular forms coming from scalar
multiplication. 
\par
The point of view of measures makes it clear that if $\Bf$ is a $\La$-adic
form, all of the specializations $\Bf_k = P_k(\Bf) = \mu_\Bf(x \mt x^k)$
satisfy the appropriate transformation property to be $p$-adic modular forms of
weight $k$ and character $\ch\om^{-k}$. However they may not be classical
forms!
\par
Finally, we note that using measures makes it easy to define Hecke algebras for
$\La$-adic forms. In fact, the Hecke algebra $\DT(\LM(N_0;\CO_F))$ naturally
acts on $\BM(N_0 p^r,\ch;\La)$! We can define a pairing 
\[ \DT(\LM(N_0;\CO_F)) \x \BM(N_0 p^r,\ch;\La) \to \La \]
by mapping $(T,\Bf)$ to $T\Bf$ where $\mu_{T\Bf}$ is determined in terms of
$\mu_\Bf$ by $\mu_{T\Bf}(f) = T \dt \mu_\Bf(f)$. This is evidently a
$\La$-invariant pairing and thus induces a map $\DT(\LM(N_0 ;\CO_F)) \to
\End_\La(\BM(N_0 p^r,\ch;\La))$. We define the image of this map to be
$\DT(\BM(N_0 p^r,\ch;\La))$; it's generated by the operators $T(n)$ which we
can check act on the $q$-expansions in $\La\bb{q}$ by
\[ a(m,\Bf|T(n)) = \s_{d|(m,n), (d,N_0 p) = 1} \g{d}_\La d\1 a(mn/d^2,\Bf), \]
where now $\g{d}_\La$ is just treated as a scalar in $\La$. 

\subsection{$\CI$-adic modular forms} \label{sec:CIadicModularForms}

We now want to expand our discussion of $\La$-adic forms by allowing
coefficients to lie in certain extensions $\CI \ce \La$ (which will ultimately
be needed for what we want to construct). To do this, we start by setting up a
slightly more sophisticated notation for dealing with specializations of
$\La$-adic forms. Recall that if we fix a topological generator $\ga$ of $\Ga$,
homomorphisms $\La \to \CO_F$ are in bijection with elements of $\Fm_F$ by
having an element $x \in \Fm_F$ correspond to the unique homomorphism $\La \to
\CO_F$ given by $\ga \mt 1+x$. We set up some notation:

\begin{defn}
	We let $\CX(\La,\CO_F)$ denote the set of all homomorphisms $P : \La \to
	\CO_F$, which is naturally in bijection with $\Fm_F$ as above. Actually, any
	such point $P$ can be thought of in three different ways that we'll pass
	between freely: 
	\begin{itemize}
		\item As a $\CO_F$-linear homomorphism $\La \to \CO_F$, characterized by
			$\ga \mt x$ 
		\item As the kernel of such a homomorphism, which is a height-one prime
			ideal of $\La$. 
		\item As a generator of such a kernel, namely $\ga - x$ (where $x$ is the
			image of $\ga$ under the homomorphism). 
	\end{itemize}
\end{defn}

With this set up, we can define a distinguished subset of $\CX(\La,\CO_F)$.

\begin{defn}
	We define $\CX_\alg(\La,\CO_F)$ as the subset of $\CX(\La,\CO_F)$ consisting
	of all points $P_{k,\ep}$ specified by $P_{k,\ep}(\ga) = \ep(\ga) \dt
	(1+p)^k $ for $k\ge 2$ and $\ep : \Ga \to \CO_F^\x$ a finite-order
	character. Given such a point $P = P_{k,\ep}$, we write $k(P) = k$, $\ep_P =
	\ep$, and $r(P) = r$ where $r$ is the conductor of $\ep$ (i.e. the kernel of
	$\ep$ is $\ga^{p^{r-1}} \Ga$). 
\end{defn}

Our definition of $\La$-adic modular forms only mentioned the specializations
at $P_k = P_{k,\ep_0}$, for the trivial character $\ep_0$. The $\La$-invariance
of the associated measure tells us that specializations at other points
$P_{k,\ep}$ would have the appropriate transformation property to be a $p$-adic
modular form of the weight and character we'd expect, but it wouldn't let us
conclude that these forms are classical. One could give a similar definition
requiring classical behavior at some larger subset of the algebraic points;
this more restrictive definition would give subspace of $\La$-adic forms, and
one can analyze its relation to the original space. However, this point is not
particularly important for out purposes, so we will continue working with our
original definition (only requiring classicality at all but finitely many of
the points $P_k$). Moreover, for \E{ordinary} forms the two definitions are
known to agree.
\par
Next, we consider enlarging our base ring $\La$. One way is to expand from
$\La_F = \CO_F\bb{\Ga}$ to $\La_L = \CO_L\bb{\Ga}$ for $L/F$ a finite
extension; this is not particularly interesting since everything we've done
before works just as well over a different base field from $F$. More
interesting is considering other sorts of extensions of $\La$; the most general
case one could reasonably work with would be to take finite flat extensions
$\CI/\La$. For our purposes, it's sufficient to consider extensions $\CI =
\CO_F\bb{\Ga'}$ where $\Ga'$ is a group containing $\Ga$ with finite index;
throughout this paper we consider only extensions $\CI$ of this type.

\begin{defn}
	We define $\CX_\alg(\CI;\CO_F)$ as the set of points $P \in
	\CX(\CI;\CO_F) = \Hom(\CI,\CO_F)$ such that the restriction $P|_\La$ lies in
	$\CX(\La,\CO_F)$. We often abuse notation and let $P_{k,\ep}$ denote any
	point of $\CX(\CI;\CO_F)$ lying over the point $P_{k,\ep}$ in
	$\CX_\alg(\La,\CO_F)$. 
\end{defn}

For the type of extensions $\CI$ we're considering, if the field $F$ is large
enough (e.g. contains all $e$-th roots of unity and $e$-th roots of $(1+p)$,
where $e$ is the exponent of the finite abelian group $\Ga'/\Ga$), then there
are $[\Ga':\Ga]$ points in $\CX_\alg(\CI;\CO_F)$ lying over each point
$P_{k,\ep}$, which differ from each other by the characters of $\Ga/\Ga'$.
Therefore, such extensions satisfy conditions (3.1a) and (3.1b) of
\cite{Hida1988b}, so the results of that paper apply directly to our context.

\begin{defn}
	A \E{$\CI$-adic modular form} of level $N_0 p^r$ (for $p\nd N_0$ and $r\ge
	1$) and tame character $\ch$ (a Dirichlet character modulo $N_0 p^r$) is a
	formal power series $\Bf = \s A_n q^n \in \CI\bb{q}$ such that, for all but
	finitely many $k \ge 1$, the following is satisfied: 
	\begin{itemize}
		\item For any point $P'_k \in \CX_\alg(\CI,\CO_F)$ lying over $P_k \in
			\CX_\alg(\La,\CO_F)$, the formal power series $P_k(\Bf) = \s P_k(A_n)
			q^n$ is in fact a classical modular form lying in the space
			$M_k(N_0 p^r,\ch\om^{-k};\CO_F)$. 
	\end{itemize}
	If all but finitely many $\Bf_k$'s are actually cusp forms, we say $\Bf$ is
	a $\CI$-adic cusp form. We let $\BM(N_0 p^r,\ch;\CI)$ denote the set of all
	$\CI$-adic modular forms of level $N_0 p^r$ and character $\ch$, and
	$\BS(N_0 p^r,\ch;\CI)$ the set of $\CI$-adic cusp forms.
\end{defn}

By similar arguments as in the previous section, we could equivalently view
$\CI$-adic modular forms as $\CI$-invariant $\LM(N_0;\CO_F)$-valued measures on
$\Ga'$ (such that all but finitely many of the specializations $\mu_\Bf(P_k)$
are classical of the appropriate weight and character). Also as in the last
section, we can define a Hecke algebra $\DT(\BM(N_0 p^r,\ch;\CI))$ arising from
the action of $\DT(\LM(N_0;\CO_F))$ on measures, or equivalently from the
formula for $T(n)$ on formal $q$-expansions written there.

\subsection{Ordinary $\CI$-adic newforms}

Hida's theory of $p$-adic modular forms and their families largely focuses on
\E{ordinary} forms. If $f$ is a classical modular form which is an eigenform of
the $T(p)$ operator with eigenvalue $\la_p$, we say $f$ is \E{ordinary} (or
\E{$p$-ordinary} to emphasize the prime) if $\la_p$ is a $p$-adic unit. We work
with this idea by using the \E{ordinary idempotent} operator.

\begin{defn}
	If $\DT = \DT(M)$ is the Hecke algebra associated to a space of modular
	forms $M$ over $\CO_F$ (for $F$ a $p$-adic field), we define its \E{ordinary
	idempotent} $e$ as the unique idempotent $e \in \DT$ such that $e T(p)$ is a
	unit in $e \DT$ and $(1-e) T(p)$ is topologically nilpotent in $(1-e) T(p)$.
	We define the ordinary Hecke algebra $\DT^\ord$ to be the direct factor
	$e\DT$, and the ordinary subspace $M^\ord$ of $M$ to be the image $e[M]$. 
\end{defn}

For classical spaces $M_k(N_0 p^r,\ch;\CO_F)$, one can construct this $e$ by
noting that the Hecke algebra is finite-dimensional over $F$ and thus
decomposes as a finite product of local rings; thus $e$ can just be the
projection onto those local rings in which $T(p)$ acts as a unit. By taking
inverse limits we can obtain an $e$ for $\LM(N_0;\CO_F)$ and then we can
further get one for $\BM(N_0 p^r,\ch;\CI)$ by using the surjection of Hecke
algebras we have. Alternatively, we can define $e = \lim_{n\to\oo} T(p)^{n!}$
(interpreted in an appropriate way in each of our contexts). However we define
it, we note that the $e$'s are compatible between inclusions of different
spaces of $p$-adic modular forms, and commute with specializations of
$\CI$-adic forms (i.e. $P(e\Bf) = e \dt P(\Bf)$). Also, all of these statements
are the same cusp forms in place of holomorphic modular forms. 
\par
An important type of ordinary $\CI$-adic modular form is one such that the
specializations are classical newforms; we quote Hida's results on such forms
from \cite{Hida1988b} (translated into our setup). Naively, we might want to
say $\Bf$ is a $\CI$-adic newform if all of the specializations
$P_{k,\ep}(\Bf)$ are actual newforms. This is sufficient for the case where all
of the newforms truly do have $p$ dividing their level; however, we will be
interested in $p$-adic families related to newforms of a level $N_0$ prime
to $p$. Here there's evidently a problem: a newform of level $N_0$ is an
eigenform of the $T(p)$ operator for such a prime-to-$p$ level, which differ
from the $T(p)$ operators on $\LM(N_0;\CO_F)$ induced from $M_k(N_0
p^r;\CO_K)$. However, if we have a $p$-ordinary eigenform of level $N_0$, it
turns out there's a canonical way to associate to it a form of level $N_0 p$
that's an eigenform for the Hecke algebra of that level.

\begin{lem} \label{lem:StabilizationOfNewforms}
	Suppose $p\nd N_0$ and $f \in S_k(N_0,\ch)$ is a $p$-ordinary eigenform of
	the prime-to-$p$ Hecke operator $T(p) \in \DT(S_k(N_0,\ch))$, of weight
	$k\ge 2$. Then the polynomial 
	\[ x^2 - a_p(f) x + p^{k-1} \ch(p) \]
	has roots $\al$ and $\be$ with $|\al|_p = 1$ and $|\be|_p < 1$,
	respectively, and the space $U(f) = \C f(z) \op \C f(pz) \se S_k(N_0 p,\ch)$
	contains two forms 
	\[ f^\sh(z) = f(z) - \be f(pz) \qq\qq f^\ft(z) = f(z) - \al f(pz) \]
	which are eigenforms of the $T(p) \in \DT(S_k(N_0 p,\ch))$ with eigenvalues
	$\al$ and $\be$, respectively. 
\end{lem}
In particular, the form $f^\sh$ is $p$-ordinary; we call it the
\E{$p$-stabilization} of the $p$-ordinary eigenform $f$.
\begin{proof}
	By definition of $f$ being an eigenform of the operator $T(p)$ of level 
	$N_0$, we have 
	\[ a_p f(z) = (T(p)f)(z) 
			= p^{k-1} \ch(p) f(pz) + \f{1}{p} \s_{i=0}^{p-1} f\pf{z+i}{p}. \]
	We want to solve for what constants $\ga$ the linear combination $g(z) =
	f(z) - \ga f(pz)$ is an eigenform of the Hecke operator $T(p)$ of level $N_0
	p$ given by $(T(p)g)(z) = \f{1}{p} \s_{i=0}^{p-1} g\pf{z+i}{p}$. Thus we
	want to solve for $\la,\ga$ such that 
	\[ \la \6( f(z) - \al f(pz) \6) 
			= \f{1}{p} \s_{i=0}^{p-1} f\pf{z+i}{p} 
				- \ga \f{1}{p} \s_{i=0}^{p-1} f\b( p\f{z+i}{p} \e). \]
	Note that $f(p(z+b)/p) = f(z+i) = f(z)$, so the latter sum just reduces to
	$p f(z)$. Meanwhile, our original formula resulting from $f$ being an
	eigenform tells us that the former sum is equal to $a_p f(z) - p^{k-1}
	\ch(p) f(pz)$. Thus we conclude $f(z)-\ga f(pz)$ has eigenvalue $\la$ iff
	we have the equality
	\[ \la f(z) - \la \ga f(pz) 
			= a_p f(z) - p^{k-1} \ch(p) f(pz) - \ga f(z). \]
	Since $f(z)$ and $f(pz)$ are linearly independent as functions $\DH \to \C$,
	this holds iff $\la,\al$ satisfy the equations $\la = a_p - \ga$ and $\la\ga
	= p^{k-1} \ch(p)$. The solutions are exactly $\ga$ satisfying $\ga^2 - a_p
	\ga + p^{k-1}\ch(p) = 0$. This quadratic equation in $\ga$ has two
	solutions, and we know one must not be a $p$-adic unit (since the product
	$p^{k-1}\ch(p)$ isn't) but the other one must be (since the sum $a_p$ is);
	let $\be$ denote the non-unit root and $\al$ denote the unit root. Then we
	find that the forms $f^\sh$ and $f^\ft$ we've defined are eigenforms with
	eigenvalues $\al$ and $\be$, respectively, and these are the only possible
	ones (up to scalars) since we've found two distinct eigenvalues for a
	two-dimensional space, as desired.
\end{proof}

We can then quote Hida's characterization of $\CI$-adic newforms, translated
into our setup. 

\begin{thm} \label{thm:LambdaAdicNewform}
	For an ordinary $\CI$-adic eigenform $\Bf \in \BS^\ord(N_0 p,\ch;\CI)$, the
	following are equivalent:
	\begin{itemize}
		\item $P_{k,\ep}(\Bf)$ is a newform of level $N_0 p^{r(P)}$ for any
			element $P_{k,\ep} \in \CX_\alg(\CI;\CO_F)$ such that the $p$-part of
			$\ep_P \ch \om^{-k(P)}$ is nontrivial.
		\item $P_{k,\ep}(\Bf)$ is a newform of level $N_0 p^{r(P)}$ for every
			element $P_{k,\ep} \in \CX_\alg(\CI;\CO_F)$ such that the $p$-part of
			$\ep_P \ch \om^{-k(P)}$ is nontrivial.
	\end{itemize}
	Moreover, if these conditions are satisfied and $P$ is a point such that
	$\ep_P \ch \om^{-k(P)}$ is trivial (which forces $r(P) = 1$), then either
	$P(\Bf)$ is actually a newform and $k(P) = 2$, or $P(\Bf) = f^\sh$ is the
	$p$-stabilization of a $p$-ordinary newform $f$ of level $N_0$. Such a
	$\CI$-adic newform induces a $\CI$-algebra homomorphism 
	\[ \la_\Bf : \DT(\BS^\ord(N_0 p,\ch;\CI)) \to \CI \]
	given by $T \mt a(1,f|T)$. 
\end{thm}

Given this data, we say that $\Bf$ is a \E{$\CI$-adic newform} of level $N_0$
and character $\ch$. (We say that the level is $N_0$, even if we have
that $\Bf$ is an element of $\BS^\ord(N_0 p,\ch;\CI)$, to emphasize that some
of the specializations actually arise from newforms of level $N_0$). 

\subsection{Modules of congruence} \label{sec:ModuleOfCongruence}

Next, we recall the concept of the \E{module of congruence} of a newform (both
classical and $\CI$-adic), translating the results of Chapter 4 of
\cite{Hida1988b}. We start off with the classical case. Suppose that $f \in
S_k^\ord(N_0 p^r,\ch;\CO_F)$ is a $p$-stabilized newform (i.e. a newform of
level divisible by $p$, or the $p$-stabilization of a newform of level not
divisible by $p$). The classical theory of newforms lets us recover: 

\begin{prop} \label{prop:IdempotentClassicalNewform}
	Given a $p$-stabilized newform $f \in S_k^\ord(N_0 p^r,\ch;\CO_F)$ as above,
	let $\la_f : \DT^\ord(N_0 p^r,\ch;\CO_F) \to \CO_F$ be the associated algebra
	homomorphism. Then we have a $F$-algebra decomposition 
	\[ \DT^\ord_k(N_0 p^r,\ch;\CO_F)\ox F = \DT(f)_F \op \DT(f)^\pr_F \]
	where $\DT(f)_F^\pr$ is the kernel of $\la_f\ox F$, the other direct factor
	satisfies $\DT(f)_F\iso F$, and the projection $\DT^\ord(N_0 p^r,\ch;\CO_F)
	\to \DT(f)$ corresponds to $\la_f\ox F$ under this isomorphism. 
\end{prop}

The multiplicative identity in $\DT(f)_F$ is an idempotent in
$\DT^\ord_k(N_0 p^r,\ch;F)$ that we denote $1_f$. We also define: 

\begin{defn}
	Let $f$ be a $p$-stabilized newform as above. Let $\DT(f)_\CO$ and
	$\DT(\Bf)^\pr_\CO$ be the projections of $\DT^\ord_k(N_0 p^r,\ch;\CO_F)$ 
	onto $\DT(f)_F$ and $\DT(f)_F^\pr$, respectively. Then define the \E{module
	of congruences} for $f$ as the quotient $\CO_F$-module 
	\[ C(f) = \f{ \DT(f)_\CO \op \DT(f)^\pr_\CO } 
						{ \DT^\ord_k(N_0 p^r,\ch;\CO_F) }. \]
\end{defn}

We can check that $C(f) \iso \CO_F/H_f\CO_F$ for some element $H_f \in \CO_F$
(unique up to units), which we call the \E{congruence number} of $f$.
\par
Next, we carry out the same process for $\CI$-adic forms. Suppose $\Bf$ is an
ordinary $\CI$-adic newform of level $N_0$ and character $\ch$; for any level
$N_0 p^r$ we can obtain an associated algebra homomorphism $\la_\Bf :
\DT^\ord(N_0 p^r,\ch;\CI) \to \CI$. Let $Q(\CI)$ be the quotient field of
$\CI$, and by abuse of notation also let $\la_\Bf$ denote the extended
homomorphism $\DT^\ord(N_0 p,\ch;\CI)\ox_\CI Q(\CI) \to Q(\CI)$. Hida proves a
direct sum decomposition of this algebra (really of $\DT(\BS^\ord(N_0;\CO_F))
\ox_\La Q(\CI)$, but it descends to the one we want): 

\begin{thm}
	Given a $\CI$-adic newform $\Bf$ and the associated homomorphism $\la_\Bf$
	as above, and for any $r$, we have a $Q(\CI)$-algebra decomposition 
	\[ \DT^\ord(N_0 p^r,\ch;\CI)\ox Q(\CI) 
			= \DT(\Bf)_Q \op \DT(\Bf)^\pr_Q \]
	where $\DT(\Bf)_Q^\pr$ is the kernel of $\la_\Bf$, the other direct factor
	satisfies $\DT(\Bf)_Q\iso Q(\CI)$, and the projection
	$\DT^\ord(N_0 p,\ch;\CI)\ox Q(\CI) \to \DT(\Bf)_Q$ corresponds to $\la_\Bf$
	under this isomorphism. 
\end{thm}

As before, we let $1_\Bf$ denote the idempotent of $\DT(\Bf)_F$, and we also
let $\DT(\Bf)_\CI$ and $\DT(\Bf)_\CI^\pr$ denote the images of
$\DT^\ord(N_0 p,\ch;\CI)$ under the projections to $\DT(\Bf)_Q$ and
$\DT(\Bf)^\pr_Q$, respectively. Furthermore, we can check that these
definitions process are compatible with specialization: 

\begin{prop}
	Suppose that $\Bf$ is a $\CI$-adic newform as above, and that we fix a point
	$P \in \CX(\CI;\CO_F)_\alg$. Then the inclusion 
	\[ \DT^\ord(N_0 p^r,\ch;\CI) \into \DT(\Bf)_\CI \op \DT(\Bf)_\CI^\pr \]
	induces an isomorphism when we localize at the prime ideal generated by $P$
	which, when we take a quotient by $P$, passes to the decomposition
	associated to $\Bf_P$ by Proposition \ref{prop:IdempotentClassicalNewform}.
	Thus $1_\Bf$ projects to $1_{\Bf_P}$ under the surjection from
	$\DT^\ord(N_0 p^r,\ch;\CI)$ to the appropriate Hecke algebra for $\Bf_P$. 
\end{prop}

We can then define congruence modules for $\Bf$, and state Hida's theorem that
they are compatible with the ones for the specializations $\Bf_P$. For
technical reasons, we introduce the notation that 
\[ \T\DT(\Bf)_\CI^\pr = \I_\Fp (\DT(\Bf)_\CI^\pr)_\Fp \]
where $\Fp$ runs over all prime ideals of height 1 in $\CI$, with this
intersection taken inside of $\DT(\Bf)_Q^\pr$. Clearly $\DT(\Bf)_\CI^\pr \se
\T\DT(\Bf)_\CI^\pr$.

\begin{defn}
	Given a $\CI$-adic newform $\Bf$ as above, define the \E{modules of
	congruences} for $\Bf$ as
	\[ \CC_0(\Bf;\CI) = \f{\DT(\Bf)_\CI \op \DT(\Bf)_\CI^\pr}
				{\DT^\ord(N_0 p,\ch;\CI)} \qq\qq
		\CC(\Bf;\CI) = \f{\DT(\Bf)_\CI \op \T\DT(\Bf)_\CI^\pr}
				{\DT^\ord(N_0 p,\ch;\CI)}. \]
\end{defn}

By the second isomorphism theorem, our modules of congruence (which are defined
in terms of a Hecke algebra that's a quotient of the one Hida uses) are
isomorphic to Hida's. Then, translating \cite{Hida1988b} Theorem 4.6 into our
setup gives: 

\begin{thm} \label{thm:ExistenceOfCongruenceNumber}
	Fix a $\CI$-adic newform $\Bf$ as above, and let $R$ be a local ring of
	$\DT(\LS^\ord(N_0 p;\CO_F))$ through which $\la_\Bf$ factors. Suppose that
	$R$ is Gorenstein, i.e. that $R \iso \Hom_\CI(R,\CI)$ as an $R$-module. Then
	we have $\CC_0(\Bf;\CI) = \CC(\Bf;\CI) \iso \CI/H_\Bf\CI$ for a nonzero
	element $H_\Bf\in\CI$, and for any $P \in \CX_{alg}(\CI,\CO_F)$ with $k(P)
	\ge 2$ we have a canonical isomorphism
	\[ \CC_0(\Bf;\CI) \ox_\CI (\CI/P\CI) \iso C(\Bf_P). \]
\end{thm}

So, from now on, if we're given a $\CI$-adic modular form $\Bf$ (for which the
Gorenstein condition above is satisfied), we'll let $H_\Bf$ denote any element
$\CI$ such that $\CC_0(\Bf;\CI) \iso \CI/H_\Bf\CI$ and call it a \E{congruence
number} for $\Bf$. The theorem above says that for any $P$, the specialization
$H_{\Bf,P}$ we get by projecting $H_\Bf$ to $\CI/P\CI$ serves as a congruence
number for $\Bf_P$, i.e. $H_{\Bf,P} = H_{\Bf_P}$. However, there is some
subtlety here: even though Hida's work gives us a way to realize $H_{\Bf_P}$
from a special value of an adjoint $L$-function for any given $P$, it's only
determined up to a unit, and it's not immediately clear how to choose the units
to fit the special values into a $p$-adic family. Instead, we just know that a
family $H_\Bf$ exists and that $H_{\Bf,P}$ is a $p$-adic unit times Hida's
$L$-value formula. To be able to write $H_\Bf$ explicitly in terms of
$L$-values amounts to showing we can construct a $p$-adic $L$-function
interpolating the adjoint $L$-values in question. In general this should be
able to be recovered as a consequence of the modularity lifting apparatus
developed by Wiles. In the special case we'll deal with (where $\Bf$ comes from
a family of CM modular forms) we will show that the main conjecture of Iwasawa
theory for imaginary quadratic fields (proven by Rubin) is enough to write
$H_\Bf$ as an explicit $p$-adic $L$-function associated to a Hecke character.

\section{Families of Hecke characters and CM forms} 
\label{sec:FamiliesOfCharsCMForms}

In this section we discuss how to construct $\CI$-adic families of $p$-adic
Hecke characters we'll be concerned with, as well as the associated $\CI$-adic
CM forms. Here is where we may be forced to actually use an extension $\CI$
rather than $\La$ itself, due to the $p$-part of the class group of our
imaginary quadratic field $K$. 

\subsection{Conventions about Hecke characters} 
\label{sec:HeckeCharConventions}

To start off, we explicitly describe our conventions for Hecke characters (in
all of their guises) associated to the imaginary quadratic field $K =
\Q(\rt{-d})$. If $\Fm$ is a nonzero ideal of $\CO_K$ and $I^{S(\Fm)}$ is the
group of fractional ideals coprime to $\Fm$, a classical Hecke character is a
group homomorphism $\ph : I^{S(\Fm)} \to \C^\x$ satisfying \[ \ph(\al \CO_K) =
\ph_\fin(\al) \al^a \Lal^b \] for all $\al \in \CO_K$ that are coprime to
$\Fm$, where $\ph_\fin : (\CO_K/\Fm)^\x \to \C^\x$ is a character (the
\E{finite part} or \E{finite-type} of $\ph$) and $a,b$ are complex numbers
(with the pair $(a,b)$ called the \E{infinity type}). With this convention, the
norm character determined by $N(\Fp) = |\CO_K/\Fp|$ on prime ideals has trivial
conductor (i.e. $\Fm = \CO_K$), trivial finite part, and infinity-type $(1,1)$.
Also, we sometimes view $\ph$ as being defined on all fractional ideals of
$\CO_K$, implicitly setting $\ph(\Fa) = 0$ if $\Fa$ is not coprime to $\Fm$. We
caution that many places in the literature take the opposite convention of
calling $(-a,-b)$ the infinity-type, so one must be careful when comparing
different papers. In particular, our convention matches up with that of
Bertolini-Darmon-Prasanna \cite{BertoliniDarmonPrasanna2013}, but is opposite
of Hsieh \cite{Hsieh2014} and from most of Hida's papers. 
\par
Next, we know that (primitive) classical Hecke characters are in bijection with
adelic Hecke characters, i.e. continuous homomorphisms $\DI_K/K^\x \to \C^\x$.
A straightforward way to describe this bijection is by letting $\id(\al) = \p
\Fp^{v_\Fp(\al_\Fp)}$ denote the ideal associated to an idele $\al = (\al_v)$;
then a classical Hecke character $\ph : I^{S(\Fm)} \to \C$ corresponds to a
continuous character $\ph_\C : \DI_K/K^\x \to \C^\x$ such that $\ph_\C(\al) =
\ph(id(\al))$ for every $\al \in \DI_K^{S(\Fm),\oo}$ (the set of ideles that
are trivial at infinite places and places in $S$). Under this correspondence we
find that if $\ph$ had infinity-type $(a,b)$ then the local factor
$\ph_{\C,\oo}$ at the infinite place is given by $\ph_{\C,\oo}(z) = z^{-a}
\Lz^{-b}$. Thus our convention for the infinity-type for an adelic Hecke
character is that it corresponds to the \E{negatives} of the exponents of $z$
and $\Lz$ in the local component at the infinite place. In particular the
adelic absolute value $\|\dt\|_\A$ is a character of infinity type $(-1,-1)$,
and it corresponds to the inverse of the norm character. 
\par
A classical or adelic Hecke character of $K$ is called \E{algebraic} if its
infinity-type $(a,b)$ consists of integers. Algebraic adelic Hecke characters
are in bijection with algebraic $p$-adic Hecke characters; a $p$-adic Hecke
character is a continuous homomorphism $\ps : \DI_K/K^\x \to \L\Q_p^\x$, and
$\ps$ is algebraic of weights $(a,b)$ if its local factors $\ps_\Fp$ and
$\ps_{\L\Fp}$ on $K_\Fp^\x \iso \Q_p^\x$ and $K_{\L\Fp}^\x \iso \Q_p^\x$ are
given by $\ps_{\Fp}(x) = x^{-a}$ and $\ps_{\L\Fp}(x) = x^{-b}$ on some
neighborhoods of the identity in these multiplicative groups of local fields.
Then, an algebraic adelic Hecke character $\ph_\C$ of infinity-type $(a,b)$
corresponds to a $p$-adic Hecke character $\ph_{\Q_p}$ of weight $(a,b)$ by the
formula
\[ \ph_{\Q_p}(\al) 
	= (\io_p\o\io_\oo\1)\6( \ph_\C(\al) \al_\oo^a \Lal_\oo^b \6) 
			\al_\Fp^{-a} \al_{\L\Fp}^{-b} \]
for any idele $\al = (\al_v)$. It is straightforward that this defines a
continuous character $\DI_K \to \L\Q_p^\x$, and it's trivial on $K^\x$ because
if $\al\in K^\x$ is treated as a principal idele then $\io_\oo\1(\al_\oo) =
\io_p\1(\al_\Fp)$ and $\io_\oo\1(\Lal_\oo) = \io_p\1(\al_{\L\Fp})$ via how we
set up our embeddings.
\par
So, whenever we have an algebraic Hecke character $\ph$ we can consider any of
the three types of realizations of it discussed above. We will pass between them
fairly freely, only being as explicit as we need to be clear and to make
precise computations. In particular we'll often abuse notation and allow $\ph$
to denote whichever of the three associated Hecke characters that's most
convenient for our purposes at any given time. Also, we'll use the terminology
``weights'' and ``infinity-type'' interchangeably for the pair of parameters
$(a,b)$. 

\subsection{A $p$-adic Hecke character}

To build our $\CI$-adic families of $p$-adic Hecke characters, we start by
constructing a single $p$-adic Hecke character $\al : \DI_K/K^\x \to \L\Q_p^\x$
with certain properties. To start, we want this character to be unramified
outside of $p$; this means $\al$ should be trivial on a large quotient of the
ideles, in particular the one fitting into the short exact sequence 
\[ 1 \to (\CO_\Fp^\x \x \CO_{\L\Fp}^\x)/\CO_K^\x 
		\to \DI_K/C \to \Cl_K \to 1 \]
coming from the usual inclusion of $\CO_\Fp^\x \x \CO_{\L\Fp}^\x$ into $\DI_K$,
where $C$ is some closed subgroup of $\DI_K$ containing $K^\x$. Now, we have 
\[ \f{\CO_\Fp^\x \x \CO_{\L\Fp}^\x}{\CO_K^\x} 
	\iso \f{\Z_p^\x \x \Z_p^\x}{\CO_K^\x} 
	\iso \f{(\Z/p\Z)^\x\x (\Z/p\Z)^\x}{\de[\CO_K^\x]} 
			\x (1+p)^{\Z_p} \x (1+p)^{\Z_p}. \]
So by passing to a larger quotient of $\DI_K$, we get a short exact sequence 
\[ 1 \to (1+p)^{\Z_p} \to \DI_K/C' \to \Cl_K \to 1 \] 
where the factor of $(1+p)^{\Z_p}$ includes into $\CO_\Fp^\x \se \DI_K$ in the
usual way. Thus $\DI_K/C'$ is an abelian group that's an extension of the
finite abelian group $\Cl_K$ by something isomorphic to $\Z_p$; this means
$\DI_K/C'$ is isomorphic to a direct product of a copy of $\Z_p$ (perhaps
properly containing our original one) and some finite quotient of $\Cl_K$.
Taking a further quotient of $\DI_K/C'$ to kill this finite factor, we can
finally get to a short exact sequence
\[ 1 \to (1+p)^{\Z_p} \to \DI_K/C'' \to (\Cl_K)_0 \to 1 \] 
where $(\CL_K)_0$ is a cyclic group $Z_{p^e}$ and we extend it by $(1+p)^{\Z_p}
\iso \Z_p$ to obtain $\DI_K/C'' \iso \Z_p$. 
\par
The inverse map $(1+p)^{\Z_p} \to (1+p)^{\Z_p} \into \L\Q_p^\x$ defines a
continuous character. Using our short exact sequence we can see we can
abstractly extend this to characters $\DI_K/C'' \to \L\Q_p^\x$ in $p^e$
different ways, which differ by the $p^e$ characters of $(\Cl_K)_0$. There's no
``best'' extension to pick, but for our purposes it doesn't matter; we simply
take $\al$ to be any such extension. Now, $\DI_K/C_{K_\oo} \iso \Z_p$ and it
contains the original $(1+p)^{\Z_p} \iso \Z_p$ with index $p^c$; restricted to
this subgroup $\al$ is continuous and injective. Thus $\al$ itself is
continuous (since it's continuous on an open subgroup) and injective (because
the kernel needs to intersect $(1+p)^{\Z_p}$ trivially, and no nontrivial
subgroup does that). Thus we can conclude it maps isomorphically onto a
subgroup $(1+\pi)^{\Z_p} \se \L\Q_p^\x$ for some element $\pi$ such that
$(1+\pi)^{p^e} = (1+p)$. By picking a finite extension $F/\Q_p$ containing this
element $\pi$ (as well as any other things we might want), we can view $\al$
and as a character $\DI_K/K^\x \to \CO_F^\x$. 
\par
So we have a Hecke character $\al : \DI_K/K^\x \to \L\Q_p^\x$ which is
algebraic with weight $(1,0)$ (recall our convention is that weight $(a,b)$
means the local factors at $\Fp$ and $\L\Fp$ are $x\mt x^{-a}$ and $x\mt
x^{-b}$, which is why we needed to use the inverse map $(1+p)^{\Z_p} \to
\L\Q_p^\x$ rather than the inclusion). By definition, if $x = (x_v) \in \p
\CO_K^\x$ we have $\al(x) = x_\Fp\1 \om_\Fp(x_\Fp)$ where $\om_\Fp :
(\CO_K/\Fp)^\x \to \Z_p^\x$ is the canonical Teichm\"uller character for
$\CO_K$ (which we can then compose with $\io_\oo\o\io_p\1$ to get a canonical
Teichm\"uller character $\om_\Fp : (\CO_K/\Fp)^\x \to \C$). It's clear from our
construction that $\al$ should have conductor $\Fp$ and finite-type
$\om_\Fp\1$, and one can trace back through the definitions to fully verify
this. We note that by construction of the Teichm\"uller character, if $\ze$
is any root of unity in $K$ we have $\om_\Fp(\ze) = \ze$. 
\par
We remark that one could have just as well done this construction with the
factor of $(1+p)^{\Z_p}$ sitting inside $\CO_{\L\Fp}^\x$ instead, and obtained
a character with infinity-type $(0,1)$, conductor $\L\Fp$, and finite-type
$\om_{\L\Fp}\1$. Actually we can recover such a character from $\al$ itself! We
simply define $\al^c : \DI_K/K^\x \to (1+\pi)^{\Z_p}$ defined by $\al^c =
\al\o\Tc$, where $\Tc : \DI_K/K^\x \to \DI_K/K^\x$ is the (evidently
continuous) automorphism on ideles induced by the nontrivial automorphism $c$
of the imaginary quadratic field $K$. On the level of ideals, $\al^c$ is given
by $\al^c(\Fa) = \al(\L\Fa)$. We will need to make use of this character
$\al^c$ in what follows.

\subsection{Families of Hecke characters}
\label{sec:FamiliesOfHeckeCharacters}

What we mean by a $\CI$-adic family of Hecke characters is a continuous
homomorphism $\Ps : \DI_K/K^\x \to \CI^\x$; given this, if we take a point $P
\in \CX(\CI;\CO_F)$ and view it as a homomorphism $\CI \to \CO_F$, the
composition $P\o\Ps$ (the ``specialization at $P$'') becomes a $p$-adic Hecke
character $\DI_K/K^\x \to \CO_F^\x \se \L\Q_p^\x$. In particular, we'll want
$\Ps$ to be such that if we specialize at $P_{m,\ep} \in \CX(\CI;\CO_F)$ we end
up with an algebraic Hecke character with an infinity type determined by $m$
(say, $(m,0)$ or $(m,-m)$). 
\par
We can use the character $\al : \DI_K/K^\x \to (1+\pi)^{\Z_p} \se \CO_F^\x$
constructed in the previous section to define a $\CI$-adic family of
characters, where we take $\CI = \CO_F\bb{\Ga'}$ for $\Ga' = (1+\pi)^{\Z_p}$.
We simply define a family $\CA : \DI_K/K^\x \to \CI^\x$ as the composition of
$\al : \DI_K/K^\x \to \Ga'$ with the tautological embedding $\Ga' \into
\CO_F\bb{\Ga'}$. Continuity is immediate, as both $\al$ and the tautological
embedding are continuous by construction. Similarly we can define $\CA^c$ as
the composition of $\al^c : \DI_K/K^\x$ with the tautological embedding, or
equivalently the precomposition of $\CA$ with the automorphism $\Tc$ of
$\DI_K/K^\x$ induced by the nontrivial automorphism of $K$.
\par
Given that we've picked a $p^e$-th root $1+\pi$ of $1+p$, we can define a
canonical extension of $P_m$ to $\CI$ by requiring $P_m(1+\pi) = (1+\pi)^m$.
For this $P_m$, we have the specialization property $P_m \o \CA = \al^m$
because $P_m$ composed with the tautological embedding gives the map $\Ga' \to
\Ga'$ defined by $x\mt x^m$. In addition to working with this canonical
extension of $P_m$ to $\CX(\CI,\CO_F)$ defined above, for the purposes of the
theory of $\CI$-adic modular forms we'll have to work with \E{all} such
extensions. It's easy to see that an arbitrary such extension is of the form
$P_{m,\ze} : \ga_\pi \mt \ze(1+\pi)^m$ for $\ze$ a $p^c$-th root of unity. Then
if we let $\ep_\ze : \Ga' \to \L\Q_p^\x$ be the character taking
$\ep_\ze(\ga_\pi) = \ze$ we have $P_{m,\ze} \o \CA = \ep'_\ze \al^m$ for
$\ep'_\ze = \ep_\ze \o \al$. Note that the characters $\ep_\ze$ are exactly the
extensions of the trivial character on $\Ga$ to $\Ga'$, and that $\ep'_\ze$ is
an unramified finite-order Hecke character (because $\al$ takes the embedded
copy of $\H\CO_K \se \DI_K$ to $\Ga$, which is killed by $\ep_\ze$). 
\par
One type of family of Hecke characters we'll want to use will be translates of
$\CA$ by a fixed Hecke character. For the purposes of defining $\CI$-adic CM
forms, we'll be interested in fixing a finite-order character $\ps$ and
defining families along the lines of $\Ps = \ps \al\1 \CA$ with specializations
$P_m\o\Ps = \ps \al^{m-1}$ of infinity-type $(m-1,0)$ and finite part $\ps
\om_\Fp^{1-m}$. We can similarly define families $\Ps' = \ps' (\al^c)\1 \CA^c$,
with specializations $P_m\o\Ps'$ of infinity-type $(0,m-1)$.
\par
A third type of family that will be relevant will be translates of the product
family $\CA(\CA^c)\1$; we see that this has specializations $\al^m
(\al^c)^{-m}$ of weight $(m,-m)$. Translates of $\CA(\CA^c)\1$ will be the
``anticyclotomic families'' $\Xi$ of characters that appear in our
interpretation of the BDP $L$-function (as in Theorem \ref{thm:BDPLambda}). The
following lemma tells us that any character $\xi$ in the set $\Si_{cc}(\FN)$
considered by \cite{BertoliniDarmonPrasanna2013} fits into a family $\Xi$ where
an arithmetic progression of specializations also lie in that set. 

\begin{lem} \label{lem:AnticyclotomicHeckeFamily}
	Let $a$ be an integer, and $\Uxi_a : \DI_K/K^\x \to \CO_F^\x$ an algebraic
	Hecke character for $K$ of weight $(a-1,-a+k+1)$. Then there is a family
	$\Xi : \DI_K/K^\x \to \CI^\x$ such that if we set $\xi_m = P_m\o\Xi$ we
	have:
	\begin{itemize}
		\item $\xi_a$ is the character $\Uxi_a$ we started with.
		\item For every $m$, $\xi_m$ is a character of weight $(m-1,-m+k+1)$. 
		\item For $m \ee a\md{p-1}$, $\xi_m$ has the same conductor and same
			finite-type as $\xi_a$; for $m\nee a\md{p-1}$ the finite-type differs
			by powers of the Teichm\"uller characters $\om_\Fp$ and $\om_{\L\Fp}$.
	\end{itemize}
\end{lem}
Of course, the particular infinity-type we've chosen for $\xi_m$ is entirely a
convention that will be convenient for us later. 
\begin{proof}
	We simply set 
	\[ \Xi = \Uxi_a \al^{-a} (\al^c)^a \CA (\CA^c)\1; \]
	then by construction $P_a\o\Xi = \Uxi_a$ and more generally $P_m\o\Xi =
	\Uxi_a \al^{m-a} (\al^c)^{a-m}$. The rest of the statements follow from what
	we already know about the Hecke characters $\al$ and $\al^c$.
\end{proof}

\subsection{Families of CM forms} \label{sec:FamiliesOfCMForms}

Now that we've shown how to define families of Hecke characters for our
imaginary quadratic field $K$, we want to build from them an associated family
of CM forms. So, suppose we start with an ideal $\Fm$ prime to $p$, and a
finite-order character $\ps$ of conductor either $\Fm$ or $\Fm\Fp$. This forces
the $\Fp$-part of the finite-type of $\ps$ to be of the form $\om_\Fp^{a-1}$
for some residue class $a$ modulo $p-1$. Then we define $\Ps = \ps \al\1 \CA$,
so $\ps_m = P_m\o\Ps$ is given by $\ps \al^{m-1}$, and find that $\ps_m$ has
conductor $\Fm$ when $m\ee a\md{p-1}$, and conductor $\Fm\Fp$ otherwise.
\par
Next we need to write down the CM newform associated to the Hecke character
$\ps_m$ is. A Hecke character of type $(m-1,0)$ gives rise to a CM newform of
weight $m$ as discussed in Section \ref{sec:IchinoForCMForms}. There's a slight
wrinkle, though: we've constructed $\ps_m$ as a $p$-adic Hecke character, while
the formula above for defining classical CM forms is done in terms of complex
Hecke characters. So we need to transfer the algebraic $p$-adic character
$\ps_m : \DI_K/K^\x \to \L\Q_p^\x$ to an algebraic complex character
$\ps_{m,\C} : \DI_I/K^\x \to \C^\x$. We recall from Section
\ref{sec:HeckeCharConventions} that if $\ph$ is a $p$-adic character of weight
$(a,b)$ the relation between $\ph$ and $\ph_\C$ is
\[ \ph_\C(x) = (\io_\oo\o\io_p\1)\6(\ph(x) x_\Fp^a x_{\L\Fp}^b \6) 
		x_\oo^{-a} \Lx_\oo^{-b}. \]
It's straightforward to see that if we're given two $p$-adic Hecke characters
$\ph,\ph'$ then $(\ph\ph')_\C = \ph_\C \ph'_\C$; thus $\ps_{m,\C}$ will be
$\ps_\C \al_\C^{m-1}$. 
\par
The way that $\ps_{m,\C}$ appears in the construction of the corresponding CM
newform $g_{\ps_{m,\C}}$ is that values $\ps_{m,\C}(\Fa)$ appear in the
$q$-expansion, where $\Fa$ runs over ideals prime to the conductor of
$\ps_{m,\C}$. Here we interpret $\ps_{m,\C}(\Fa)$ to mean $\ps_{m,\C}$
evaluated on the idele $x = (x_v)$ with $x_\oo = 1$ and $x_\Fq =
\pi_\Fq^{\ord_\Fq(\Fa)}$ at finite places, where $\pi_\Fq$ is a fixed
uniformizer of the local field $(\CO_K)_\Fq$. Then we have 
\[ \ps_{m,\C}(\Fa) = \ps_\C(\Fa) \al_\C(\Fa)^{m-1}. \]
Now, $\ps$ is a finite-order character so, if we suppress the embeddings
$\io_\oo$ and $\io_p$, we can just write $\ps_\C(\Fa) = \ps(\Fa)$. For
$\al_\C$, by definition we have 
\[ \al_\C(\Fa) = (\io_\oo\o\io_p\1)\6(\al(\Fa) \pi_\Fp^{\ord_\Fp(\Fa)} \6). \]
Since $K$ is unramified at $p$, we can choose $p$ as our uniformizer $\pi_\Fp$;
this choice then forces $\al(\Fa)$ to be algebraic so (again suppressing the
embeddings) we can write $\al_\C(\Fa) = \al(\Fa) p^{\ord_\Fp(\Fa)}$. Thus
we conclude 
\[ \ph_{m,\C}(\Fa) = \ps(\Fa) \al(\Fa)^{m-1} p^{(m-1)\ord_\Fp(\Fa)}. \]
\par
Thus, the CM newform associated to $\ps_{m,\C}$ is given by 
\[ g_{\ps_m} = \s_{\Fa:(\Fa,\Fm)=1} \ps(\Fa) \al(\Fa)^{m-1} 
		p^{(m-1)\ord_\Fp(\Fa)} q^{N(\Fa)} 
			\in S_m(d N(\Fm),\ch_K\ch_\ps \om^{1-m}) \]
in the case when $m\ee a\md{p-1}$ and 
\[ g_{\ps_m} = \s_{\Fa:(\Fa,\Fm\Fp)=1} \ps(\Fa) \al(\Fa)^{m-1} q^{N(\Fa)}
		\in S_m(dN(\Fm)p,\ch_K\ch_\ps \om^{1-m}) \]
when $m\nee a\md{p-1}$. We note that in both cases the form is $p$-ordinary
(with respect to $\io_p\o\io_\oo\1$), as the coefficient of $q^p$ is
$\ps_m(\L\Fp)$ (a $p$-adic unit by construction) or $\ps_m(\L\Fp) + \ps_m(\Fp)
p^{m-1} \ee \ps_m(\L\Fp)$ in the two cases.
\par
Looking at the formula for $g_{\ps_m}$ when $m\nee a\md{p-1}$, it's easy to see
how to interpolate this: if we define 
\[ \Bg_\Ps = \s_{\Fa:(\Fa,\Fm\Fp)=1} \Ps(\Fa) q^{N(\Fa)} \in \CI\bb{q} \]
then $P_m(\Bg_\Ps) = g_{m,\ps}$ (under our embeddings) for all such $m$. For a
general extension $P_{m,\ze}$ of $P_m$, we can similarly see that
$P_{m,\ze}(\Bg_\Ps)$ is the $p$-ordinary CM modular form associated to the
Hecke character $\ep'_\ze \ps \al^{m-1}$. 
\par
It remains to check what $P_m(\Bg_\Ps)$ is when $m\ee a\md{p-1}$; Theorem
\ref{thm:LambdaAdicNewform} suggests it should be the $p$-stabilization
$g_{\ps_m}^\sh$ of the CM modular form $g_{\ps_m}$ of prime-to-$p$ level.
Indeed this is true, and can be calculated directly. The roots of the Hecke
polynomial at $p$ for $g_{\ps_m}$ are $\ps_m(\Fp)p^{m-1}$ and $\ps_m(\L\Fp)$,
respectively, and the latter one is clearly a unit while the former isn't. So
we have 
\[ g_{\ps_m}^\sh = g_{\ps,m}(z) - \ps_m(\Fp)p^{m-1} g_{\ps,m}(pz). \]
But the $q$-series for the latter term is 
\[ \ps_m(\Fp)p^{m-1} 
		\s_{\Fa:(\Fa,\Fm)=1} \ps_m(\Fa) p^{(m-1)\ord_\Fp(\Fa)} q^{N(\Fa)p} 
 	= \s_{\Fa\Fp:(\Fa,\Fm)=1} 
			\ps_m(\Fa\Fp) p^{(m-1)\ord_\Fp(\Fa\Fp)} q^{N(\Fa\Fp)}, \]
so subtracting it off from the sum defining $g_{\ps_m}$ kills all of the terms
corresponding to ideals divisible by $\Fp$ and leaves us with 
\[ g_{\ps_m}^\sh 
	= \s_{\Fa:(\Fa,\Fm\Fp)=1} \ps_m(\Fa) p^{(m-1)\ord_\Fp(\Fa)} q^{N(\Fa)}. \]
But this is $P_m(\Bg_\Ps)$ by definition. Similarly, we can see that
$P_{m,\ze}(\Bg_\Ps) = g_{\ep'_\ze\ps_m}^\sh$, which is also the
$p$-stabilization of a newform. 
\par
Thus every specialization $P_{m,\ze}(\Bg_\Ps)$ is a classical modular form, and
in fact either an ordinary newform or a $p$-stabilization of an ordinary
newform. So $\Bg_\Ps$ is a $\CI$-adic modular form, and Theorem
\ref{thm:LambdaAdicNewform} lets us conclude: 

\begin{prop} \label{prop:FamilyOfCMNewforms}
	Let $\Fm$ be an ideal coprime to $p$, and let $\ps$ be a finite-order Hecke
	character of conductor $\Fm$ or $\Fm\Fp$ such that the $\Fp$-part of the
	finite-type of $\ps$ is $\om_\Fp^{a-1}$. Then we can construct a family of
	Hecke characters $\Ps : \DI_K/K^\x \to \CI$ such that $P_m\o\Ps = \ps
	\al^{m-1}$, and an associated $\CI$-adic CM newform 
	\[ \Bg_\Ps = \s_{\Fa:(\Fa,\Fm\Fp)=1} \Ps(\Fa) p^{(m-1)\ord_\Fp(\Fa)} 
			q^{N(\Fa)} \in \BS^\ord(dN(\Fm),\ch_K\ch_\ps\om;\CI) \]
	such that $P_m(\Bg_\Ps) = g_{\ps_m}^\sh$ for $m\ee a\md{p-1}$ and
	$P_m(\Bg_\Ps) = g_{\ps_m}$ for $m \nee a\md{p-1}$. If we view $\ps_m$ as a
	classical (complex-valued) Hecke character, in the $m\ee a\md{p-1}$ case the
	roots of the Hecke polynomial are $\al_g = \ps_{m,\C}(\L\Fp)$ and $\be_g =
	\ps_{m,\C}(\Fp)$.
\end{prop}

\subsection{Complex conjugates of $p$-adic Hecke characters}
\label{sec:ComplexConjugatesHeckeChars}

For the purposes of interpolating Petersson inner products of the form
$\g{fg_\ph,g_\ps}$ (with $\ph$ and $\ps$ varying simultaneously in a $p$-adic family) we will work with the family $\Bg_\Ps$ constructed above and a
``shifted'' variant $f\Bg_\Ph$ we'll define soon. However, Ichino's formula
involves the absolute value $|\g{fg_\ph,g_\ps}|^2$, so in addition to
interpolating $\g{fg_\ph,g_\ps}$ we'll also need to interpolate the complex
conjugate 
\[ \L{\g{fg_\ph,g_\ps}} = \g{f^\rh g^\rh_\ph,g_\ps^\rh} \]
where if $f = \s a_n q^n$ then $f^\rh = \s \L{a}_n q^n$. Recall that if $f \in
S_k(N,\ch)$ then $f^\rh \in S_k(N,\Lch)$. 
\par
For a fixed $p$-adic modular form $f$ with algebraic coefficients, we can
directly define a $p$-adic modular form $f^\rh$ using our chosen embeddings
$\L\Q \into \C,\L\Q_p$. If $f$ is an eigenform for a Hecke operator $T(n)$ with
eigenvalue $a_n$, then $f^\rh$ is an eigenform with eigenvalue $\L{a}_n$.
However, if $f$ is an eigenform for the prime-to-$p$ Hecke operator $T(p)$, we
need to be a little careful with about the roots $\al_f$ and $\be_f$ of the
Hecke polynomial (as in Lemma \ref{lem:StabilizationOfNewforms}): 

\begin{lem}
	Let $f \in S_k(N_0,\ch)$ be a $p$-ordinary eigenform of the prime-to-$p$
	Hecke operator $T(p)$ of weight $k\ge 2$, and let $\al_f,\be_f$ be the roots
	of the Hecke polynomial with $|\al_f|_p = 1$ and $|\be_f|_p < 1$. Then
	$f^\rh$ is also a $p$-ordinary eigenform of $T(p)$, with $\al_{f^\rh} =
	\Lbe_f$ and $\be_{f^\rh} = \Lal_f$.
\end{lem}
In other words, when we replace $f$ by $f^\rh$ the roots $\al_f,\be_f$ are
replaced by $\Lal_f,\Lbe_f$ as we'd expect, but we need to swap them because
$\Lbe_f$ is now the one that's a $p$-adic unit.
\begin{proof}
	Recall that we know $\al_f\be_f = \ch(p) p^{k-1}$, and also that we have
	$\al_f\Lal_f = p^{k-1}$ and $\be_f\Lbe_f = p^{k-1}$ by the Ramanujan
	conjecture for newforms (proven in general by \cite{Deligne1971},
	\cite{Deligne1974}, though for our applications where $f$ is a CM form it
	can be obtained from basic properties of Hecke characters). Since $\al_f$
	and $\ch(p)$ are $p$-adic units we obtain $|\be_f|_p = |\Lal_f|_p =
	|p^{k-1}|_p < 1$, and also $|\Lbe_f|_p = |p^{k-1}|_p/|\be_f|_p = 1$. Thus
	$\Lbe_f$ is a $p$-adic unit while $\Lal_f$ is not, and $f^\rh$ is
	$p$-ordinary because we can also conclude $\L{a}_f = \Lal_f+\Lbe_f$ is a
	$p$-adic unit.
\end{proof}

However, for a $\CI$-adic family $\Bg$, there is no way to ``conjugate'' the
coefficients directly; we need to instead conjugate the specializations
$P_m(\Bg)$ and come up with a $\Bg^\rh$ that best interpolates them. For a CM
form, we have $g_\ps^\rh = g_\Lps$ where $\Lps$ is the complex conjugate of the
complex Hecke character $\ps$. So to interpolate this quantity, we need to
analyze how complex conjugates of our complex Hecke characters pass back to the
$p$-adic side. 
\par 
So suppose $\ph$ is a $p$-adic character of weight $(a,b)$, so as above we 
define the associated complex Hecke character is 
\[ \ph_\C(x) = (\io_\oo\o\io_p\1)\6(\ph(x) x_\Fp^a x_{\L\Fp}^b \6) 
		x_\oo^{-a} \Lx_\oo^{-b}. \]
Since this is a complex character we can consider its complex conjugate 
\[ \L{\ph_\C(x)} = (c\o\io_\oo\o\io_p\1)\6(\ph(x) x_\Fp^a x_{\L\Fp}^b \6) 
		x_\oo^{-b} \Lx_\oo^{-a}, \]
which is thus a Hecke character of weight $(b,a)$ (where $c : \C \to \C$ is 
complex conjugation). Transferring back we get an associated $p$-adic character
$\ph^\rh$ defined by 
\[ \ph^\rh(x) = (\io_p\o c\o\io_p\1)\6(\ph(x) x_\Fp^a x_{\L\Fp}^b \6) 
							x_\Fp^{-b} x_{\L\Fp}^{-a}, \]
where here we abuse notation and let $c : \L\Q \to \L\Q$ be the complex
conjugation induced by $\io_\oo$. This $\ph^\rh$ satisfies $\ph^\rh_\C =
\L{\ph_\C}$, as we wanted. 
\par
It's straightforward to check that $(\ph_1\ph_2)^\rh = \ph_1^\rh \ph_2^\rh$ for
any characters $\ph_1,\ph_2$, so therefore if we go to our family of characters
$\ps_m$ from above we have 
\[ \ps_m^\rh = ( \ps \al^{m-1} )^\rh = \ps^\rh (\al^\rh)^{m-1}. \] 
So to construct a family with specializations $\ps_m^\rh$, we need to
interpolate powers of $\al^\rh$. Unfortunately we cannot do this directly - to
mimic the construction of $\CA$ from $\al$ we need a character which maps onto
a subgroup $(1+\pi)^{\Z_p} \se \CO_F^\x$, and there's no reason that $\al^\rh$
should satisfy this. So we need to compare the character $\al^c$ we can
interpolate powers of to the character $\al^\rh$ we need. Since $\al$ has
infinity-type $(1,0)$ and finite-type character $\om_\Fp\1$, we see for any
ideal we have 
\[ \al(\Fa)\al^c(\Fa) = \al(\Fa\L\Fa) = \al(N(\Fa)\CO_K) 
		= \om\1_\Fp(N(\Fa)) N(\Fa), \]
and thus $\al\al^c = (\om\1\o N_\Q^K) \dt N$ as $p$-adic Hecke characters (with
$\om$ the Teichm\"uller character and $N$ the $p$-adic norm character). On the
other hand, the complex conjugate satisfies 
\[ \al_\C(\Fa) \al^\rh_\C(\Fa) = |\al_\C(\Fa)|^2 = N_\C(\Fa) \]
because $N_\C$ is the unique Hecke character of infinity-type $(1,1)$ that
takes values in positive real numbers. So $\al_\C \al^\rh_\C = N_\C$, and
converting back to $p$-adic characters we have $\al \al^\rh = N$. Comparing
these two equations we find 
\[ \al^c = (\om\1\o N_\Q^K) \al^\rh. \]
Since $\om^{p-1}$ is trivial, we can conclude that $P_m(\CA^c)$ is actually
equal to $(\al^\rh)^m$ in the arithmetic progression $m \ee 0\md{p-1}$. By
multiplying by an appropriate fixed character we can of course shift which
arithmetic progression we get:

\begin{lem}
	Suppose $\ps$ is a fixed finite-order character, and $\Ps$ is a family with
	specializations $\ps_m = P_m\o\Ps = \ps\al^{m-1}$ as considered before. Fix
	a residue class $a\md{p-1}$. Then if we define the family 
	\[ \Ps^\rh = \ps^\rh (\al^\rh)^{a-1} (\al^c)^{-a} \CA^c \]
	it satisfies $P_m \o \Ps^\rh = \ps_m^\rh$ for $m \ee a\md{p-1}$. Outside of
	this arithmetic progression, $P_m\o\Ps^\rh$ is $\ps_m^\rh \ox \om^{a-m}$.
\end{lem}

Our construction of CM forms from the previous section goes through nearly
identically (keeping in mind that $\al^c$ has conductor $\L\Fp$ rather than
$\Fp$), as long as we take care with the Teichm\"uller twists we needed. For
our purposes we'll want to choose the arithmetic progression where the twist is
trivial to match up with the arithmetic progression on which $p$ doesn't divide
the level of our newforms and thus $\Bg_\Ps$ contains $p$-stabilizations.

\begin{prop} \label{prop:FamilyOfConjugateCMNewforms}
	Let $\Fm$ be an ideal coprime to $p$, and let $\ps$ be a finite-order Hecke
	character of conductor $\Fm$ or $\Fm\Fp$ such that the $\Fp$-part of the
	finite-type of $\ps$ is $\om_\Fp^{a-1}$. Let $\Ps$ be the usual family of
	Hecke characters $\Ps : \DI_K/K^\x \to \CI$ such that $\ps_m = P_m\o\Ps = \ps
	\al^{m-1}$, and let $\Ps^\rh$ be the family above such that $P_m\o\Ps^\rh =
	\ps_m^\rh$ for $m\ee a\md{p-1}$. Then there is an associated $\CI$-adic
	newform
	\[ \Bg_{\Ps^\rh} = \s_{\Fa:(\Fa,\L\Fm\L\Fp)=1} 
			\Ps^\rh(\Fa) p^{(m-1)\ord_{\L\Fp(\Fa)}} q^{N(\Fa)} 
					\in \BS^\ord(dN(\L\Fm),\ch_K\ch_\ps\1\om^{2a-1};\CI) \]
	such that $P_m(\Bg_\Ps) = (g_{\ps_m}^\rh)^\sh = (g_{\ps_m}^\ft)^\rh$ for
	$m\ee a\md{p-1}$ and $P_m(\Bg_\Ps) = (g_{\ps_m})^\rh \ox \om^{a-m}$ for $m
	\nee a\md{p-1}$. If we view $\ps_m$ as a classical (complex-valued) Hecke
	character, in the $m\ee a\md{p-1}$ case the roots of the Hecke polynomial
	are $\al_{g^\rh} = \Lps_{m,\C}(\Fp)$ and $\be_{g^\rh} = \Lps_{m,\C}(\L\Fp)$.
\end{prop}

We remark that this is the best we could hope for - the characters of the
conjugated forms $(g_{\ps_m})^\rh$ vary as $\om^m$, but specializations
of any $\CI$-adic modular form need to vary as $\om^{-m}$, so a twist is
necessary to correct this outside of a single arithmetic progression. 

\subsection{Shifts of $\CI$-adic forms} \label{sec:ShiftsOfIAdicForms}

Finally, we take two general operations that one can do with classical modular
forms and describe explicitly how we apply these to $\CI$-adic forms. The first
is replacing a modular form $g(z)$ by $g(Mz/L)$ for $M,L$ natural numbers; if
$g$ lies in the space $S_m(N,\ch)$ then one checks $g(Mz/L)$ lies in the space
$S_m(\Ga_0(MN,L),\ch)$. The effect on $q$-expansions is easy to understand: if
$g(z) = \s a_n q^n$ then $g(Mz/L) = \s a_n q^{Mn/L}$. We take the notation that
$g|_{M/L}$ denotes the modular form $g|_{M/L}(z) = g(Mz/L)$. If $\Bg = \s A_n
q^n$ is a $\CI$-adic modular form, it's clear that to have this effect on
specializations we'd just want to use $\s A_n q^{Mn/L}$. 

\begin{lem}
	Suppose $M,L$ are positive integers that are coprime to each other. Let $\Bg
	= \s A_n q^n$ be a $\CI$-adic form in $\BS(N,\ch;\CI)$. Then if we define 
	\[ \Bg|_{M/L} = \s A_n q^{Mn/L}, \]
	we have $\Bg|_{M/L} \in \BS(\Ga_0(NM,L),\ch;\CI)$ and moreover
	$P_m(\Bg|_{M/L}) = P_m(\Bg)|_{M/L}$. 
\end{lem}

The second operation is taking a modular form $g(z)$ and replacing it with the
multiple $f(z) g(z)$, where $f$ is another fixed modular form. We want to take
a $\CI$-adic form $\Bg$ and ``shift'' it by a constant modular form $f$. We can
certainly multiply together the two formal $q$-series to get a new series such
that applying $P_m$ gives us $f \dt P_m(\Bg)$, a modular form of weight $m+k$.
\par
This is almost what we need; we just need to modify the series so that the
specializations have the correct weight. For $\La$-adic forms, we can do this 
by applying the automorphism $\si_{-k} : \La \to \La$ specified by
$\si_{-k}(1+X) = (1+X)(1+p)^{-k}$, as it's clear that $P_m\o\si_{-k} =
P_{m-k}$. To do this in the $\CI$-adic setting we need to extend $\si_{-k}$
to an automorphism of $\CI$; for the $\CI$ we construct in Section
\ref{sec:FamiliesOfHeckeCharacters} we can define $\si_{-k}$ explicitly as an
automorphism taking $\ga_\pi$ to $\ga_\pi (1+\pi)^{-k}$ (which satisfies
$P_m\o\si_{-k} = P_{m-k}$ for our canonical extensions of $P_m$ to $\CI$). 

\begin{prop}
	Suppose that we have an extension of $\si_{-k}$ to $\CI$ as discussed above.
	Let $\Bg = \s A_n q^n\in \BS(\Ga,\ch;\CI)$ be a $\CI$-adic form, and $f = \s
	b_m q^m \in M_k(\Ga_f,\ch_f;\CO_F)$ be a fixed $p$-adic modular form. Then
	the product of formal $q$-series 
	\[ f\dt \Bg = \b( \s_m b_m q^m \e) \b( \s_m \si_{-k}(A_n) q^n \e) \]
	is a $\CI$-adic form in $\BS(\Ga\i\Ga_f,\ch\ch_f;\CI)$ with specializations
	$P_m(f\dt\Bg) = f \dt P_{m-k}(\Bg)$.
\end{prop}
 
\section{$\CI$-adic Petersson inner products}
\label{sec:PeterssonInterpolate}

Now that we've set up the basic framework of $\CI$-adic modular forms (for the
types of extensions $\CI/\La$ discussed), in this section we'll recall results
of Hida that allow us to define an element in $\CI$ that interpolates the
Petersson inner products of specializations of two $\CI$-adic cusp forms. This
will be our main tool from Hida theory that will let us take the identities we
get from Ichino's formula and turn them into an identity of $p$-adic analytic
functions.

\subsection{Hida's linear functional} 

Naively, one might want a pairing on the space of $\CI$-adic cusp forms that
interpolates the Petersson inner products, but this is too much to hope for.
Instead, if we're given an ordinary $\CI$-adic newform $\Bh$ of level $N_0$,
we can define a linear functional $\l_\Bh : \BS^\ord(N_0 p,\ch;\CI) \to \CI$
such that $\l_\Bh(\Bg)$ interpolates the Petersson inner products of $\Bh_P$
and $\Bg_P$. Hida constructs such forms in Section 7 of \cite{Hida1988b}, and
we give a construction along the same lines. It's built from the results quoted
in Section \ref{sec:ModuleOfCongruence}, in particular the existence of an
idempotent $1_\Bh$ and the congruence number $H_\Bh$. 

\begin{prop}
	Let $\Bh \in \BS^\ord(N_0 p,\ch;\CI)$ be a $\CI$-adic newform of level
	$N_0$. Then for any $r$ we can define a linear functional $\l_\Bh :
	\BS^\ord(N_0 p^r,\ch;\CI) \to \CI$ by the formula 
	\[ \l_\Bh(\Bg) = a(1,1_\Bh \Bg) H_\Bh \]
	where $H_\Bh$ is the congruence number associated to $\Bh$ and $1_\Bh$ is
	the idempotent of the Hecke algebra associated to $\Bh$ (both as defined in
	the previous section). 
\end{prop}
\begin{proof}
	The assignment $\Bg \mt a(1,1_\Bh \Bg)$ is evidently linear. The only issue
	is that by definition, $1_\Bh$ is only an element of $\DT^\ord(N_0
	p^r,\ch;\CI) \ox Q(\CI)$, so we only get a linear functional
	$\BS^\ord(N_0 p^r,\ch;\CI) \to Q(\CI)$. However, by definition $H_\Bh$ is
	exactly what's needed to fix this; it annihilates the congruence module
	\[ \CC_0(\Bh;\CI) = \f{\DT(\Bh)_\CI \op \DT(\Bh)_\CI^\pr}
				{\DT^\ord(N_0 p^r,\ch;\CI)}; \]
	since $1_\Bh$ lies in $\DT(\Bh)_\CI$ by construction, $H_\Bh \dt 1_\Bh$ is
	trivial in this quotient, i.e. $H_\Bh \dt 1_\Bh$ lies in
	$\DT^\ord(N_0 p^r,\ch;\CI)$. So $H_\Bh 1_\Bh \Bg \in \BS^\ord(N_0
	p^r,\ch;\CI)$ and so its first Fourier coefficient is actually in in $\CI$. 
\end{proof}

We remark that if $r < s$ then the linear functional for level $N_0 p^s$
restricts to the linear functional for $N_0 p^r$; this follows because we can
check that the idempotent $1_\Bh$ for level $N_0 p^s$ projects to $1_\Bh$ for
level $N_0 p^r$. We can also extend this to a linear functional $\BS(N_0
p^r,\ch;\CI) \to \CI$ by precomposing it with the ordinary projector $e$; on
this larger space it's thus given by 
\[ \l_\Bh(\Bg) = a(1,e 1_\Bh \Bg) H_\Bh.  \] 
For simplicity, we'll abuse notation and let $1_\Bh \in \DT(N_0 p^r,\ch;\CI)$
denote $e 1_\Bh$, since this just amounts to taking the image of $1_\Bh \in
\DT^\ord$ under the inclusion $\DT^\ord \to \DT$ coming from multiplication by
$e$. If $P \in \CX(\CI,\CO_F)_\alg$ then we know we have 
\[ \l_\Bh(\Bg)_P = a(1,1_{\Bh_P} \Bg_P) H_{\Bh,P}. \]
and the right-hand side is computed in the appropriate space of classical cusp
forms (with $F$ coefficients). 

\subsection{Realizing the functional as a Petersson inner product}

To write the values $\l_\Bh(\Bg)_P$ explicitly, we need to compute $a(1,1_h
g) H_h$ for the classical modular forms $h = \Bh_P$ and $g = \Bg_P$. We won't
say anything about $H_h$ for now, and instead focus the term $a(1,1_h g)$. In
this section we deal with the case where $g$ is has the same level as the
newform $h$, and discuss how to work with $g$ of higher level in the next
section. Since $\Bh$ is an ordinary $\CI$-adic newform, $h$ will always be
either an ordinary newform (of level $N_0 p$) or a $p$-stabilization of an
ordinary newform (of level $N_0$); we'll need to deal with those two cases
separately. The form $g$, meanwhile, can be any cusp form of level $N_0 p$.
\par
Our first observation is that the projection $1_h$ is characterized by the
properties that $1_h h = h$, and $1_h h'= 0$ if $h'$ is an eigenform of any
Hecke operator $t$ and has a different eigenvalue from $h$. The latter property
follows from noting that if $t$ is such an operator, and $th' = \la' h'$ with
$\la_h(t) \ne \la'$, then we can compute 
\[ \la_h(t) 1_h h'= t\dt (1_h h') = 1_h \dt (th') 
		= 1_h \dt (\la' h') = \la' 1_h h' \]
which forces $1_h h' = 0$. 
\par
If we assume $h \in S_m(N_0 p,\ch;F)$ actually has coefficients in a number
field $F_0 \se F$, and restrict to elements $g \in S_k(N_0 p,\ch;F_0)$ then
we can compute $a(1,1_h g) \in F_0$ by computing this in $\C$ (where we can
pass to an element $1_h \in \DT_\C(S_k(N_0 p,\ch))$ characterized by the same
property as before). Since $g \mt a(1,1_h g)$ is then a linear functional on
$S_k(N_0 p,\ch)$, by duality for the Petersson inner product there exists a
modular form $h' \in S_k(N_0 p,\ch)$ such that $\g{g,h'} = a(1,1_h g)$ for all
$g$. In fact, it's sufficient to find $h'$ such that $g \mt \g{g,h'}$ is a
scalar multiple of $g \mt a(1,1_h g)$, and then comparing values at $h$ gives
us $a(1,1_h g) = \g{g,h'}/\g{h,h'}$. It will turn out that $h'$ is either equal
to or closely related to $h$ (and its exact form will come out from our
analysis). We start by checking: 

\begin{lem}
	Let $h \in S_m(N_0 p,\ch)$ be either a $p$-ordinary newform or the
	$p$-stabilization of an ordinary newform. Then the unique form $h' \in
	S_m(N_0 p,\ch)$ such that $\g{g,h'} = a(1,1_h g)$ lies in the prime-to-$N_0
	p$ Hecke eigenspace $U(h) \se S_m(N_0 p,\ch)$ of $h$.
\end{lem}
\begin{proof}
	We know $S_k(N_0 p,\ch)$ decomposes as an orthogonal direct sum of
	prime-to-$N_0 p$ Hecke eigenspaces, and in fact this is a decomposition of
	spaces $U(f)$ as $g'$ ranges over all newforms of character $\ch$ and level
	dividing $N_0 p$. Moreover, we know $1_h$ annihilates these subspaces for
	every $f \ne h$, and thus $a(1,1_h -)$ is trivial on all of them. The only
	way that $\g{-,h'}$ can be trivial on all of them is if $h'$ has no component
	in each of these spaces $U(f)$, i.e. $h' \in U(h)$. 
\end{proof}

We now need to split up into cases based on whether $h$ is a newform or a
$p$-stabilization of one. If $h \in S_m(N_0 p,\ch)$ is actually a newform of
level $N_0 p$, he multiplicity one theorem tells us that $U(h)$ is
one-dimensional, so $h'$ must be a scalar multiple of $h$ itself. Thus we get

\begin{cor}
	Let $h \in S_m(N_0 p,\ch)$ be a $p$-ordinary newform. Then for any $g\in
	S_m(N_0 p,\ch)$ we have 
	\[ a(1,1_h g) = \f{\g{g,h}}{\g{h,h}}. \]
\end{cor}

So, for the rest of this section we work with the more involved case of the
$p$-stabilization of a newform. For convenience we change our notation a bit
for this case, and now let $h$ denote the original $p$-ordinary newform in
$S_m(N_0,\ch)$. We know that $U(h) \se S_m(N_0 p,\ch)$ is two-dimensional,
spanned by $h(z)$ and $h(pz)$. Moreover, the operator $T(p)$ of level $N_0 p$
has two eigenforms in this space with different eigenvalues; more specifically
we take the notation that 
\[ x^2 - a_p(h) x + p^{m-1} \ch(p) = (x - \al_h)(x - \be_h) \]
with $|\al_h|_p = 1$ and $|\be_h|_p < 1$. By Lemma
\ref{lem:StabilizationOfNewforms} the two eigenforms are 
\[ h^\sh(z) = h(z) - \be_h h(pz) \qq\qq h^\ft(z) = h(z) - \al_h h(z), \]
with eigenvalues $\al_h$ and $\be_h$, respectively. Thus $h^\sh$ spans the
one-dimensional ordinary subspace of $U(h)$, and $h^\sh$ is what actually
arises as the specialization of a $\CI$-adic newform. 
\par
So, we want to compute the linear functional $a(1,1_{h^\sh}-)$ on $S_m(N_0
p,\ch)$. We can pin this down exactly because we know it's zero on the
orthogonal complement of the two-dimensional space $U(h)$, and we also know its
behavior on $U(h)$ because $1_{h^\sh} h^\sh = h^\sh$ (by definition) and $1_h
h^\ft = 0$ (since $h^\ft$ has a $T(p)$-eigenvalue different from $h^\sh$). So
$a(1,1_{h^\sh}-)$ is (up to a scalar) equal to the form $\g{-,h'}$ for some $h'
\in U(h) = \C h(z) \op \C h(pz)$ that's orthogonal to $h^\ft$. Finding the
appropriate linear combination $h'$ is still a bit delicate, though; to do this
we'll need to understand how to compute the Petersson inner products
$\g{h(z),h(pz)}$, $\g{h(pz),h(z)}$, and $\g{h(pz),h(pz)}$ in terms of
$\g{h(z),h(z)} = \g{h,h}$. The starting point is computing the last product;
this can be done in great generality by change-of-variables as follows. 

\begin{lem} \label{lem:ScalingPetersson}
	Let $f,g$ be two weight-$k$ cusp forms (of any level). Then for any prime
	$p$ and any integer $i$, we have 
	\[ \g{f(pz+i),g(pz+i)} = p^{-k} \g{f(z),g(z)}. \]
\end{lem}
\begin{proof}
	We can pick some common congruence subgroup $\Ga(M)$ for which all of $f(z)$,
	$f(pz+i)$, $g(z)$, and $g(pz+i)$ are modular. We have 
	\[ \g{f(pz+i),g(pz+i)} 
		= \f{1}{[\LGa:\LGa(M)]} \S_{D(M)} f(pz+i) g(pz+i) \Im(z)^k d\mu(z) \]
	\[ = p^{-k} \f{1}{[\LGa:\LGa(M)]} 
				\S_{D(M)} f(pz+i) g(pz+i) \Im(pz+i)^k d\mu(z),\]
	where $d\mu$ is the standard hyperbolic measure on $\DH$ and $D(M)$ is a
	fundamental domain for $\Ga(M)$. Using change-of-variables on the M\"obius
	transformation $\si : z \mt pz+i$, under which $d\mu$ is invariant, gives
	\[ \g{f(pz+i),g(pz+i)} 
		= p^{-k} \f{1}{[\LGa:\LGa(M)]} \S_{\si D(M)} f(z) g(z) \Im(z)^k d\mu(z).\]
	The set $\si D(M)$ is a fundamental domain for the group $\si \LGa(M)
	\si\1$, and we can compute that this conjugate still lies in $\LGa$ and has
	index $[\LGa:\LGa(M)]$. Thus the right hand side of our equation is $p^{-k}
	\g{f(z),g(z)}$ computed on this new congruence subgroup.
\end{proof}

Returning to our situation where $h$ is an eigenform of level $N_0$ with $p\nd
N_0$, we can use the result above plus the Hecke relation for $T(p) h$ to work
out what $\g{h(z),h(pz)}$ and $\g{h(pz),h(z)}$ are.

\begin{lem} \label{lem:PairingFormWithSelfTranslate}
	Let $h \in S_m(N_0,\ch)$ be a $T(p)$-eigenform with eigenvalue $a_p$, for a
	prime $p\nd N_0$. Then we have 
	\[ \g{h(z),h(pz)} = \f{a_p}{p^{m-1}(p+1)} \g{h,h} \qq\qq
		\g{h(pz),h(z)} = \f{\Lch(p) a_p}{p^{m-1}(p+1)} \g{h,h}. \]
\end{lem}
\begin{proof}
	To compute these, we start by noting that $f$ being an eigenform tells us 
	\[ a_p h(z) = \ch(p) p^{m-1} h(pz) + \f{1}{p} \s_{i=0}^{p-1} h\pf{z+i}{p}.\]
	Applying $\g{-,h(z)}$ to both sides and using linearity in the first
	coordinate we get 
	\[ a_p \g{h,h} = \ch(p) p^{m-1} \g{h(pz),h(z)} 
			+ \f{1}{p} \s_{i=0}^{p-1} \G{h\pf{z+i}{p},h(z)}. \] 
	The previous lemma plus invariance of $h$ under $z \mt z+i$ gives that 
	$\g{h(z),h(pz)} = p^{-m} \g{h(\f{z+i}{p}),h(z)}$ and thus
	\[ a_p \g{h,h} = \ch(p) p^{m-1} \g{h(pz),h(z)} + p^m \g{h(z),h(pz)}. \]
	An identical computation with $\g{h(z),-}$ (using antilinearity in the
	second coordinate) tells us 
	\[ \Lch(p) a_p \g{h,h} 
				= \Lch(p) p^{m-1} \g{h(z),h(pz)} + p^m \g{h(pz),h(z)}. \]
	Here we use the relation $\L{a}_p = \Lch(p) a_p$, which follows from
	the identity $T(p)^* = \ch(p) T(p)$. We can then take these two equations
	and solve for $\g{h(z),h(pz)}$ and $\g{h(pz),h(z)}$ to obtain the lemma.
\end{proof}

We can then complete our computation of $a(1,1_h g)$ in the case when we're
working with a $p$-stabilization of a newform. 

\begin{prop} \label{prop:PeterssonForPStabilization}
	Suppose $p\nd N_0$ and $h \in S_m(N_0,\ch)$ is a $p$-ordinary newform, and 
	let $\al = \al_h$ and $\be = \be_h$ be the roots of $x^2 - a_p(h) x +
	p^{m-1} \ch(p)$ with $|\al|_p = 1$ and $|\be|_p < 1$, as before. Then the
	linear functional $a(1,1_{f^\sh} -)$ is given by 
	\[ a(1,1_{h^\sh} g) = \f{\g{g,h^\na}}{\g{h^\sh,h^\na}} 
			\qq\qq h^\na = h(z) - p\be h(pz) \]
\end{prop}
\begin{proof}
	By the above, we know we can write $a(1,1_{h^\sh} -) =
	\g{-,h'}/\g{h^\sh,h'}$ for any $h'$ in a unique one-dimensional subspace of
	this. Assuming the subspace isn't $\C h(pz)$ (which will be justified when
	we find a different subspace that works), we can take $h' = h(z) - C h(pz)$;
	since $\g{h^\sh,h'}/\g{h^\sh,h'} = 1 = a(1,1_h -)$ trivially, we just need
	to choose $C$ such that we have $\g{h^\ft,h'} = 0 = a(1,1_{h^\sh} h^\ft)$. 
	\par
	Thus, we want to solve the equation 
	\[ \g{h^\ft,h'} = \g{h(z) - \al h(pz), h(z) - C h(pz)} = 0 \]
	for $C$. Expanding this by linearity we get 
	\[ 0 = \g{h(z),h(z)} - \al \g{h(pz),h(z)} - \LC \g{h(z),h(pz)} 
			+ \al\LC \g{h(pz),h(pz)}. \]
	Substituting in the results of the previous two lemmas, we get 
	\[ 0 = \g{h,h} - \al \f{\Lch(p) a_p}{p^{m-1}(p+1)} \g{h,h} - 
		\LC \f{a_p}{p^{m-1}(p+1)} \g{h,h} + \al \LC p^{-m} \g{h,h}.\]
	Rearranging and simplifying we ultimately get $\LC = \Lch(p) \al p$. By the
	Ramanujan conjecture (which is elementary in the CM case we're interested
	in), we have $\Lch(p)\al = \Lbe$, so we get that $C = p\be$ is a solution to
	our equation, and thus $h^\na = h(z) - p\be h(pz)$ satisfies
	$\g{h^\ft,h^\na} = 0$. 
\end{proof}

We end this section with a few remarks. First of all, in Lemma
\ref{lem:StabilizationOfNewforms} we define $h^\sh$ and $h^\ft$ for any modular
form $h$ that's an eigenform of the prime-to-$p$ Hecke operator $T(p)$, so it
is reasonable to also define $h^\na = h(z) - p\be h(pz)$ for any such $h$. In
later sections we will use this notation in the case where $h(z) = h_0(Lz)$ for
$h_0$ a newform and $L$ an integer prime to $p$. Such a form remains at
prime-to-$p$ level and an eigenform for $T(p)$; note that in this case
$h_0^\sh(Lz)$ and $h^\sh(z)$ agree and likewise for $\ft$ and $\na$, so there
is no chance for notational confusion.
\par
Also, we recall that Hida makes a similar computation in Proposition 4.5 of
\cite{Hida1985}, and concludes that he can take $h^\na = h^\sh_\rh|w_{N_0p}$,
where $w_N$ is the matrix 
\[ w_N = \M{cc}{ 0 & -1 \\ N & 0 } \]
and if $h = \s a_n q^n$ we let $h_\rh = \s \L{a}_n q^n$. Comparing our results
(and using the Atkin-Lehner relation that $h_\rh|w_{N_0}$ is a multiple of $h$)
we find that if $h$ is a $p$-ordinary newform of level $N_0$ then $h^\na$ and
$h^\sh_\rh|w_{N_0p}$ are scalar multiples of each other. (Quoting Hida's result
and doing this would be a quicker way to get our formula for $h^\na$, but the
computations and techniques above will be important later on)

\subsection{Getting Petersson inner products for higher levels}
\label{sec:PeterssonOfHigherLevels}

To summarize the computations above, for a $\CI$-adic newform $\Bh \in
\BS^\ord(N_0 p,\ch;\CI)$ of level $N_0$ and an arbitrary $\CI$-adic form $\Bg
\in \BS(N_0 p,\ch;\CI)$, the element $\l_\Bh(\Bg) \in \CI$ has specialization
at $P \in \CX_\alg(\CI;\CO_F)$ given by 
\[ a(1,1_h g) = \f{\g{g(z),h(z)}}{\g{h(z),h(z)}} H_{\Bh,P} \qq\tx{or}\qq 
	a(1,1_{h^\sh} g) = \f{\g{g(z),h^\na(z)}}{\g{h^\sh(z),h^\na(z)}} H_{\Bh,P} \]
in the cases where $\Bh_P = h$ is a newform or $\Bh_P = h^\sh$ is a
$p$-stabilization of a newform, respectively (and setting $g = \Bg_P$). In this
section we want to expand this construction to allow $\Bg$ to be a $\CI$-adic
form in a space $\BS(M_0 p^r,\ch;\CI)$ of higher level, which we'll need for
our applications. 
\par
First of all, we recall that we've actually defined extensions of our linear
functional $\l_\Bh$ to spaces $\BS(N_0 p^r,\ch;\CI)$ where we allow higher
$p$-power level. So we just need to compute $\l_\Bh(\Bg)$ and thus $\l_h(g)$ on higher-level forms $g$. Fortunately for us there's a trick we can use involving
the $T(p)$ operator, as in Sections 3 and 4 of \cite{Hida1985}. On one hand
$T(p)$ is known to take $S_m(M_0 p^r,\ch)$ into $S_m(M_0 p^{r-1},\ch)$ for
$r\ge 2$. On the other, since $h$ or $h^\sh$ (as appropriate) is an eigenform
for $T(p)$ with eigenvalue $\al_h$ (the root of the Hecke polynomial which is a
$p$-adic unit under our embedding into $\L\Q_p$) we can see $1_h T(p) = 1_h
\al_h$ or $1_{h^\sh} T(p) = 1_{h^\sh} \al_h$. 
\par
Analyzing our functional in these two different ways, we find that if $g
\in S_m(M_0 p^r,\ch)$ then we have 
\[ a(1,1_h g) = \al_h^{-(r-1)} a(1,1_h T(p)^{r-1} g) 
		= \al_h^{-(r-1)} \f{\g{T(p)^{r-1} g(z),h(z)}}{\g{h(z),h(z)}} \]
or 
\[ a(1,1_{h^\sh} g) = \al_h^{-(r-1)} a(1,1_{h^\sh} T(p)^{r-1} g) 
	= \al_h^{-(r-1)} \f{\g{T(p)^{r-1} g(z),h^\na(z)}}{\g{h^\sh(z),h^\na(z)}}. \]
Finally, we note that if we view the Petersson inner products here as being
over $\Ga_0(N_0 p^r)$, we can rewrite it in a way that avoids the $T(p)$
operator, using the following lemma. 

\begin{lem} \label{lem:AdjointToTpOperator}
	Suppose $f,f'$ are two elements of $S_k(N_0 p^r,\ch)$ for $r \ge 1$. Then we
	have 
	\[ \g{T(p)f,f'} = p^k \g{f(z),f'(pz)}. \] 
\end{lem}
\begin{proof}
	We can write out the definition of the Hecke operator and get
	\[ \g{T(p)f,f'} = \G{ \s_{i=0}^{p-1} p^{k/2-1} f \9|_k \M{cc}{1&i\\0&p},f'}
			= p^{k/2-1} \s_{i=0}^{p-1} \G{f\9|_k \M{cc}{1&i\\0&p},f'}. \]
	Now, slash operators don't preserve individual spaces like $S_k(N_0
	p^r,\ch)$, but they can be viewed as operators on the infinite-dimensional
	space of all weight-$k$ modular forms. There, it's straightforward to check
	that if $\al$ is a matrix, the slash operator $[\al]_k$ has an adjoint given
	by $[\be]_k$ where $\be = \det(A) A\1$ is the adjugate matrix. Thus: 
	\begin{multline*}
		\g{T(p)f,f'} = p^{k/2-1} \s_{i=0}^{p-1} \G{f,f'\9|_k \M{cc}{p&-i\\0&1}}\\
			= p^{k/2-1} \s_{i=0}^{p-1} \g{f(z), p^{k/2} f'(pz-i))} 
			= p^k \g{f(z),f'(pz)}. \QED
	\end{multline*}
\end{proof}

Therefore, we conclude:

\begin{prop} \label{prop:InterpolatePeterssonHigherPPowerLevel}
	Let $\Bh$ be a $\CI$-adic newform in $\BS(N_0 p,\ch;\CI)$ with weight-$m$
	specialization $h_m$ or $h_m^\sh$ (as appropriate), and $\Bg$ any $\CI$-adic
	form in $\BS(N_0 p^r,\ch;\CI)$ with weight-$m$ specialization $g_m$. Then 
	$P_m(\l_\Bh(\Bg))$ is equal to 
	\[ \f{p^{m(r-1)}}{\al_h^{r-1}} \f{\g{g(z),h(p^{r-1} z)}}{\g{h(z),h(z)}} 
		H_{\Bh,P} \qq\tx{or}\qq \f{p^{m(r-1)}}{\al_h^{r-1}} 
			\f{\g{g(z),h^\na(p^{r-1} z)}}{\g{h^\sh(z),h^\na(z)}} H_{\Bh,P} \]
	as appropriate.  
\end{prop}

Next we deal with raising the prime-to-$p$ part of the level. In this case
$\l_\Bh$ isn't already defined on $\BS(M_0 p^r,\ch;\CI)$, so we have to do 
some sort of construction to extend it. Here we use some \E{trace operators}. 
If $\Ga_0(M_0 p^r,L_0)$ contains $\Ga_0(N_0 p^r)$, and we want to start with a
modular form on the former group and obtain one on the latter, we can define a
linear function 
\[ \trc : S_k(\Ga_0(M_0 p^r,L_0),\ch) \to S_k(N_0 p^r,\ch)
	\qq\qq \trc(f) = \s_\ga \ch\1(\ga) f|[\ga]_k \]
where $\ga$ runs over a set of coset representatives of $\Ga_0(M_0
p^r,L_0)\q\Ga_0(N_0 p^r)$. This has the property that if $f \in S_k(\Ga_0(M_0
p^r,L_0),\ch)$ and $f' \in S_k(N_0 p^r,\ch)$ we have 
\[ \g{\trc(f),f'} = \s_\ga \ch\1(\ga) \G{f|[\ga]_k,f'} 
		= \s_\ga \G{f,\ch(\ga)f'|[\ga\1]_k} = \s_\ga \g{f,f'}. \]
Thus the resulting sum is just $[\Ga_0(M_0 p,L_0):\Ga_0(N_0 p)]$ times
$\g{f,f'}$, so we have 
\[ \g{\trc(f),f'} 
		= \g{f,f} \f{M_0 L_0}{N_0} \p_{q | M_0 L_0, q\nd N_0} (1+q\1). \]
\par
It's straightforward that this linear transformation on complex vector spaces
preserves $K_0$-rationality, so would pass to a transformation 
\[ \trc : S_k(\Ga_0(M_0 p^r,L_0),\ch;K) \to S_k(N_0 p^r,\ch;K). \]
Furthermore, in this case it preserves integrality in $K$. Since we aren't
changing the $p$-power of our level, we can choose our coset representatives to
satisfy $\ga \ee 1\md p$ and then the action of these $\ga$'s matches up with
the action of $\SL_2(\Z/M_0 p \Z)$ discussed in Section 1.IV of
\cite{Hida1988b}, which preserve the space $\LS(M_0 p;\CO_F)$. Thus, if $\ch$
is any character of $(\Z/N_0\Z)^\x$, we have a trace operator 
\[ \trc : \LS(\Ga_0(M_0,L_0);\CO_F) \to \LS(M_0;\CO_F)
	\qq\qq \trc(f) = \s_\ga \ch\1(\ga) f|[\ga]_k \]
Finally we can extend this to a $\CI$-linear operator
\[ \trc : \BS(\Ga_0(M_0 p^r,L_0),\ch;\CI) \to \BS(N_0 p^r,\ch;\CI) \]
via our measure-theoretic perspective, defining $\mu_{\trc(\Bg)}(x \mt f(x)) =
\trc(\mu_\Bg(x\mt f(x)))$. We conclude: 

\begin{prop} \label{prop:InterpolatePeterssonForHigherLevel}
	Let $\Bh$ be a $\CI$-adic newform of level $N_0$ and character $\ch$, and
	let $h_m$ be the newform associated to $P_m(\Bh)$ (so either $h_m =
	P_m(\Bh)$ is of level $N_0 p$, or $h_m$ is of level $N_0$ and $P_m(\Bh) =
	h_m^\sh$) and $\al_h$ the $T(p)$-eigenvalue of $P_m(\Bh)$. If $\Bg \in
	\BS(\Ga_0(M_0 p^r,L_0),\ch;\CI)$ is any $\CI$-adic form (for $r\ge 1$ and
	$M_0$ a multiple of $N_0$), and we let $g_m = P_m(\Bg)$, then the
	specialization of the element $\l_\Bh(\trc(\Bg)) \in \CI$ at $P_m$ is given
	by 
	\[ C \f{\g{g_m(z),h_m(p^{r-1} z)}}{\g{h_m(z),h_m(z)}} H_{\Bh,P} 
			\qq\tx{or}\qq   C 
	 	\f{\g{g_m(z),h_m^\na(p^{r-1} z)}}{\g{h_m^\sh(z),h_m^\na(z)}} H_{\Bh,P} \]
	in the cases where $P_m(\Bf)$ is $h_m$ or $h_m^\sh$, respectively, where  
	the constant is 
	\[ C = \f{p^{m(r-1)}}{\al_h^{r-1}} \f{M_0 L_0}{N_0} 
			\p_{q | M_0 L_0, q\nd N_0} (1+q\1). \]
\end{prop}

\subsection{Euler factors from $p$-stabilizations}
\label{sec:EulerFactorsStabilization}

From the previous section, we've seen that our element $\l_\Bh(\Bg) \in \CI$
interpolating Petersson inner products will have specializations of the form
\[ \f{\g{g,h^\na}}{\g{h^\sh,h^\na}} \]
when we specialize at points $P$ where $\Bh_P$ is the $p$-stabilization of a
newform $h$. Similarly, $g$ may also be a $p$-stabilization of a newform (or
related to one). On the other hand, our automorphic formulas usually only
involve the newforms themselves. So we'd like to relate the ratio above to one
not involving $p$-stabilizations.
\par
More specifically, to fit the setup and notation we'll need later on, we'll
take the following setup. We will take $f,g,h$ to be modular forms of weights
$k$, $m-k$, and $m$, respectively, such that the central characters satisfy
$\ch_f \ch_g = \ch_h$. We assume the forms $g$ and $h$ are of prime-to-$p$
level, and are $p$-ordinary eigenforms of the prime-to-$p$ Hecke operator
$T(p)$. On the other hand, $f$ is not assumed to be either $p$-ordinary or to
have prime-to-$p$ level; but we do require that one of the following conditions
two holds for it:
\begin{itemize}
	\item $f$ is of prime-to-$p$ level, and an eigenform for the prime-to-$p$
		Hecke operator $T(p)$.
	\item $f$ has level exactly divisible by $p^r$ for some $r > 1$, $f$ is an
		eigenform for the $p$-level Hecke operator $T(p)$, and $f$ is new at 
		level $p^r$.
\end{itemize}
The reason we specify ``eigenforms of $T(p)$'' rather than ``newforms'' in our
setup is because we will need to allow forms like $h_0(Lz)$ where $h_0$ is an
actual newform and $L$ is a prime-to-$p$ integer; such forms do indeed remain
eigenforms for the appropriate $T(p)$. Also, for our purposes $f$ being ``new
at level $p^r$'' means that $f$ is perpendicular to $\ps(z)$ and $\ps(pz)$ for
any $\ps$ of level $N_0 p^{r-1}$. Thus if $f_0$ is a newform of level $N_0 p^r$
and $L$ is coprime to $p$, $f_0(Lz)$ is still new at $p^r$.
\par
Given this setup, we will need to compute an explicit constant $(*)$ such that
we have an equality
\[ \f{\g{f(z) g^\sh(z),h^\na(p^{r-1} z)}}{\g{h^\sh(z),h^\na(z)}}
		= (*) \f{\g{f(z)g(z),h(p^{r_0}z)}}{\g{h(z),h(z)}} \]
where $p^{r_0}$ is the exact power of $p$ dividing the level of $f$ and $r =
\max\{ 1,r_0 \}$. The constant $(*)$ will turn out to be the Euler factor at
$p$ which needs to be removed from the $p$-adic $L$-function we're
constructing. It will be described in terms of the roots of the Hecke
polynomial at $p$ for $f$, $g$, and $h$. Such roots will always include
$\al_g,\be_g,\al_h,\be_h$ in the notation of Lemma
\ref{lem:StabilizationOfNewforms}. For $f$, we will have $\al_f,\be_f$
determined similarly if $r_0 = 0$ (though neither will necessarily be a
$p$-adic unit), or just $a_f$ if $r_0 \ge 1$ (which will satisfy $a_f = 0$ if
$r_0 \ge 2$ in the cases we're considering).
\par
We break the computation of $(*)$ up into three parts: one relating the
denominators on both sides, the second relating the numerators if $p$ does not
divide the level of $f$, and the last relating the numerators when $p$ does
divide the level of $f$. This computation is related to the one carried out in
Section 4 of \cite{DarmonRotger2014}.

\subsubsection{Relating $\g{h^\sh,h^\na}$ to $\g{h,h}$} 
We start by working with the denominator of the expression, since here the
computation is an extension of what we've already done in the previous section.
Expanding out the definition of $h^\sh$ and $h^\na$ tells us that
$\g{h^\sh,h^\na}$ equals
\[ \g{h(z),h(z)} - \be_h \g{h(pz),h(z)} - p\Lbe_h \g{h(z),h(pz)} 
		+ p\be_h\Lbe_h \g{h(pz),h(pz)}. \]
We can use Lemmas \ref{lem:ScalingPetersson} and
\ref{lem:PairingFormWithSelfTranslate} to express each inner product in terms
of $\g{h,h} = \g{h(z),h(z)}$ to obtain
\[ \g{h^\sh,h^\na} = \g{h,h} 
		\b( 1 - \f{\Lch(p) \be_h a_p(f)}{p^{k-1}(1+p)} 
				- \f{p\Lbe_h a_p(f)}{p^{k-1}(1+p)} + \f{p\be_h\Lbe_h}{p^k} \e). \]
Now, we can get rid of the complex conjugates via the Ramanujan conjecture: we
know $\Lbe_h = \Lch(p) \al_h$ and also $|\be_h|^2 = p^{k-1}$ (so the rightmost
term is identically 1). Thus the term in parentheses simplifies to 
\[ 2 - \f{\Lch(p)(\be_h + p \al_h)a_p(f)}{p^{k-1}(1+p)}. \]
Using that $p^{k-1} = \Lch(p) \al_h \be_h$, that $a_p(f) = \al_h+\be_h$, and
collecting everything with a common denominator gives 
\[ \f{ 2(1+p)\al_h\be_h - (\be_h + p\al_h)(\al_h+\be_h)}{\al_h\be_h(1+p)}
	= \f{\al_h\be_h + p\al_h\be_h - \be_h^2 - p\al_h^2}{\al_h\be_h(1+p)} \]
Cancelling and factoring gives
\[ \f{1 + p - \be_h/\al_h - p\al_h/\be_h}{1+p}
	= \f{(1 - \be_h/\al_h)(1 - p\al_h/\be_h)}{(1+p)}. \]
For our purposes later it's helpful to rearrange this as 

\begin{prop} \label{prop:EulerFactorDenominator}
	With our setup as above (with $h$ a newform of prime-to-$p$ level), we have
	\[ \g{h^\sh,h^\na} = \g{h,h} 
	\dt \f{(-\al_h/\be_h)(1 - \be_h/\al_h)(1 - p\1 \be_h/\al_h)}{(1+p\1)}. \]
\end{prop}

\subsubsection{Relating $\g{fg^\sh,h^\na}$ to $\g{fg,h}$ when the level of $f$
is prime to $p$} 
To simplify notation in this case, we set $\g{fg,h} = \g{f(z) g(z), h(z)}$,
$\g{fg_p,h} = \g{f(z) g(pz),h(z)}$, $\g{fg_p,h_p} = \g{f(z)g(pz),h(pz)}$, and
so on; with this convention our goal is to express
\[ \g{fg^\sh,h^\na} = \g{fg,h} - \be_g \g{fg_p,h} - p\Lbe_h \g{fg,h_p} 
		+ p \be_g \Lbe_h \g{fg_p,h_p} \]
as a constant times $\g{fg,h}$. 
\par
To express $\g{fg_p,h_p}$ in terms of $\g{fg,h}$ we will obtain two equations
that are linear in the quantities $\g{fg_p,h_p}$, $\g{f_pg,h}$, and $\g{fg,h}$
and solve. The first equation comes from the Hecke eigenvalue relation for $f$: 

\begin{lem}
	In our setup, we have 
	\[ (\al_f+\be_f)\g{fg,h} = \al_f\be_f \g{f_p g,h} + p^m \g{fg_p,h_p}. \]
\end{lem}
\begin{proof}
	This follows from starting with the equation $a_f f(z) = T(p) f(z)$;
	expanding out the definition of the Hecke operator we can write this as 
	\[ (\al_f + \be_f) f(z) 
			= \al_f \be_f f(pz) + \f{1}{p} \s_{i=0}^{p-1} f\pf{z+i}{p}. \]
	We then apply the linear functional $\g{-g,h}$ to both sides of this. Since
	we can compute $\g{f(\f{z+i}{p})g(z),h(z)} = p^m \g{f(z)g(pz),h(pz)}$ by
	substituting $z \mt pz-i$ and applying Lemma \ref{lem:ScalingPetersson}, 
	we obtain the desired identity. 
\end{proof}

To get a second equation we will use the fact that the Petersson inner product
is invariant under applying a slash operator to both sides simultaneously; in
particular we will apply an Atkin-Lehner operator defined by a matrix 
\[ \om_p = \M{cc}{ p & 1 \\ Npc & pd } = \M{cc}{ 1 & 1 \\ Nc & pd } 
		\M{cc}{ p & 0 \\ 0 & 1 } = \ga_p \de_p \]
where $N$ is the common level of $f,g,h$ (prime to $p$ by assumption) and $c,d$
chosen so this has determinant $p$ (which forces $pd \ee 1 \md N$; we follow
the conventions of \cite{Asai1976}). Then:

\begin{lem} \label{lem:AtkinLehnerToPetersson}
	In our setup, we have $\g{f_p g,h} = \Lch_f(p) p^{m-k} \g{fg_p,h_p}$.
\end{lem}
\begin{proof}
	By invariance of the Petersson inner product we have $\g{f_p g,h} =
	\g{f_p|[\om_p]_kg|[\om_p]_{m-k}, h|[\om_p]_m}$ so we just need to determine
	what these translates are. This is straightforward; $h$ is invariant under
	$\ga_p \in \Ga_1(N)$ so we get $h|[\om_p]_m = h|[\de_p]_m = p^{m/2} h(pz)$
	and similarly for $g$. Meanwhile writing $f_p = p^{-k/2} f|[\de_p]_k$ and
	computing 
	\[ \de_p \ga_p \de_p = \M{cc}{ ap & b \\ Nc & d } \M{cc}{ p & 0 \\ 0 & p }\]
	we find $f_p|[\om_p]_k = p^{-k/2} \Lch_f(p) f(z)$ as the first matrix is in
	$\Ga_0(N)$ so transforms $f$ by $\ch_f(d) = \Lch_f(p)$ and the second fixes
	any modular form. 
\end{proof}

Combining the equations in the two lemmas allows us to conclude 
\[ \g{fg_p,h_p} = \f{\al_f + \be_f}{p^{m-1}(1 + p)} \g{fg,h}. \]
We can then carry out the process of the first lemma for the eigenform
equations for $g$ and $h$, and the process of the second lemma for $\g{fg_p,h}$
and $\g{fg,h_p}$ and obtain 
\[ \g{f g_p,h} = \Lch_g(p) \f{\al_g + \be_g}{p^{m-k-1}(1 + p)} \g{fg,h}, 
\qq\qq
 	\g{f g,h_p} = \ch_h(p) \f{\Lal_h + \Lbe_h}{p^{m-1}(1 + p)} \g{fg,h}. \]
\par
We now can substitute these terms to our expansion of $\g{fg^\sh,h^\na}$ above
and obtain 
\[ \g{fg,h} \b( 1 
		- \f{\be_g \Lch_g(p) (\al_g + \be_g)}{p^{m-k-1}(1 + p)} 
		- \f{p \Lbe_h \ch_h(p) (\Lal_h + \Lbe_h)}{p^{m-1}(1 + p)} 
		+ \f{ p \be_g \Lbe_h (\al_f + \be_f)}{p^{m-1}(1 + p)} \e). \]
The term in parenthesis needs to be manipulated and factored. This tedious but
elementary computation is omitted here, but is carried out in detail in Section 3.4.2 of the author's thesis \cite{Collins2015}.

\begin{prop} \label{prop:EulerFactorNumeratorPrimeToP}
	Suppose we're in the above setup (in particular with the level of $f$
	coprime to $p$). Then we have
	\[ \g{f(z) g^\sh(z),h^\na(z)} = \g{f(z)g(z),h(z)} 
			\f{(-\al_h/\be_h)}{1+p\1} 
		\b( 1 - \f{\be_g\al_f}{\al_h} \e) \b( 1 - \f{\be_g\be_f}{\al_h} \e). \]
\end{prop}

\subsubsection{Relating $\g{fg^\sh,h^\na}$ to $\g{fg,h}$ when $p$ divides the
level of $f$} 
In this case $f$ has level divisible by some power $p^r$ with $r\ge 1$, and our
goal is to relate the quantity $\g{f(z) g^\sh(z),h^\na(p^{r-1} z)}$ to the
quantity $\g{f(z) g(z),h(p^r z)}$ which appears in our automorphic theory. Here
we follow the same approach of writing 
\[ \g{fg^\sh,h^\na} = \g{fg,h} - \be_g \g{fg_p,h} - p\Lbe_h \g{fg,h_p} 
		+ p \be_g \Lbe_h \g{fg_p,h_p}, \]
though here $h$ and $h_p$ denote $h(p^{r-1}z)$ and $h(p^rz)$ respectively. In
this situation several of the terms will be zero, with the underlying reason
being that $f$ is new at $p$ while $g$ and $h$ are not: 

\begin{lem} \label{lem:PairNewformWithLowerLevel}
	Suppose $f$ is a modular form of weight $k$ and level $M_0 p^r$ which is new
	at $p^r$, and $g_0,h_0$ are any modular forms of weight $m-k$ and $m$,
	respectively, which are modular for $M_0 p^{r-1}$. Then
	$\g{f(z)g_0(z),h_0(z)} = 0$.
\end{lem}
We remark that it's straightforward to check (using trace operators as in the
proof here) that condition of ``being new at $p^r$'' does not depend on the
prime-to-$p$ level. Thus it doesn't matter if the $M_0$ in the statement of the
lemma is larger than the original prime-to-$p$ level of $f$.
\begin{proof}
	Similar to what we discussed in Section \ref{sec:PeterssonOfHigherLevels},
	we can construct a trace operator from modular forms of level $M_0 p^r$ to
	$M_0 p^{r-1}$, taking a form $\ps$ to the sum of all translates $\ch\1(\ga)
	\ps|[\ga]$ for $\ga$ running over coset representatives of $\Ga_0(M_0 p^r)
	\q \Ga_0(M_0 p^{r-1})$. Furthermore, we have that if $\ps$ is modular of
	level $M_0 p^{r-1}$ then $\g{\trc(f),\ps}$ is a scalar multiple of
	$\g{f,\ps}$. 
	\par
	Since $f$ being new at $p^r$ means it's perpendicular to every $\ps$ of
	level $M_0 p^{r-1}$, we can thus conclude $\g{\trc(f),\ps} = 0$ for every
	such $\ps$. Since $\trc(f)$ lies in this space of modular forms $M_0
	p^{r-1}$ we conclude that $\trc(f) = 0$.
	\par
	Thus we obviously have $0 = \g{\trc(f)g_0,h_0}$. But we can also carry out a
	similar manipulation for this Petersson inner product: 
	\[ \g{\trc(f)g_0,h_0} = \s_\ga \g{\ch_f\1(\ga) f|[\ga]_k g_0, h_0} 
			= \s_\ga \g{\ch_f\1(\ga) f g_0|[\ga\1]_{m-k}, h_0|[\ga\1]_m}, \]
	and using the transformation laws for $\ga \in \Ga_0(M_0 p^{r-1})$ on
	$h_0,g_0$ we conclude that $\g{\trc(f)g_0,h_0}$ is a scalar multiple of
	$\g{fg_0,h_0}$, and thus that $\g{fg_0,h_0} = 0$.
\end{proof}

This lemma directly implies $\g{fg,h} = \g{f(z)g(z),h(p^{r-1}z)}$ vanishes. 
A further argument allows it to handle $\g{fg_p,h_p}$ as well. 

\begin{lem}
	We have $\g{fg_p,h_p} = 0$. 
\end{lem}
\begin{proof}
	By Lemma \ref{lem:AdjointToTpOperator}, we have $\g{f(z)g(pz),h(p^r z)} =
	p^m \g{T(p)(f(z)g(pz)),h(p^{r-1}z)}$. If we write if $f = \s a_n q^n$ and $g
	= \s b_n q^n$ then we have 
	\[ T(p)(f(z)g(pz))
		= T(p)\b( \s_{n=0}^\oo \8( \s_{i+pj = n} a_i b_j \8) q^n \e) 
		= \s_{n=0}^\oo \8( \s_{i+pj = pn} a_i b_j \8) q^n
		= \s_{n=0}^\oo \8( \s_{i'+j = n} a_{pi'} b_j \8) q^n, \]
	which is $T(p)(f(z)) \dt g(z) = a_f f(z) g(z)$. Thus $\g{fg_p,h_p} = a_f p^k
	\g{fg,h} = 0$. 
\end{proof}

At this point we've reduced our equation to 
\[ \g{fg^\sh,h^\na} = -\be_g \g{fg_p,h} - p\Lbe_h \g{fg,h_p}. \]
If $r > 1$ then we can in fact also apply Lemma
\ref{lem:PairNewformWithLowerLevel} to $\g{fg_p,h_p}$ and conclude: 

\begin{prop} \label{prop:EulerFactorNumeratorPPartGe2}
	Suppose that $f$ is new of level $p^r$ with $r\ge 2$. Then we have 
	\[ \g{f(z) g^\sh(z),h^\na(p^{r-1} z)} = - p\Lbe_h \g{f(z)g(z),h(p^r z)}.\]
\end{prop}

If $r = 1$ then $\g{fg_p,h}$ will generally not be zero. Instead we can relate
it to $\g{fg,h_p}$ by using the Atkin-Lehner operator $\om_p$ as in  Lemma
\ref{lem:AtkinLehnerToPetersson}. 

\begin{lem} 
	In the case $r=1$, we have $\g{f_p g,h} = -p a_f \Lch_h(p) \g{fg_p,h_p}$. 
\end{lem}
\begin{proof}
	As in Lemma \ref{lem:AtkinLehnerToPetersson} we may write $\g{fg_p,h} =
	\g{f|[\om_p]_kg_p|[\om_p]_{m-k}, h|[\om_p]_m}$, and the arguments of that
	lemma give $g_p|[\om_p]_{m-k} = \Lch_g(p) p^{-(m-k)/2} g$ and $h|[\om_p]_m =
	p^{m/2} h_p$. Meanwhile, one can compute $f|[\om_p]_k = -p^{1-k/2} \L{a}_p f$
	arguing with Atkin-Lehner operators as in \cite{Asai1976} or Section 4.6 of
	\cite{MiyakeModForms}; in particular, the arguments of Theorem 1 and Lemma 3
	of \cite{Asai1976} go through in our case (specifically because $f$ is an
	eigenform of $T(p)$ with eigenvalue $a_f$, $p$ exactly divides the level of
	$f$, and the character $\ch_f$ has trivial $p$-part due to the equation
	$\ch_f\ch_g = \ch_h$). Combining these equations and using $\L{a}_p =
	\Lch_f(p) a_f$ and $\Lch_f(p)\Lch_g(p) = \Lch_h(p)$ gives the lemma.
\end{proof}

Substituting, we have
\[ \g{fg^\sh,h^\na} = ( p a_f \be_g \Lch_h(p) - p\Lbe_h ) \g{fg,h_p}. \]
Using the identity $\Lch_h(p) = \Lbe_h / \al_h$ gives: 

\begin{prop} \label{prop:EulerFactorNumeratorPPart1}
	Suppose that $f$ is new of level $p^r$ with $r = 1$. Then we have 
	\[ \g{f(z) g^\sh(z),h^\na(p^{r-1} z)} = - p\Lbe_h 
			\b( 1 - \f{\be_g a_f}{\al_h} \e) \g{f(z)g(z),h(p^r z)}.\]
\end{prop}

We remark that we have carried out numerical computations to verify the
formulas obtained in in Propositions \ref{prop:EulerFactorDenominator},
\ref{prop:EulerFactorNumeratorPrimeToP},
\ref{prop:EulerFactorNumeratorPPartGe2}, and
\ref{prop:EulerFactorNumeratorPPart1}. This is described in Section 3.4 of our
paper \cite{Collins2018}. 

\subsection{The module of congruences for a CM family}
\label{sec:CongruenceModuleCM}

Finally, the last bit of $\La$-adic theory we'll need to use is an explicit
computation of the congruence number $H_\Bg$ in the case $\Bg = \Bg_\Ps$ is a
$\CI$-adic CM newform as constructed in the previous section. As we've
mentioned, this computation is essentially a consequence of the main conjecture
of Iwasawa theory for imaginary quadratic fields, which was proven by Rubin
\cite{Rubin1991}. The result is that $H_\Bg$ is essentially a $p$-adic
$L$-function associated to an anticyclotomic family of Hecke characters closely
related to $\Ps$; the connection between the main conjecture and the
computation of $H_\Bg$ is studied by Hida and Tilouine in
\cite{HidaTilouine1993} and \cite{HidaTilouine1994}. 
\par
A number of constructions have been given for these \E{$p$-adic Hecke
$L$-functions}, for instance by Manin-Vishik \cite{VisikManin1974}, by Katz
\cite{Katz1976} and \cite{Katz1978}, and by Coates-Wiles
\cite{CoatesWiles1978}; see also expositions in Chapter II.4 of
\cite{ShalitCMEllipticCurves} and Chapter 9 of \cite{HidaArithInvariants}.
We'll give the necessary setup to state the special case of these results that
we'll actually use.  
\par
The main thing we need to do is to define the appropriate periods for
algebraicity results. We actually define a pair $(\Om_\oo,\Om_p)$ of a complex
and a $p$-adic period. The idea is to fix an auxiliary prime-to-$p$ conductor
$c$, and take $A_0$ to be the elliptic curve defined by $A_0(\C) = \C/\CO_c$,
which we can show is defined over a number field. We then pick a
nowhere-vanishing differential $\om$ defined over this number field. Then,
working over $\C$ we have a standard differential $\om_\oo$ coming from the
standard coordinate on $\C$ passed to $\C/\CO_c$. Working over $\CO_{\C_p}$ we
can get an isomorphism between $\H\DG_m$ and the formal completion $\H\CA_0$ of
a good integral model $\CA_0$ of $A_0$, well-defined up to $p$-adic units; from
this we get a differential $\om_p$ induced by the standard period $dt/t$ on
$\H\DG_m$. We can then define our periods: 

\begin{defn} \label{defn:HeckePeriods} 
	Given a fixed differential $\om$ defined over $\CO_c$ as above, we define
	two periods $\Om_\oo \in \C$ and $\Om_p \in \C_p$ by 
	\[ \om = \Om_\oo \om_\oo \qq\qq \om = \Om_p \om_p. \]
\end{defn}

The pair $(\Om_\oo,\Om_p)$ is evidently well-defined up to algebraic scalars.
We usually take the convention that we choose our scalar so that $\Om_p$ is a
$p$-adic unit; thus $(\Om_\oo,\Om_p)$ is well-defined up to $p$-adic units of
$\L\Q$ (via the embedding $\L\Q \into \L\Q_p$ that we've fixed since the
beginning). Also, we note that while we started with a conductor $c$ (and
defined our elliptic curve $A_0$ as the one having CM by the order $\CO_c$ of
$\CO_K$), the dependence on $c$ is entirely illusory: since there's a
prime-to-$p$ isogeny between the resulting curves for any two such choices of
$c$, we can choose the same pair of periods $(\Om_\oo,\Om_p)$ to arise from
both cases.
\par
With our periods constructed, we can state the interpolation property that
characterizes the $p$-adic Hecke $L$-function, and quote the existence result of
this $L$-function. In full generality, it can be defined to
interpolate special values of Hecke $L$-functions where we can vary the
character over all things in a certain range of conductors; we will be
interested in a specialization that gives us something defined on a domain
similar to $\La$, which interpolates the special values in a particular
anticyclotomic family of characters. In particular we'll consider a family $\Ph
: \DI_K \to \CI$ constructed along the lines of what we did in Section \ref{sec:FamiliesOfHeckeCharacters}, in particular satisfying $P_m\o\Ph =
\ph_m = \ph \al^{m-1} (\al^c)^{-(m-1)}$. More specifically, we'll ask that the
finite-order character $\ph$ here has prime-to-$p$ conductor $\Fc$, and that
the $p$-part of its finite-type is $\om_\Fp^a \om_{\L\Fp}^{-a}$ for a residue
class $a$ mod $p-1$. Thus $\ph_m$ will be unramified at $p$ (and have conductor
exactly $\Fc$) for $m \ee a\md{p-1}$. 
\par
So, the interpolation property we want is at points $P_m$ for $m\ee a\md{p-1}$.
There is a slight additional wrinkle: even though our family $\Ph$ of Hecke
characters are defined over $\CI$, the $p$-adic $L$-function will need to be
defined over a larger ring; to deal with the periods we need to extend scalars
to $\CO_F^\uR$, the integer ring of the maximal unramified extension $F^\uR/F$.
So we define $\CI^\uR = \CI\ox_{\CO_F} \CO_F^\uR = \CO_F^\uR\bb{\Ga'}$. On this
larger ring, our distinguished specializations $P_m : \CI \to \CO_F$ have a
unique $\CO_F^\uR$-linear extension to $\CI^\uR \to \CO_F^\uR$ that we also
denote $P_m$, characterized in the usual way (by mapping the distinguished
topological generator $\ga_\pi$ to $(1+\pi)^m$). Then: 

\begin{thm} \label{thm:PAdicHeckeLFunction}
	Suppose $\Ph = \ph \CA \al\1 (\CA^c)\1 \al^c$ is an anticyclotomic family as
	described above, so that for $m\ee a\md{p-1}$ we have that $\ph_m$ has
	infinity-type $(m-1,-m+1)$ and conductor $\Fc$ coprime to $p$. Then, for
	$m\ge 3$ satisfying $m\ee a\md{p-1}$ we can define $L_\alg(\ph_m,1) \in
	\L\Q$ and $L_p(\ph_m,1) \in \C_p$ satisfying
	\[ \f{L_p(\ph_m,1)}
			{\Om_p^{2m-2} (1-\ph_m\1(\L\Fp)/p) (1-\ph_m(\Fp))} 
			= L_\alg(\ph_m,1) = \f{(m-1)!\pi^{m-2}}{2^{m-2} \rt{d}^{m-2}} 
					\f{L(\ph_m,1)}{\Om_\oo^{2m-2}}, \]
	(with these equalities really being under our embeddings $\io_\oo$ and
	$\io_p$) and there exists an element $\CL_p(\Ph) \in \CI^{ur}$ such that
	specializing at $P_m$ for such $m$ gives $L_p(\ph_m,1)$. 
\end{thm}

The theorem holds (with the same interpolation formula) for families of the
form $\ph \CA^c (\al^c)\1 \CA\1 \al$ as well. 
\par
Next, we connect this back to the problem of finding the congruence number 
associated to a $\CI$-adic family $\Bg_\Ps$ where $\Ps$ is a family with
$P_m\o\Ps = \ps_m = \ps \al^{m-1}$ of infinity-type $(m-1,0)$ (where here
$\al$ has infinity-type $(1,0)$ and finite-type $\om_\Fp\1$). As we've said,
the Iwasawa main conjecture for the imaginary quadratic field $K$ is enough to
compute this. By a general argument in deformation theory (see Section 3.2 of
\cite{Hida1996a}, for instance), the congruence module $\CC(\Bf,\CI)$ can be
identified with (the Pontryagin dual of) a Selmer group for the adjoint Galois
representation $\ad(\Bf)$. In the case where $\Bf = \Bg_\Ps$ is a CM family,
this adjoint Galois representation decomposes as a direct sum of
$\Ind(\Ps/\Ps^c)$ and $\ch_K$, and thus the Selmer group breaks up similarly,
and finally the Selmer group of $\Ind(\Ps/\Ps^c)$ is related to
$\CL_p(\Ps/\Ps^c)$ by the main conjecture. 
\par
The precise relation between $H_{\Bg_\Ps}$ to $\CL_p(\Ps/\Ps^c)$ is worked out
\cite{HidaTilouine1993} and \cite{HidaTilouine1994}, and we quote a version of
their result. 

\begin{thm} \label{thm:HeckeLFunctionIsCongruenceNumber}
	Suppose that our family is such that for $m\ee a\md{p-1}$, the character
	$\ps_m$ has conductor $\Fc$ which is a nonempty product of inert primes.
	Then the CM forms $g_{\ps_m}$ have residually irreducible $p$-adic Galois
	representations, and the element 
	\[ \b( \p_{q|N(\Fc)}(1+q\1)(1-q^{-2})\e) h_K \CL_p(\Ps/\Ps^c) \in \CI^{ur}\]
	is equal (up to a unit in $\CI^{ur}$) to the congruence number $H_{\Bg_\Ps}
	\in \CI$. Here $h_K$ is the class number of $K$.
\end{thm}

Similarly, the congruence number $H_{\Bg_{\Ps^\rh}}$ (for the conjugate
$\La$-adic family of Proposition \ref{prop:FamilyOfConjugateCMNewforms}) is 
\[ \b( \p_{q|N(\Fc)}(1+q\1)(1-q^{-2})\e) h_K \CL_p(\Ps^c/\Ps). \]
\par
We remark that the hypothesis about residually irreducible Galois
representations is automatically satisfied if $\Ps$ is ramified at any inert
prime, since then even the induced local representation at that prime is
residually irreducible (since it corresponds to a supercuspidal). Also, along
this chain of reasoning we implicitly use the fact that the local ring of the
Hecke algebra through which $\la_{\Bg_\Ps}$ factors is Gorenstein (as needed in
Theorem \ref{thm:ExistenceOfCongruenceNumber} to conclude the existence of
$H_{\Bg_\Ps}$ in the first place); this Gorenstein condition follows from
modularity lifting work, e.g. from \cite{Wiles1995}, Corollary 2 to Theorem
2.1.
\par
Finally, it is useful to link the values $h_K L_\alg(\ps_m/\ps_m^c,1)$ that
arises from this congruence number computation back to $L(\ad(g_{\ps_m}),1)$
and then to the related Petersson inner product $\g{g_{\ps_m},g_{\ps_m}}$ - we
can roughly think about it as the ``algebraic part'' of these quantities. The
relation between the adjoint $L$-value and the Petersson inner product is
stated in Theorem \ref{thm:AdjointVsPetersson}. We then want to specialize to
the case of $f = g_\ps$ and relate $L_N(\ad g_\ps,1)$ to $L(\ps/\ps^c,1)$, the
$L$-value we're interpolating. On the level of automorphic $L$-functions we
have an identity 
\[ L(\ad g_\ps,1) = L(\ps/\ps^c,1) L(\ch_K,1) \]
coming from a corresponding decomposition of Weil-Deligne groups. To get an
identity involving $L_N(\ad g_\ps,1)$ we need to work out the Euler factors of
$L(\ps/\ps^c,1) L(\ch_K,1)$ at places dividing $N$. For our purposes we have: 

\begin{prop}
	Let $K = \Q(\rt{-d})$ be an imaginary quadratic field (with $-d$ an odd
	fundamental discriminant) and $\ps$ a Hecke character for $K$ which has
	conductor $\Fc$ divisible only by inert primes of $K$. Then we have 
	\[ L(\ps/\ps^c,1) L(\ch_K,1)  = L_N(\ad g_\ps,1) 
			\p_{q|d} (1-q\1)\1 \p_{q|N(\Fc)} (1 - q^{-2})\1 (1+q\1)\1 \] 
\end{prop}

Combining this, Theorem \ref{thm:AdjointVsPetersson}, and the analytic class
number formula $h_K = (w_K\rt{d}/2\pi) L(\ch_K,1)$ lets us conclude

\begin{cor} \label{cor:AlgebraicPartCongruenceNumber}
	Let $K = \Q(\rt{-d})$ be an imaginary quadratic field (with $-d$ an odd
	fundamental discriminant) and $\ps$ a Hecke character of infinity-type
	$(m-1,0)$ for $K$ which has conductor $\Fc$ divisible only by inert primes
	of $K$. Then we have 
	\[ \b( \p_{q|N(\Fc)} (1 + q\1) (1-q^{-2}) \e) h_K L_\alg(\ps/\ps^c,1) 
		= \f{2^m \pi^{2m-1} w_K}{3 \rt{d}^{m-3}} 
			\f{\g{g_\ps,g_\ps}}{ \Om_\oo^{2m-2} } \p_{q|dN(\Fc)} (1+q\1). \]
\end{cor}

\subsection{Summarizing our results}
\label{sec:SummarizingPeterssonInterpolation}

We now restate the constructions of this section in a way that mirrors the
way the BDP $p$-adic $L$-functions is built up (see Section
\ref{sec:BDPLfunction}): passing from a Petersson inner product $\g{fg,h}$ to
an algebraic value $\g{fg,h}_\alg$ and then to a $p$-adic value $\g{fh,h}_p$,
and realizing these $p$-adic values as specializations of a $p$-adic function
$\g{f\Bg,\Bh} \in \CI^\uR$. We work with the following setup: 
\begin{itemize}
	\item $f$ is a fixed newform of weight $k$, level $N_f p^{r_0}$ and
		character $\ch_f$, with $N_f$ prime to $p$. 
	\item $\Bg$ is a $\CI$-adic family of newforms, such that for a particular
		residue class $m-k \ee a\md{p-1}$ we have $g_{m-k}$ is of prime-to-$p$
		level $N_g$ and character $\ch_g$.
	\item $\Bh$ is a $\CI$-adic family of newforms, such that for $m \ee
		a\md{p-1}$ (the same $a$ as above) we have $h_m$ is of prime-to-$p$ level
		$N_g$ and character $\ch_h = \ch_f \ch_g$.
	\item We want to interpolate the Petersson inner products of the form
		\[ \g{f|_{M_f}\dt g_{m-k}|_{M_g}, h_m|_{M_h p^{r_0}}}  
				= \g{f(M_f z) \dt g_{m-k}(M_g z), h_m(M_h p^{r_0} z)} \]
		for $m \ee a\md{p-1}$, for some prime-to-$p$ integers $M_f,M_g,M_h$. 
\end{itemize}

For convenience we'll sometimes suppress the constants $M_f,M_g,M_h p^{r_0}$ in
the above notation and simply write $\g{fg_{m-k},h_m}$ for the inner product
there. This inner product can be computed on $\Ga_0(N_{fgh} p^r)$, where 
\[ N_{fgh} = \lcm(N_f M_f,N_g M_g,N_h M_g) \]
is the common prime-to-$p$ level of the three modular forms involved and $r =
\max\{ r,1 \}$. 
\par
To define $\g{fg_{m-k},h_m}_\alg$, we want to use $a(1,1_{h_m}
fg_{m-k})$, which we know is an algebraic number that's equal to the ratio of
$\g{fg_{m-k},h_m}$ by $\g{h_m,h_m} = \g{h_m(z),h_m(z)}$; this will be an
algebraic number corresponding to the orthogonal projection of $fg_{m-k}$ onto
$h_m$ in the space of cusp forms. We also need to multiply this by the
algebraic part of the congruence number $H_{g_m}$. Since we've only discussed
congruence numbers for CM forms, we restrict to that situation: 
\begin{itemize}
	\item The $\CI$-adic form $\Bh$ is actually a CM form $\Bg_\Ps$, where $\Ps$
		is a family of Hecke characters for $K = \Q(\rt{-d})$ such that $\ps_m$
		has prime-to-$p$ conductor $\Bc$ (so $N_h = d N(\Fc)$), and also such
		that the Galois representations associated to $g_{\ps_m}$ are residually
		irreducible.
\end{itemize}
This covers both the ``usual'' families $\Bg_\Ps$ of Section
\ref{sec:FamiliesOfCMForms} and the ``conjugate'' families $\Bg_{\Ps^\rh}$ of
Section \ref{sec:ComplexConjugatesHeckeChars}. Of course, if one wants to have
some other $\CI$-adic newform $\Bh$ and works out the congruence ideal $H_\Bh$,
it's easy to modify what follows for that situation.
\par
We then define 
\begin{align*}
	\g{fg_{m-k},h_m}_\alg &= 
		\f{ \g{f|_{M_f}\dt g_{m-k}|_{M_g}, h_m|_{M_h p^{r_0}}} }{\g{h_m,h_m}} \\
		&\qq\qq\dt h_K L_\alg\6(\ps_m/\ps_m^c,1\6) 
			\b( \p_{q|N(\Fc)} (1 + q\1) (1-q^{-2}) \e) \\
		&\qq\qq\dt \f{p^{m(r-1)}}{\al_{h_m}^{r-1}} \f{N_{fgh} M_h^{m+1}}{N_h} 
			\b( \p_{q | N_{fgh}, q\nd N_h} (1+q\1) \e)  
\end{align*}
and 
\begin{align*}
	\g{fg_{m-k},h_m}_p &= 
		\f{ \g{f|_{M_f}\dt g^\sh_{m-k}|_{M_g}, h^\na|_{M_h p^{r_0-1}}} } 
			{\g{h^\sh_m,h^\na_m}} \\ 
		&\qq\qq\dt h_K L_p\6(\ps_m/\ps_m^c,1\6) 
			\b( \p_{q|N(\Fc)} (1 + q\1) (1-q^{-2}) \e) \\
		&\qq\qq\dt \f{p^{m(r-1)}}{\al_{h_m}^{r-1}} \f{N_{fgh} M_h^{m+1}}{N_h} 
			\b( \p_{q | N_{fgh}, q\nd N_h} (1+q\1) \e)  
\end{align*}
This is set up so that if we define 
\[ \g{f\Bg,\Bh} = \l_\Bh(\trc(f|_{M_f/M_h} \dt \Bg|_{M_g/M_h})) \in \CI^\uR \]
(using the shifting and scaling constructions of Section
\ref{sec:ShiftsOfIAdicForms} and the trace operator $\trc : \BS(\Ga_0(N_{fgh}
p^r,M_h),\ch;\CI) \to \BS(N_h p^r,\ch;\CI)$ defined in Section
\ref{sec:PeterssonOfHigherLevels}, and taking the congruence number $H_\Bh$ to
actually be the element of $\CI^\uR$ written in Theorem
\ref{thm:HeckeLFunctionIsCongruenceNumber}), then we have $P_m(\g{f\Bg,\Bh}) =
\g{fg_{m-k},h_m}_p$. We can see this by applying Proposition
\ref{prop:InterpolatePeterssonForHigherLevel} and unravelling the definitions;
we get 
\begin{align*}
	P_m\7( \l_\Bh(\trc(f|_{M_f/M_h} \dt \Bg|_{M_f/M_h})) \7) &= 
		\f{\g{f|_{M_f/M_h} \dt g_{m-k}^\sh|_{M_g/M_h}, h_m^\na|_{p^{r-1}}}}
				{\g{h_m^\sh,h_m^\na}} \dt H_{\Bh,P_m} \\ 
		& \qq \dt \f{p^{m(r-1)}}{\al_{h_m}^{r-1}} \f{N_{fgh} M_h}{N_h} 
			\p_{q | N_{fgh}, q\nd N_h} (1+q\1)
\end{align*}
Making the change of variables $z \mt M_h z$ (by Lemma
\ref{lem:ScalingPetersson}) we further see that this is equal to 
\[ \f{\g{f|_{M_f} \dt g_{m-k}^\sh|_{M_g}, h_m^\na|_{M_h p^{r-1}}}}
				{\g{h_m^\sh,h_m^\na}} \dt H_{\Bh,P_m} 
		\dt \f{p^{m(r-1)}}{\al_{h_m}^{r-1}} \f{N_{fgh} M_h^{m+1} }{N_h} 
			\p_{q | N_{fgh}, q\nd N_h} (1+q\1). \]
Since $H_\Bh$ is realized by $h_k \CL_p(\Ps(\Ps^c)\1)$ by Theorem
\ref{thm:HeckeLFunctionIsCongruenceNumber}, we conclude that this is exactly
the value of $\g{fg_{m-k},h_m}_p$: 

\begin{thm} \label{thm:PeterssonCMInterpolate}
	Suppose that we have $f$, $\Bg$, and $\Bh = \Bg_\Ps$ satisfying the
	conditions listed above for a residue class $a$. Then the element
	$\g{f\Bg,\Bh} \in \CI$ defined above satisfies $P_m(\g{f\Bg,\Bh}) =
	\g{fg_{m-k},h_m}_p$ for all $m\ee a\md{p-1}$. 
\end{thm}

To properly \E{use} this theorem, we'll want formulas more directly relating
our initial complex-analytic value $\g{fg_{m-k},h_m}$ to the algebraic value
$\g{fg_{m-k},h_m}_\alg$ and then the $p$-adic value $\g{fg_{m-k},h_m}_p$. We 
do this as follows. (Remember that our notation suppresses the shifts
$|_{M_f}$ and so on). 

\begin{prop} \label{prop:PeterssonAlgebraicPAdicParts} 
	Suppose that we have $f$, $\Bg$, and $\Bh = \Bg_\Ps$ satisfying the
	conditions listed above for a residue class $a$. Then we have 
	\[ \g{fg_{m-k},h_m}_\alg = 
			\f{\g{fg_{m-k},h_m}}{\Om_\oo^{2m-2}} \dt 
			\f{2^m w_K \pi^{2m-1}}{3 \rt{d}^{m-3}} \f{p^{m(r-1)}}{\al_{h_m}^{r-1}}
			\f{N_{fgh} M_h^{m+1}}{N_h} \p_{q | N_{fgh}} (1+q\1) \]
	and also 
	\[ \g{fg_{m-k},h_m}_p = e_p(fg_{m-k},h_m)\Om_p^{2m-2}\g{fg_{m-k},h_m}_\alg\] 
	where $e_p(fg_{m-k},h_m)$ is a removed ``Euler factor'' given by 
	\begin{itemize}
		\item $e_p(fg_{m-k},h_m) =  (1-\f{\be_g\al_f}{\al_h})
			(1-\f{\be_g\be_f}{\al_h})$ if $r_0 = 0$,
		\item $e_p(fg_{m-k},h_m) = p\Lal_h (1+p\1) 
			(1-\f{\be_g a_f}{\al_h})$ if $r_0 = 1$,
		\item $e_p(fg_{m-k},h_m) = p\Lal_h (1+p\1)$ if $r_0 \ge 2$.
	\end{itemize}
	Here $r_0$ is the exact power of $p$ dividing the level of $f$. 
\end{prop}
\begin{proof}
	The equation for $\g{fg_{m-k},h_m}_\alg$ follows from taking the definition
	and plugging in Corollary \ref{cor:AlgebraicPartCongruenceNumber} and 
	noting that $N_h = d N(\Fc)$ so the $\p_{q | N_{fgh}, q\nd N_h}$ and $\p_{q
	| dN(\Fc)}$ combine to a single product.
	\par
	The equation for $\g{fg_{m-k},h_m}_p$ follows from comparing the definitions
	and noting we have 
	\[ \g{fg_{m-k},h_m}_p 
			= \f{\g{fg^\sh,h^\na}}{\g{fg,h}} \f{\g{h,h}}{\g{h^\sh,h^\na}} 
				\f{L_p(\ps_m/\ps_m^c,1)}{L_\alg(\ps_m/\ps_m^c,1)} 
					\g{fg_{m-k},h_m}_p. \] 
	The ratio $L_p/L_\alg$ is 
	\[ \Om_p^{2m-2} \b(1 - \f{\ps_m^c}{\ps_m}(\L\Fp)p\1\e)
			\b(1 - \f{\ps_m}{\ps_m^c}(\Fp)\e)
			= \Om_p^{2m-2} \b(1 - \f{\be_h}{\al_h}p\1\e)
				\b(1 - \f{\be_h}{\al_h}\e) \]
	by definition, which multiplies together with the ratio 
	\[ \f{\g{h,h}}{\g{h^\sh,h^\na}} = 
		\f{(1+p\1)}{(-\al_h/\be_h)(1 - \be_h/\al_h)(1 - p\1 \be_h/\al_h)} \]
	given by Proposition \ref{prop:EulerFactorDenominator} to give $\Om_p^{2m-2}
	(1+p\1)$. The specified equations for $e_p(fg_{m-k},h_m)$ then follow from
	the computations of $\g{fg^\sh,h^\na}/\g{fg,h}$ in Propositions
	\ref{prop:EulerFactorNumeratorPrimeToP}, 
	\ref{prop:EulerFactorNumeratorPPart1}, and
	\ref{prop:EulerFactorNumeratorPPartGe2}, respectively.
\end{proof}
 
\section{Constructing the $p$-adic $L$-function} 

We now proceed with the main result of this paper, the construction of the
anticyclotomic $p$-adic $L$-function $\CL_p(f,\Xi\1)$, by combining our
explicit version of Ichino's formula from Theorem \ref{thm:ExplicitIchinoCM}
with the $\La$-adic modular form theory we've developed in the previous
sections.

\subsection{The $L$-function of Bertolini-Darmon-Prasanna} 
\label{sec:BDPLfunction}

We begin in this section by recalling previous work on the $p$-adic
$L$-functions we're interested in studying. Given a classical modular form $f$
and a Hecke character $\xi$ associated to an imaginary quadratic field $K$, $f$
determines an automorphic representation $\pi_f$ for $\GL_2/\Q$ and $\xi$
determines an automorphic representation for $\GL_1/K$ which can be induced to
$\pi_\xi$ for $\GL_2/\Q$. We are then interested in studying the central value
$L(\pi_f\x\pi_\xi,1/2)$ of the $\GL_2\x\GL_2$ $L$-function corresponding to
this pair.
\par
Hida \cite{Hida1988b} constructed $p$-adic $L$-functions for $\GL_2\x\GL_2$ in
a great deal of generality. However, the starting point for our work is the
more recent paper of Bertolini, Darmon, and Prasanna
\cite{BertoliniDarmonPrasanna2013} which gave a different construction of a
particular such $p$-adic $L$-function (with $f$ fixed and $\xi$ varying
``anticyclotomically'', in a way we'll make precise shortly). From the formula
they used in their construction, they were able to obtain special value
formulas connecting their $p$-adic $L$-function to certain algebraic cycles. 
\par
To set up the construction of the BDP $L$-function precisely, let us fix a
classical newform $f = \s a_n q^n \in S_k(N,\ch_f)$, along with our imaginary
quadratic field $K = \Q(\rt{-d})$ with fundamental discriminant $-d$. We then
look at the classical Rankin-Selberg $L$-function associated to $f$ and the
theta function of $\xi$; at good primes $q$ the Euler factor is 
\[ L_{(q)}(f,\xi,s)\1 
	= \p_{i,j=1}^2 (1 - \al_{f,i}(q) \al_{\xi,i}(q) q^{-s}) \] 
where the coefficients $\al$ are the roots of the appropriate Hecke
polynomials: 
\[ L_{(q)}(f,s)\1 = (1-a_q q^{-s} + \ch_f(q) q^{k-1} q^{-2s}) 
		= {\ts\p_{i=1}^2} (1 - \al_{f,i}(q) q^{-s}), \]
\[ L_{(q)}(\xi,s)\1 = {\ts\p_{\Fq|q}} (1 - \xi(\Fq) N(\Fq)^{-s})  
		= {\ts\p_{i=1}^2} (1 - \al_{\xi,i}(q) q^{-s}). \]
(We use the notation $L_{(q)}$ to denote Euler factors, since $L_p$ is already
our notation for $p$-adic $L$-functions). At primes where $f$, $\xi$, or $K$
has ramification the correct local factor comes from the local representation
theory and is more difficult to write down directly from the coefficients of
$f$ and the values of $\xi$; such local factors are worked out case-by-case in
\cite{JacquetAutForms}. 
\par
The paper \cite{BertoliniDarmonPrasanna2013} studies the special value
$L(f,\xi\1,0)$ for certain characters $\xi$. They study the set
$\Si_{cc}^{(2)}$ of all ``central critical characters'' of infinity-type
$(j+k,-j)$ for $j \ge k$, and then focus on a subset $\Si_{cc}^{(2)}(\FN,c)$ of
characters with a certain prescribed ramification. For characters $\xi$ in this
set they suitably define $L_\alg^{BDP}(f,\xi\1,0) \in \L\Q$ and then
$L_p^{BDP}(f,\xi\1) \in \C_p$ from the resulting $L$-value, and prove the
following interpolation result: 

\begin{thm}
	Given the data above (under some hypotheses), the map $\Si_{cc}^{(2)}(\FN,c)
	\to \C_p$ defined by $\xi \mt L_p^{BDP}(f,\xi\1)$ is uniformly continuous
	(relative to a certain function space topology on the domain) and thus
	extends to the completion of $\Si_{cc}^{(2)}(\FN)$ (which includes other
	functions such as the characters in $\Si_{cc}^{(1)}(\FN)$).
\end{thm}

We would now like to rephrase the results to view $L_p^{BDP}(f,\xi)$ as an
analytic function of the parameter $\xi$, rather than just a continuous one. Of
course, doing this requires putting some sort of analytic structure on the
domain $\Si_{cc}^{(2)}(\FN)$. The correct general framework for formalizing
this is rigid analytic geometry. 
\par
For our purposes, we can get away working entirely algebraically, by focusing
on the formal power series ring $\La = \CO_F\bb{X}$ (where $\CO_F$ is the ring
of integers of a finite extension $F/\Q_p$). In rigid geometry, $\La$ is the
coordinate ring of the open unit disc, and the points of the disc correspond to
homomorphisms $P : \La \to \CO_F$. So our goal will be to construct our
$p$-adic $L$-function $\CL$ as an element of $\La$, which is a ``function'' in
the sense that it associates a value $P(\CL)$ to each point $P$. For technical
reasons, we'll actually need to work with rings $\CI$ that are extensions of
$\La$ - so geometrically, with certain covers of the open disk - and then
extend our scalars to the maximal unramified extension $F^\uR/F$. This will
leave us with the ring $\CI^\uR = \CO_{F^\uR}\ox_{\CO_F}\CI$.
\par
To formulate the interpolation properties used to characterize the $p$-adic
$L$-function, for each integer $m$ we define a distinguished point $P_m : \La
\to \CO_F$ by $P_m(X) = (1+p)^m - 1$ (and lift these to points $P_m : \CI \to
\CO_F$, as discussed in Section \ref{sec:CIadicModularForms}, and then further
extend linearly to $P_m : \CI^\uR \to \CO_{F^\uR}$). We then want our element
$\CL$ to satisfy $P_m(\CL) = L_p^{BDP}(f,\xi_m\1)$ for $m$ in a certain
arithmetic progression, where $\xi_m$ ``varies $\La$-adically''. We note that
specifying values of $\CL$ on any infinite set of points is enough to specify
it uniquely by the following lemma.

\begin{lem} \label{lem:EqualityOfSpecializations}
	If $A,B\in\CI^\uR$ are two elements such that $P(A) = P(B)$ for infinitely
	many points $P \in \CX(\CI^\uR;\CO_F^\uR) = \Hom_\cont(\CI^\uR,\CO_F^\uR)$,
	then $A = B$.
\end{lem}
\begin{proof}
	In our case this will follow from the Weierstrass Preparation Theorem
	because the rings $\CI^\uR$ we use will be abstractly isomorphic to
	$\CO_F^\uR\bb{X}$. Alternatively it can be proven by commutative algebra
	using that $\CI^\uR$ is Noetherian of dimension 2 and that the points $P$
	correspond to prime ideals of height 1. 
\end{proof}

To formulate the idea of $\La$-adically varying the characters $\xi_m$, we
define a $\La$-adic family of Hecke characters to be a continuous homomorphism
$\Xi : \DI_K/K^\x \to \La^\x$ such that $\xi_m = P_m\o\Xi$ has weight
$(m-1,-m+k+1)$; in Section \ref{sec:FamiliesOfHeckeCharacters} we'll construct
such families, and see that any $\xi \in \Si_{cc}(\FN)$ fits into a family
$\Xi$ with infinitely many specializations also lying in $\Si_{cc}(\FN)$. With
this formalism, we can reinterpret the BDP $p$-adic $L$-function as follows. 

\begin{thm} \label{thm:BDPLambda}
	Fix a Hecke character $\xi_m \in \Si_{cc}(\FN,c)$ of infinity-type
	$(a-1,-a+k+1)$, and put it into a family $\Xi : \DI_K/K^\x \to \CI^\x$. Then
	there exists an element $\CL^{BDP}_p(f,\Xi\1) \in \CI \ox \CO_F^\uR$ that
	satisfies 
	\[ P_m(\CL^{BDP}_p(f,\Xi\1)) = L_p^{BDP}(f,\xi_m\1) \]
	for all $m\ge k$ satisfying $m\ee a\md{p-1}$.
\end{thm}

The result in \cite{BertoliniDarmonPrasanna2013} is actually stated to produce
$\CL^{BDP}$ as a continuous function on a space of characters, but the formula
for $L_p(f,\ch)$ in Theorem 5.9 of their paper (which they use to prove uniform
continuity on the full space of characters) can be used instead to give an
explicit construction of for $\CL^{BDP}_p(f,\Xi\1)$ for an analytic family
$\Xi$. A similar result (again based on Waldspurger's special value formula,
but using a different method to interpolate the $p$-adic values) was proven
independently by Brako{\aV{c}}evi{\'c} \cite{Brakocevic2011}. A more general
result of the form we've stated, under a different set of hypotheses (and
extending the situation of an imaginary quadratic field $K/\Q$ to a CM
extension of a totally real field $K/F$) is proven by Hsieh \cite{Hsieh2014}. 

\subsection{Our normalizations of $L$-values}
\label{sec:NormalizationsOfLFunctions}

In this section we describe the constants and other factors used in
constructing our $p$-adic $L$-functions. We start with a central critical
$L$-value $L(f,\xi\1,0)$ which is (presumably) a transcendental complex number,
and define the corresponding ``algebraic value'' and then the ``$p$-adic
value'', which is what will actually be interpolated by the $L$-function. There
is no canonical choice for this (at best everything is only defined up to a
$p$-adic unit anyway), so we choose our precise definitions to be what's most
convenient for our formulas. Thus our definitions of $L_\alg(f,\xi\1,0)$ and
$L_p(f,\xi\1,0)$ are not exactly the same as those in other papers, but they 
are of a very similar form. 
\par
So suppose $f$ is a weight $k$ modular form with character $\ch_f$, and let $K
= \Q(\rt{-d})$ be an imaginary quadratic field of odd fundamental discriminant
$-d$. We will consider Hecke characters $\xi$ on $K$ with central character
equal to $\ch_f$, of weight $(m-1,k+1-m)$ (matching the convention we'll be
using later in the paper). Let $c$ denote the norm of the conductor of $\xi$,
which we assume is coprime to $p$. We define the \E{algebraic part} of the
$L$-value $L(f,\xi\1,0)$ as 
\[ L_\alg(f,\xi\1,0) = \f{L(f,\xi\1,0)}{\Om_\oo^{4m-2k-4}} w_K (m-2)! (m-k-1)! 
		\f{c^{2m-k+2} \pi^{2m-k-3}}{\rt{d}^{2m-k-3} 2^{2m-k-3}}. \]
Here $w_K$ is the order of the group of units in $\CO_K$ (so is $6$ if $d = 3$
and $2$ otherwise) and $\Om_\oo$ is a real period which we'll define explicitly
in Section \ref{sec:CongruenceModuleCM}. An important case of this is when
$\et$ is a Hecke character of infinity-type $(k,0)$, so $m = k+1$ and thus 
\[ L_\alg(f,\et\1,0) = \f{L(f,\et\1,0)}{\Om_\oo^{2k}} w_K (k-1)! 
		\f{c^{k+4} \pi^{k-1}}{\rt{d}^{k-1} 2^{k-1}} \in \L\Q. \]
To $p$-adically interpolate these values as $\xi$ varies in an anticyclotomic
family as described before, we embed the algebraic value $L(f,\xi\1,0)$ into
$\L\Q_p$ via our embeddings $\io_\oo$ and $\io_p$, and define a modified
\E{$p$-adic part} as 
\[ L_p(f,\xi\1,0) = e_p(f,\xi\1) \Om_p^{4m-2k-4} L_\alg(f,\xi\1,0) \] 
where $\Om_p$ is a $p$-adic period defined in parallel with $\Om_\oo$ in
Section \ref{sec:CongruenceModuleCM}, and $e_p(f,\xi\1)$ is an Euler factor at
$p$ that modified the $L$-value so it interpolates $p$-adically. The specific
form of $e_p(f,\et\1)$ depends on the local behavior of $f$ at the prime $p$;
for this paper (where $f$ has trivial central character) it will be one of the
following three cases based on how $p$ divides the level $N$ of $f$ and
involving the $p$-th Fourier coefficient $a_p$ of $f$:
\begin{itemize}
	\item If $p\nd N$ then $e_p(f,\xi\1) 
			= (1 - a_p \xi\1(\L\Fp) + \xi^{-2}(\L\Fp) p^{k-1} )^2$.
	\item If $p \| N$ then $e_p(f,\xi\1) = p^{k/2} \xi_m(\L\Fp)\1 
		(1 - a_p \xi\1(\L\Fp))^2$
	\item If $p^r \| N$ with $r \ge 2$ then $e_p(f,\xi\1) = (p^{k/2}
		\xi_m(\L\Fp)\1)^r$.
\end{itemize}
This modified Euler factor is only relevant when $\xi$ is varying in a family
and we want to interpolate it; for fixed values (which will occur for a
character $\et$ of weight $(k,0)$ in our ultimate formula) it's useful to
simply define  
\[ L_p^*(f,\xi\1,0) = \Om_p^{4m-2k-4} L_\alg(f,\xi\1,0). \]
Since $\Om_p$ is generally taken to be a $p$-adic unit, $L_p^*$ has the same
$p$-adic valuation as $L_\alg$.

\subsection{Precise setup for our calculation} 

Our version of Ichino's formula relates $|\g{f(z)g_\ph(z),g_\ps(c^2 N z)}|^2$
to the product of $L(f,\ph\ps\1,0)$ and $L(f,\ps\1\ph\1 N^{m-k-1},0)$. With
this in hand, our strategy for constructing $\CL_p(f,\Xi\1)$ is as follows: let
the characters $\ph,\ps$ vary in $\CI$-adic families $\Ph,\Ps$ such that $\Xi =
\Ps\Ph N^{-m+k+1}$ is the family we're aiming for, while $\Ph\Ps\1$ is the
constant family for some character $\et$. Rearranging, we get a family of
equations relating $L(f,\xi_m\1,0)$ to a product of things we can $p$-adically
interpolate: various explicit constants and terms with mild dependence on $m$
(which can either be $p$-adically interpolated or folded into the definition of
the algebraic part of the $L$-value), the reciprocal of a single $L$-value
$L(f,\et\1,0)$, and finally a product of Petersson inner products. Carrying
this out will give us our $p$-adic $L$-function. 
\par
We first of all need to specify exactly what our input data will be, and our
hypotheses on it. To start off, we fix the following things (where the
hypotheses listed are fundamental to the method we're using): 
\begin{itemize}
	\item A holomorphic newform $f$ of some weight $k$, level $N = N_0 p^{r_0}$,
		and character $\ch_f$.
	\item An imaginary quadratic field $K = \Q(\rt{-d})$ (where $-d$ is the
		fundamental discriminant).
	\item An odd prime $p$ which splits in $K$.
	\item A Hecke character $\Uxi_a$ of weight $(a-1,-a+k+1)$ for some
		integer $a$ satisfying $a-1 \ee k\md 2$, with central character $\ch_\xi
		= \Uxi_{a,\fin}|_\Z$ equal to $\ch_f$. This determines a $\CI$-adic
		family $\Xi = \Uxi_a \al^{-a} (\al^c)^a \CA (\CA^c)\1$ as in Lemma
		\ref{lem:AnticyclotomicHeckeFamily}. 
\end{itemize}
In addition we impose the following hypotheses on these objects. For the most
part these assumptions come from our version of Ichino's formula, Theorem
\ref{thm:ExplicitIchinoCM}. In principle one could eliminate most or all
of these by working out more cases of the local integrals arising in Ichino's
formula - so these hypotheses should not be fundamental roadblocks to carrying
out the arguments of this paper. 
\begin{itemize}
	\item The character $\ch_f$ of the newform $f$ is assumed to be trivial.
		Note that this forces the weight $k$ and thus the integer $a$ to be
		even. 
	\item The fundamental discriminant $-d$ is odd (i.e. $d$ is squarefree and 
		satisfies $d\ee 3\md 4$), and $d$ is assumed coprime to $N$.
	\item The Hecke character $\Uxi_a$ has conductor $(c)$ for $c$ some
		$c\in\Z$ coprime to $p$, $d$, and $N$.
\end{itemize}
Given this setup, we want to define families of characters $\Ph$ and $\Ps$
which give us $\CI$-adic CM forms $\Dg_\Ph$ and $\Dg_\Ps$ so that we can carry
out the computation outlined in the previous section. More precisely, we want
to construct finite-order characters $\ph,\ps$ and take the associated families
$\Ph = \ph \al\1 \CA$ and $\Ps = \ps \al\1 \CA$ with specializations $\ph_{m-k}
= \ph \al^{m-k-1}$ and $\ps_m = \ps \al^{m-1}$. If we look at the right-hand
side of Ichino's formula, the $L$-values involved are
\[ L(f,\ph\ps\1 \al^{-k},0) L(f,\ps\1\ph\1 \al^{-(2m-k-2)} N^{m-k-1},0). \]
So we want to define $\ps$ and $\ph$ such that the product $\ps\1\ph\1
\al^{2m-2} N^{m-k-1}$ equals $\xi_m\1$ for $m \ee a\md{p-1}$, and such that
$\ph\ps\1 \al^{-k}$ is a ``well-behaved'' auxiliary character. 
\par
For applications, it will be important to be able to have flexibility in what
we're allowed to choose as $\ph\ps\1$, and also the Hida theory we want to
apply demands that our CM newform $\Dg_\Ps$ has residually irreducible Galois
representations. We can deal with both of these issues by introducing a
character $\nu$ mod $\l$ for any prime $\l$; if we modify $\ps$ by $\nu$
and $\ph$ by $\nu\1$, then the product $\ps\ph$ (which corresponds to the
family $\Xi$ we've already fixed) doesn't change while $\ps\ph\1$ is modified
by $\nu^2$. Accordingly, we add in the following auxiliary data to our setup:
\begin{itemize}
	\item An auxiliary prime $\l$ not dividing $2pcdN$ that is inert in $K$, and
		an auxiliary character $\nu$ of $(\CO_K/\l^{c_\l}\CO_K)^\x$ of conductor
		$\l^{c_\l}$ and which is trivial on $(\Z/\l^{c_\l}\Z)^\x$. 
\end{itemize}
Also, since the field $K = \Q(\rt{-3})$ has some extra roots of unity, in that case we need to add some additional hypotheses. 
\begin{itemize}
	\item If $K = \Q(\rt{-3})$ we also require $\nu$ is trivial on $\CO_K^\x$
		(this is automatic in all other cases), and that $a-1 \ee 0 \md 6$.
\end{itemize}
We can then explicitly construct a Hecke character satisfying the properties we
want. 

\begin{lem} \label{lem:ConstructPhiPsi}
	Suppose we have the data and hypotheses specified above. Then we can define
	finite-order Hecke characters $\ph$ and $\ps$ of $K$ (giving rise to
	families $\Ph,\Ps$ and $\La$-adic CM forms $\Dg_\Ph,\Dg_\Ps$) with the
	following properties
	\begin{enumerate}
		\item The conductor of $\ph$ is either $c \l^{c_\l}$ or $\Fp c
			\l^{c_\l}$, and for $m \ee a-1\md{p-1}$ the character $\ph_{m-k}$ has
			conductor $\l^{c_\l}$ and finite-type $\xi_{a,fin} \nu\1$.
		\item The conductor of $\ps$ is either $\l^{c_\l}$ or $\Fp \l^{c_\l}$, 
			and for $m \ee a-1\md{p-1}$ the character $\ps_m$ has conductor 
			$\l^{c_\l}$ and finite-type $\nu$.
		\item The product $\ph_{m-k}\1 \ps_m = \ph\1 \ps \al^k$ is independent
			of $m$ and has conductor $c\l^{c_\l}$ and finite-type $\xi_{a,fin}\1
			\nu^2$. 
		\item For $m\ee a\md{p-1}$, the product 
			\[ \ph_{m-k} \ps_m N^{-(m-k-1)}
				= \ph \ps \al^{2m-k-2} N^{-(m-k-1)} \] 
			is equal to $\xi_m$.
	\end{enumerate}
\end{lem}
\begin{proof}
	We want to define $\ph$ and $\ps$ to have finite-type parts
	\[ \ph_\fin(x) = \om_\Fp(x)^{a-k-1} \xi_{a,\fin}(x) \nu\1(x) \qq\qq 
			\ps_\fin(x) = \om_\Fp(x)^{a-1} \nu(x). \]
	For these to define finite-order Hecke characters we just need to check that
	the right-hand sides of these characters are trivial for $x$ in the unit
	group $\CO_K^\x$. This follows from the following the following three facts: 
	\begin{itemize}
		\item We have $\nu(x) = 1$ because we've specified $\nu$ is trivial on
			units (as an explicit hypothesis for $\Q(\rt{-3})$, and as a
			consequence of it being trivial on $\Z$ in all other cases). 
		\item We have $a-1 \ee 0\md{|\CO_K|^\x}$ and thus $\om_\Fp(x)^{a-1}
			= 1$ (as an explicit hypothesis for $\Q(\rt{-3})$, and as a
			consequence of $a-1 \ee k\ee 0\md 2$ in the other cases). Thus 
			$\ps_\fin(x) = 1$ for all $x \in \CO_K^\x$. 
		\item We have $\xi_{a,\fin}(x) = x^{-(2a-k-2)}$ (because $\xi_a$ is 
			a Hecke character of weight $(a-1,-a+k+1)$) and $\om_\Fp(x)^{a-1} =
			x^{a-1}$ (by construction of the Teichm\"uller character and because
			$x$ must be a root of unity of order prime to $p$), so $\ph_\fin(x) =
			x^{-(a-1)} = 1$ because $a-1\ee 0\md{|\CO_K|^\x}$. 
	\end{itemize}
	So $\ph_\fin,\ps_\fin$ give us characters on $K^\x/\CO_K^\x$; embedding this
	into the group of ideals we can extend these (in $|\Cl_K|$ different ways)
	to finite-order Hecke characters $\ph,\ps$. In fact, to make (4) true we 
	extend $\ph$ arbitrarily and then define $\ps$ by $\ps = \xi_a \ph\1
	\al^{-2a+k+2} N^{m-k-1}$ (which we can check has the correct finite-type and
	infinity-type). 
	\par
	So we've constructed $\ps$ and $\ph$ in such a way that (1) and (2) are
	true, (3) is immediate, and (4) is true at least for $m = a$. To prove (4)
	for general $m\ee a\md{p-1}$ we remember $\xi_m$ is defined as $\xi_a
	\al^{m-a} (\al^c)^{a-m}$ and use the fact that $\al \al^c$ is a twist of $N$
	by a Teichm\"uller character $\om\1\o N_\Q^K$ and thus $(\al\al^c)^{p-1} =
	N^{p-1}$ (see Section \ref{sec:ComplexConjugatesHeckeChars}). 
\end{proof}

So, for the rest of this section we fix the characters $\ps,\ph$ and the
associated families $\Ps,\Ph$ as constructed in this lemma. We also take the
notation that $\et = \ph\1 \ps \al^k$ is the constant Hecke character (of
infinity-type $(k,0)$). Then, for every positive weight $m\ee a\md{p-1}$,
Theorem \ref{thm:ExplicitIchinoCM} tells us we have an identity
\begin{multline*}
	\g{fg_{\ph,m-k},g_{\ps,m}|_{c^2 N}} 
		\g{f^\rh g^\rh_{\ph,m-k},g^\rh_{\ps,m}|_{c^2 N}} \\
			= \f{3^2 (m-2)! (k-1)! (m-k-1)!}
						{\pi^{2m+2} 2^{4m-2} d \l^{2c_\l} (c^2 N)^{m+1}} 
			\dt \p_{q|c N d\l} (1+q\1)^{-2} \dt (*)_{f,\l} 
				 \dt L(f,\et\1,0) L(f,\xi_m\1,0) 
\end{multline*}
where we set 
\[ (*)_{f,\l} = \b( \s_{i=0}^{c_\l} \pf{\al}{\l^{(k-1)/2}}^{2i-c_\l} 
			- \f{1}{\l} \s_{i=i}^{c_\l-1} \pf{\al}{\l^{(k-1)/2}}^{2i-c_\l} \e). \]
with $\al_{f,\l}$ one of the roots of the Hecke polynomial for $f$ at $\l$.

\subsection{The equation for algebraic and $p$-adic values}

To use the above family of equations to produce a $p$-adic $L$-function, we
first need to manipulate it to give equations relating the ``algebraic parts''
of Petersson inner products (as defined in Section
\ref{sec:SummarizingPeterssonInterpolation}) to those of $L$-values (as defined
in section \ref{sec:NormalizationsOfLFunctions}), and then equations of the
corresponding ``$p$-adic parts''. 
\par
Following the notation of Section \ref{sec:SummarizingPeterssonInterpolation},
we suppress the shift by $c^2 N$ and write $\g{fg_{\ph,m-k},g_{\ps,m}}$
in place of $\g{fg_{\ph,m-k},g_{\ps,m}|_{c^2 N}}$. Then using Proposition
\ref{prop:PeterssonAlgebraicPAdicParts} we write 
\[ \g{fg_{\ph,m-k},g_{\ps,m}}_\alg \g{f^\rh g_{\ph,m-k}^\rh,g_{\ps,m}^\rh}_\alg
		= C_1 \f{\g{fg_{\ph,m-k},g_{\ps,m}}}{\Om_\oo^{2m-2}} 
		\f{\g{f^\rh g_{\ph,m-k}^\rh,g_{\ps,m}^\rh}}{\Om_\oo^{2m-2}} \] 
for
\[ C_1 = \f{2^{2m} w_K^2 \pi^{4m-2}}{3^2 d^{m-3}} 
	c^{4m+8} N_0^{2m+4} \f{p^{2m(r-1)}}{\ps_m(\L\Fp)^{2(r-1)}} 
					\p_{q | cN_0 d\l} (1+q\1)^2, \]
noting that the original formula gives us powers of $\al_{g_{\ps,m}} =
\ps_m(\L\Fp)$ and $\al_{g_{\ps,m}^\rh} = \Lbe_{g_{\ps,m}} = \Lps_m(\Fp) =
\ps_m(\L\Fp)$ in the denominator. Similarly, using
our definitions in Section \ref{sec:NormalizationsOfLFunctions} we have 
\[ L_\alg(f,\xi_m\1,0) L_\alg(f,\et\1,0) 
	= C_2 \f{L(f,\xi_m\1,0)}{\Om_\oo^{4m-4-2k}} \f{L(f,\et\1,0)}{\Om_\oo^{2k}}\]
where 
\[ C_2 = w_K^2 (m-2)! (m-k-1)! (k-1)! 
		\f{c^{2m+6} \l^{2c_\l(k+4)} \pi^{2m-4}}{d^{m-2} 2^{2m-4}}. \]
Taking our explicit Ichino's formula and manipulating it to get the constants
$C_1$ and $C_2$ on each side, we obtain 
\[ \g{fg_{\ph,m-k},g_{\ps,m}}_\alg \g{f^\rh g_{\ph,m-k}^\rh,g_{\ps,m}^\rh}_\alg
		= C_3 \dt L_\alg(f,\xi_m\1,0) L_\alg(f,\et\1,0) \]
for 
\[ C_3 = \f{2^2 N_0^{m+3}}{\l^{2c_\l(k+5)}} (*)_{f,\l} 
		\dt \f{p^{(m-1)(r-1)}}{\ps_m(\L\Fp)^{2(r-1)}} 
				\7( (1+p\1)^{2} p^{(m+1)} \7)^{r-r_0-1}; \]
note that $1+r_0-r$ is $0$ if $p$ does not divide the level of $f$ and $-1$ if
$p$ does divide the level of $f$, so the parenthesized term on the right
is $1$ in the former case and is $(1+p\1)^{-2} p^{-(m+1)}$ in the latter case.
\par
Next we pass to an equation of $p$-adic values (except for $L_\alg(f,\et\1,0)$
which we will be treating as a constant). Again, going back to Proposition
\ref{prop:PeterssonAlgebraicPAdicParts} and Section
\ref{sec:NormalizationsOfLFunctions} we see that there are removed Euler
factors on each side - $e_p(fg_{\ph,m-k},g_{\ps,m}) e_p(f^\rh
g_{\ph,m-k}^\rh,g_{\ps,m}^\rh)$ on the left and $e_p(f,\xi\1)$ on the right.
\par
To compare these, we start by considering the case $r_0 = 0$. Here we have 
\[ e_p(fg_{\ph,m-k},g_{\ps,m}) 
		= \b(1-\f{\ph_{m-k}(\Fp)\al_f}{\ps_m(\L\Fp)}\e)
				\b(1-\f{\ph_{m-k}(\Fp)\be_f}{\ps_m(\L\Fp)}\e) \]
from our computations of $\al$ and $\be$ for a CM form in Section
\ref{sec:FamiliesOfCMForms}. We can also check that $\xi_m(\L\Fp) =
\ph_{m-k}(\L\Fp) \ps_m(\L\Fp) N(\Fp)^{-(m-k-1)} = \ps_m(\L\Fp)/\ph_{m-k}(\Fp)$,
and therefore conclude 
\[ e_p(fg_{\ph,m-k},g_{\ps,m}) 
	= (1-\xi_m\1(\L\Fp)\al_f)(1-\xi_m\1(\L\Fp)\be_f)
		= (1 - a_p \xi\1(\L\Fp) + \xi^{-2}(\L\Fp) p^{k-1} ); \]
this matches up with (the square root of) $e_p(f,\xi\1)$. 
\par
We can similarly analyze $e_p(f^\rh g_{\ph,m-k}^\rh,g_{\ps,m}^\rh)$, recalling
from Section \ref{sec:ComplexConjugatesHeckeChars} that $\al_{h^\rh} = \Lbe_h$
and $\be_{h^\rh} = \Lal_h$; we get 
\[ e_p(f^\rh g_{\ph,m-k}^\rh,g_{\ps,m}^\rh)
 	= \b( 1-\f{\Lal_g\Lbe_f}{\Lbe_h} \e) \b( 1-\f{\Lal_g\Lal_f}{\Lbe_h} \e). \]
By comparing the equalities $\al_f\be_f = \ch_f(p)p^{k-1}$ (by definition) and
$\al_f\Lal_f = p^{k-1}$ (by the Ramanujan conjecture) we know $\Lal_f =
\Lch_f(p)\be_f$, and we get similar identities for all of the other conjugates;
and since $\ch_f\ch_g = \ch_h$ the character values we have cancel and we
conclude 
\[ e_p(f^\rh g_{\ph,m-k}^\rh,g_{\ps,m}^\rh)
 		= \b( 1-\f{\be_g\al_f}{\al_h} \e) \b( 1-\f{\be_g\be_f}{\al_h} \e). \]
So for $r_0 = 0$ we get an equality 
\[ e_p(fg_{\ph,m-k},g_{\ps,m}) e_p(f^\rh g_{\ph,m-k}^\rh,g_{\ps,m}^\rh)
		= e_p(f,\xi\1). \]
\par
For $r_0 \ge 2$, we directly calculate 
\[ e_p(fg_{\ph,m-k},g_{\ps,m}) e_p(f^\rh g_{\ph,m-k}^\rh,g_{\ps,m}^\rh)
		= p^2 (1+p\1)^2 \ps_m(\Fp)^2, \]
and for $r_0 = 1$ similar arguments give 
\[ e_p(fg_{\ph,m-k},g_{\ps,m}) e_p(f^\rh g_{\ph,m-k}^\rh,g_{\ps,m}^\rh)
		= p^2 (1+p\1)^2 \ps_m(\Fp)^2 (1-\xi_m\1(\L\Fp)a_f)^2. \]
\par
On the other hand, we've defined $e_p(f,\xi\1)$ on a case-by-case basis to
consist of exactly the Euler factors that appear here, times an additional
factor of $(p^{k/2} \xi_m(\L\Fp)\1)^r$. Further manipulations of our formula 
result in

\begin{prop} \label{prop:PAdicPartIchino}
	We have
	\[ \g{fg_{\ph,m-k},g_{\ps,m}}_p \g{f^\rh g_{\ph,m-k}^\rh,g_{\ps,m}^\rh}_p
			= C_4 \dt L_p^*(f,\et\1,0)  L_p(f,\xi_m\1,0) \]
	for 
	\[ C_4 = \et(\L\Fp)^{-r} (*)_{f,\l} \f{2^2 N_0^{m+3}}{\l^{2c_\l(k+5)}} , \]
	where we recall $L_p^*(f,\et\1,0)$ is the value $\Om_p^{2k}
	L_\alg(f,\et\1,0)$ without an Euler factor removed. 
\end{prop}

\subsection{The $p$-adic $L$-function} 

Finally, we can use our final equation of the previous section to finish our
construction of $\CL(f,\Xi\1)$ in terms of the element
\[ \g{f\Bg_\Ph,\Bg_\Ps} 
		= \l_{\Bg_\Ps}(\trc(f|_{1/c^2 N} \dt \Bg_\Ph|_{1/c^2 N})) \in \CI^\uR \]
and the conjugate element $\g{f^\rh\Bg_{\Ph^\rh},\Bg_{\Ps^\rh}} \in \CI^\uR$,
as constructed in Chapters \ref{sec:FamiliesOfCharsCMForms} and
\ref{sec:PeterssonInterpolate}. 
\par
In our constant $C_4$, there's only one term that depends on $m$, the parameter
that varies in our $p$-adic family, the factor of $N_0^{m+3}$. This is
straightforward to interpolate explicitly (see e.g. Section 7.1 of
\cite{HidaEisSeries}), so we can find an element $\CC \in \La^\x$ satisfying 
\[ P_m(\CC) = \et(\L\Fp)^r \f{2^2 \l^{2c_\l(k+5)}}{N_0^{m+3}}. \]
We can then conclude our main theorem: 

\begin{thm}
	We have the following identity in $\CI^\uR$:
	\[ \g{f\Bg_\Ph,\Bg_\Ps} \g{f^\rh\Bg_{\Ph^\rh},\Bg_{\Ps^\rh}} 
		= (*)_{f,\l} \dt L_p^*(f,\et\1,0) \dt \CC\1 \dt 
				\CL(f,\Xi\1). \]
\end{thm}
\begin{proof}
	For the infinite set of $m > k$ satisfying $m\ee a\md{m-1}$, the
	specializations of both sides of the equation are equal by Proposition
	\ref{prop:PAdicPartIchino}; thus the elements of $\CI^\uR$ themselves are
	equal by Lemma \ref{lem:EqualityOfSpecializations}. 
\end{proof}

We can then rearrange to get an explicit expression for $\CL(f,\Xi\1)$ (as 
stated in Theorem \ref{thm:MainTheorem}), assuming the constants don't make the
equation identically zero. If the constants are in fact units mod $p$, then our
rearrangement genuinely gives an expression occurring in $\CI^\uR$. 
 
%\input{AnticyclotomicIchino_ch8.tex} 

%% TODO still: 
% • The author gives very detailed reviews on results that he needs from the
% * S2: As I said earlier, you could perhaps mention works of Brakocevic and

\bibliographystyle{amsalpha}
\addcontentsline{toc}{section}{References}
\bibliography{AnticyclotomicIchino_bib}

\end{document}